\documentclass[aap,preprint]{imsart}

\RequirePackage{amsthm,amsmath,amsfonts,amssymb}
\RequirePackage[numbers]{natbib}
\RequirePackage[colorlinks,citecolor=blue,urlcolor=blue]{hyperref}
\RequirePackage{graphicx}

\usepackage[T1]{fontenc}
\usepackage[utf8]{inputenc}
\usepackage{lmodern}
\usepackage{hyperref}
\usepackage{bussproofs}
\usepackage{pdfpages}
\usepackage{float}
\usepackage{svg}
\usepackage{array}
\usepackage{multirow}
\usepackage{bibentry}
\usepackage{comment}
\usepackage{faktor}
\usepackage[mathlines]{lineno}
\usepackage{subcaption}
\usepackage{algorithm}
\usepackage{algpseudocode}
\usepackage{float}
\usepackage{enumitem}
\usepackage{epigraph}
\usepackage{etoc}
\usepackage{tikz-cd}
\usepackage{etoolbox}
\usepackage{bookmark}
\usepackage{ragged2e}
\usepackage{CJKutf8}
\usepackage{upgreek}
\usepackage[multiple]{footmisc}
\usepackage{wrapfig}
\usepackage{pgfplots}
\usepackage{tikz}
\usepackage{tikzscale}
\usetikzlibrary{automata,positioning,arrows,calc,decorations.shapes,decorations.markings,quotes,shapes.misc}

 \usetikzlibrary{arrows.meta}
 \usetikzlibrary{backgrounds}
 \usepgfplotslibrary{patchplots}
 \usepgfplotslibrary{fillbetween}
 \pgfplotsset{
     layers/standard/.define layer set={%
         background,axis background,axis grid,axis ticks,axis lines,axis tick labels,pre main,main,axis descriptions,axis foreground%
     }{
         grid style={/pgfplots/on layer=axis grid},%
         tick style={/pgfplots/on layer=axis ticks},%
         axis line style={/pgfplots/on layer=axis lines},%
         label style={/pgfplots/on layer=axis descriptions},%
         legend style={/pgfplots/on layer=axis descriptions},%
         title style={/pgfplots/on layer=axis descriptions},%
         colorbar style={/pgfplots/on layer=axis descriptions},%
         ticklabel style={/pgfplots/on layer=axis tick labels},%
         axis background@ style={/pgfplots/on layer=axis background},%
         3d box foreground style={/pgfplots/on layer=axis foreground},%
     },
 }

\algdef{SE}[SUBALG]{Indent}{EndIndent}{}{\algorithmicend\ }%
\algtext*{Indent}
\algtext*{EndIndent}

\startlocaldefs
\numberwithin{equation}{section}

\theoremstyle{plain}
\newtheorem{thm}{Theorem}[section]
\newtheorem{prop}[thm]{Proposition}

\newtheorem{lem}[thm]{Lemma}
\newtheorem{cor}[thm]{Corollary}

\theoremstyle{definition}
\newtheorem{definition}{Definition}[thm]
\newtheorem{asmp}[thm]{Assumption}

\newtheorem{rem}[thm]{Remark}

\endlocaldefs

\newcommand\R{\mathbf{R}}

\newcommand\Z{\mathbf{Z}}
\newcommand\N{\mathbf{N}}
\newcommand\C{\mathbf{C}}

\newcommand\prb{\mathbf{P}}

\newcommand\eps{\varepsilon}

\newcommand\val{\mathtt{v}}

\newcommand\supp{\mathrm{supp}}

\newcommand{\argmin}{\operatornamewithlimits{argmin}}

\begin{document}

\begin{frontmatter}
\title{Statistical Consistency of Discrete-to-Continuous Limits of Determinantal Point Processes}
\runtitle{Discrete-to-Continuous Limits of Determinantal Point Processes}

\begin{aug}
\author[A]{\fnms{Hugo}~\snm{Jaquard}\ead[label=e1]{hugo.jaquard@irisa.fr}}
\and
\author[A]{\fnms{Nicolas}~\snm{Keriven}\ead[label=e2]{nicolas.keriven@cnrs.fr}}
\address[A]{CNRS, IRISA\printead[presep={,\ }]{e1,e2}}
\end{aug}

\begin{abstract}

    We investigate the limiting behavior of discrete determinantal point processes (DPPs) towards continuous DPPs when the size of the set to sample from goes to infinity. We propose a non-asymptotic characterization of this limit in terms of the concentration of statistics associated to these processes, which we refer to as ``weak coherency''. This allows to translate statistical guarantees from the limiting process to the original, discrete one. Our main result describes sufficient conditions for weak coherency to hold. In particular, our study encompasses settings where both the kernel of the continuous process and its underlying space are inaccessible, or when the discrete marginal kernel is a noisy version of its continuous counterpart.

    We illustrate our theory on several examples. We prove that a discrete multivariate orthogonal polynomial ensemble can be used to produce coresets \emph{strictly} smaller than independent sampling for the same error. We propose a process achieving repulsive sampling on an \emph{unknown} manifold from a set of points sampled from an \emph{unknown} density. Finally, we show that continuous DPPs can be obtained as limits on \emph{random} graphs with Bernoulli edges, even when only observing the graph structure. We obtain interesting byproduct results along the way.

\end{abstract}



\end{frontmatter}


\section{Introduction}
\label{sect:intro}

Determinantal point processes (DPPs) are a class of probability distributions over ---possibly infinite--- random point sets of a space $\Gamma$ that can exhibit repulsive interactions between points; loosely speaking, this means that two points that are close together are less likely to be included in a sample than two points far away from each other. Determinantal processes occupy a special place in the current landscape of repulsive processes due to being uniquely theoretically and numerically tractable among such models. Originally singled out in mathematical physics by O.~Macchi as a tool to describe the anti-bunching phenomenon observed in electron detection-times (see~\cite{bardenet:hal-03837697} for a historical account), they have since been the object of sustained research in probability theory~\cite{soshnikov2000determinantal,shirai2003random}, machine learning~\cite{kulesza_determinantal_2012}, spatial statistics~\cite{lavancier2015determinantal}, and even numerics and signal processing~\cite{bardenet2020monte}; this list is of course far from exhaustive.

It is typical to distinguish between two types of DPPs. 
On the one hand, DPPs over a \emph{finite} set $\Gamma$ will be referred to as \emph{discrete} DPPs. For instance, in computational statistics, discrete DPPs are often employed to subsample a data set $\Gamma = X_n$ of $n$ data points, when it is either too large for a specific processing pipeline, or if it is not possible to (hand-)label it entirely. The inter-point repulsion here serves as a way to obtain subsamples that better cover the data set when compared to an independent sampling of the same size. Discrete DPPs are straightforward to define and sample~\cite{hough_determinantal_2006,anari2022optimal,barthelme2023faster}. At the other end of the spectrum, when $\Gamma = \mathcal{X}$ is a ``continuous'' space, for instance $\R^d$ or a manifold, we will speak of \emph{continuous} DPPs. Continuous DPPs can be studied using refined analytical tools, so that strong desirable statistical properties may be established (\emph{e.g.},~\cite{bardenet2020monte,bardenet2021determinantal,levi2024linear}), but are generally more cumbersome to manipulate than their discrete counterpart.

Though this distinction is more practical than mathematical, for the practitioner, there is generally a world between sampling a discrete and a continuous DPP. Can we then exploit the statistical guarantees of continuous DPPs when sampling discrete DPPs? More precisely, we study the following, natural question.
\begin{quote}
    \normalsize 
    \emph{Assuming that the data $X_n$ has been drawn from an underlying space $\mathcal{X}$ (in particular, for independently and identically distributed ---iid--- data), can a discrete DPP defined on $X_n$ ``nicely approximate'' a continuous DPP on $\mathcal{X}$ as the size $n \rightarrow \infty$?} 
\end{quote}
In turn, can this approximation be used to translate statistical guarantees from continuous DPPs to discrete DPPs? If so, it is likely that many existing discrete DPPs already used in the literature could be revisited from a ``limit'' point of view to obtain novel guarantees. In parallel, \emph{new} discrete DPPs may also be specifically constructed to approximate continuous DPPs with good properties.

To better motivate our approach, let us now describe a concrete example which is, to our knowledge, one of the only previous works where a special case of this question has been considered. It is revisited in section \ref{sect:ope_main}.
In this example, we consider a standard setup of statistical learning, which consists in finding a function $f_\theta$ with $\theta \in \Theta$ that minimizes a loss function $L: \Theta \rightarrow \R_{>0}$ that depends on some available data $X_n=\{x_1,\ldots,x_n\}$. For many classical learning tasks such as empirical risk minimization,
the loss function takes the form
    \begin{equation*}
        L(\theta) = \sum_{i=1}^n f_\theta(x_i)
    \end{equation*}
    which, when $n$ is large, may be expensive to optimize.\footnote{For instance, in k-means, support vector machines or low-rank approximations.}  In order to reduce the computational load, one may then wish to build a \emph{coreset}, \emph{i.e.} a subset $\mathcal{S} \subseteq X_n$ of size $m \ll n$ together with weights $w_x$ such that  
    \begin{equation}
        L_\mathcal{S}(\theta) = \sum_{x \in \mathcal{S}} w_x f_\theta(x)  
        \label{eq:coreset_statistics}
    \end{equation} 
    is close to $L(\theta)$, so that the training need only be performed on this much smaller coreset~\cite{munteanu2018coresets}. A discrete DPP can be used to sample $\mathcal{S}$, which is provably no-worse and experimentally advantageous with respect to independent sampling~\cite{tremblay2019determinantal}. Mathematical guarantees of a \emph{strict} advantage on the other hand have only been established for one specific example~\cite{bardenet2021determinantal,bardenet2024smallcoresetsnegativedependence}: when the $x_i$ are drawn iid from $\mathcal{X} = [-1,1]^d$, for a target approximation quality of $L$, using a carefully defined DPP on $X_n$ allows to obtain coresets $\mathcal{S}$ than are smaller than independent-sampling methods when $n$ is large enough. The proof of~\cite{bardenet2021determinantal} reveals that a specific quantity associated to this process approaches that associated to a \emph{continous DPP} on $\mathcal{X}$ with remarkable analytical properties, which is then used to obtain the previous guarantee. Note that this argument is not explicitly framed as a ``discrete-to-continuous'' limit in~\cite{bardenet2021determinantal}, while this is the main purpose of the present work. Based on our general framework, we show in section \ref{sect:ope_main} that an alternative, much simpler discrete DPP satisfies the same type of guarantee.

\subsection{Summary of our contributions and outline}

Let us now informally sketch our main contribution. 
Loosely speaking, a DPP $\mathcal{P} = \textup{DPP}(\mathcal{K}, \mu)$ on $\Gamma$ is defined by a \emph{kernel} $\mathcal{K} : \Gamma \times \Gamma \rightarrow \C$ and a \emph{reference measure} $\mu$ on $\Gamma$; the precise definitions are deferred to section~\ref{sect:background_DPP} and we remark that, when the set $\Gamma$ is finite, the kernel can be represented by a \emph{matrix}\footnote{Strictly speaking, this depends on an ordering of the elements in $\Gamma = \{x_1, \ldots, x_n\}$ such that $K_{ij}$ is the kernel value at $(x_i,x_j)$. Such an ordering is implicit and never ambiguous in this paper.} $K \in \C^{|\Gamma| \times |\Gamma|}$.

Going forward, and with the exception of some more general, technical results, we consider that $\mathcal{X}$ is a compact space and $X_n = \{x_1,...,x_n\} \subset \mathcal{X}$ has been sampled iid according to a probability measure $\mu$. Denoting by $\mu_n = \frac{1}{n} \sum_{i} \delta_{x_i}$ the empirical measure, and by $K_n \in \C^{n \times n}$ some matrix (observed or constructed by the practitioner), we are interested in the ``convergence'' of $\mathcal{P}_n = \textup{DPP}(K_n, \mu_n)$ towards some $\mathcal{P} = \textup{DPP}(\mathcal{K}, \mu)$. We characterize this convergence by a novel notion that we call \emph{weak coherency}, which consists in guaranteeing the concentration of the expectations of all \emph{linear statistics} of samples of the processes. For a bounded and measurable function $\varphi_r:\mathcal{X}^r \to \R$ and a finite sample $\mathcal{S} \subseteq \mathcal{X}$, the ($r$-points) linear statistic of $\varphi_r$ with respect to $\mathcal{S}$, and its expectation over a given point process $\mathcal{Q}$, are defined by
\begin{equation*}
    \Lambda^{(\varphi_r)}(\mathcal{S}) := \sum_{\substack{{x_1,...,x_r \in \mathcal{S}} \\ {x_i \neq x_j}}} \varphi_r(x_1,...,.x_r), \quad \Phi^{(\varphi_r)}(\mathcal{Q}) := \mathbf{E}_{\mathcal{S} \sim \mathcal{Q}} [\Lambda^{(\varphi_r)}(\mathcal{Q})]
\end{equation*}
That is, $\Phi^{(\varphi_r)}(\mathcal{Q})$ is the expectation of the sum of evaluations of $\varphi_r$ over all $r$-element subsets of $\mathcal{S}$. Loosely speaking, we say that $\mathcal{P}_n$ is weakly coherent with $\mathcal{P}$ when, for all $r \in \N$, $\varphi_r:\mathcal{X}^r \to \R$ and $\delta \in (0,1)$ and $\eps>0$, there exists is $N(\delta, \eps, \varphi_r)$ such that, for all $n \geq N(\delta,\eps,\varphi_r)$,
\begin{align*}
    \mathbb{P}\left(\left\vert \Phi^{(\varphi_r)}(\mathcal{P}_n)  -  \Phi^{(\varphi_r)}(\mathcal{P}) \right\vert \geq \eps\right) \leq \delta,
\end{align*}
where the probability $\mathbb{P}$, here and in the rest of the paper, refers to the randomness \emph{over $K_n, \mu_n$}; that is, on the randomly drawn $X_n$ and, sometimes, on other random quantities defining $K_n$. We stress that this does \emph{not} concern randomness with respect to \emph{samples} from $\mathcal{P}_n$ or $\mathcal{P}$. See section~\ref{sect:weak_consistency_general_section} for a detailed definition. 
Typically, these linear statistics $\Lambda^{(\varphi_r)}(\mathcal{S})$ and the associated moments are precisely what one is interested in computing when manipulating DPPs: they appear for instance in the coreset-estimate from equation~\eqref{eq:coreset_statistics} above, or in the quadrature rule for Monte-Carlo methods~\cite{mcbook}, and their variance controls their concentration towards their mean~\cite{bardenet2024smallcoresetsnegativedependence}. As such, weak coherency naturally allows to relate the statistical properties of discrete DPPs with those of their continuous limit, which we will illustrate on several examples.

The crucial element here is the matrix $K_n$, as it solely defines the discrete DPP $\mathcal{P}_n = \textup{DPP}(K_n, \mu_n)$. A primary example is the so-called \emph{Gram matrix} $K_n = \mathcal{K}_{\vert X_n \times X_n} := [\mathcal{K}(x_i,x_j)]_{i,j=1}^n$, which is the easiest setting for which we show that weak coherency holds. However, in many cases, the kernel $\mathcal{K}$ (and sometimes even the points $x_i$ themselves\footnote{In this discrete setting, a sample from a point process over $X_n$ can for many purposes be substituted with a point process over \emph{indices} $[n] = \{1,...,n\}$, in which case the $x_i$ themselves are not needed, only the matrix $K_n$. As this is very dependent on the use case, we do not elaborate further on these technicalities here.}) is not accessible, and one observes, or \emph{estimates}, a noisy version $K_n$ instead. Our main result in this paper provides sufficient conditions for weak coherency to hold in this more difficult situation. Informally, we show the following.

\begin{thm}[Weak coherency conditions, informal]
    \label{th:intro}
    Let $\mathcal{K}$ be a continuous kernel, and $K_n$ a collection of random $n \times n$ matrices. Assume that, for all $\delta \in (0,1)$ and $\eps >0$, there exists $N_K(\delta,\eps)$ such that either:
\begin{enumerate}[label=\roman*)]
    \item $\mathcal{K}$ and $K_n$ are complex-valued, and
    \begin{equation}\label{eq:main_thm_cond1}
        n \geq N_K(\delta,\eps) ~\Rightarrow~ \mathbb{P}\left(\max_{i,j \in [n]} \left\vert [K_n]_{i,j} - \mathcal{K}(x_i,x_j) \right\vert \geq \eps\right) \leq \delta;
    \end{equation}
    \item $\mathcal{K}$ and $K_n$ are real-valued, and
    \begin{equation}\label{eq:main_thm_cond2}
        n \geq N_K(\delta,\eps) ~\Rightarrow~ \mathbb{P}\left(\max\left(\left\Vert \tfrac{K_n - \mathcal{K}_{\vert X_n \times X_n}}{n} \right\Vert_F ,
            \left\vert \mathrm{tr}\left( \tfrac{K_n}{n} \right) - \mathrm{tr}\left( \tfrac{\mathcal{K}_{\vert X_n \times X_n}}{n} \right)\right\vert \right)\geq \eps \right)\leq \delta,
    \end{equation}
\end{enumerate}
    where we recall that $\mathbb{P}$ refers to randomness over $(K_n, \mu_n)$. Then the sequence $\mathcal{P}_n = \textup{DPP}(K_n, \mu_n)$ is weakly coherent with $\mathcal{P} = \textup{DPP}(\mathcal{K}, \mu)$.
\end{thm}

We state a more detailed version of this theorem in section~\ref{sect:awc_compact}. Putting aside the complex/real-valued difference, we remark that condition \emph{i)} is strictly stronger than condition \emph{ii)} in the theorem above. However, we still elected to keep both in our formulation of the theorem, as the two cases additionally yield slightly different weak-coherency rates.
As we shall see, the rates $N(\delta,\eps,\varphi_r)$ for the concentration of the expectations of linear statistics $\Lambda^{(\varphi_r)}(\mathcal{S})$ (and, as we will see, of their moments) can be directly related to the concentration rate $N_K(\delta,\eps)$ of the kernels, and this allows us to characterize properties of discrete DPPs in a non-asymptotic manner, which we will do in three examples, as outlined below.

\subsubsection{Outline} The paper is roughly divided into two parts: the first one pertains to the general theory of weak coherency; the second, to concrete examples. It is organized as follows.
\begin{itemize}
    \item In section~\ref{sect:background_DPP}, we review the necessary background on determinantal point processes, with an emphasis on their definition using correlation functions and linear statistics.
    \item In section~\ref{sect:weak_consistency_general_section}, we introduce 
    weak coherency and show that it entails the concentration of \emph{all the moments} of linear statistics. This framework is very general and would extend beyond the cases of interest in this paper, which is interesting for future developments. In particular, it allows to consider weak coherency between any \emph{not necessarily determinantal} point processes, and it is also not restricted to discrete-to-continuous limits.
    \item In section~\ref{sect:awc_compact}, we provide a detailed version of theorem~\ref{th:intro}, along with its proof. The proof contains several intermediate results interesting in their own rights, including a concentration inequality for determinants of sub-matrices (lemma~\ref{lem:concentration_determinant_matrices_main}), and weak-coherency results for points processes more general than the $\mathcal{P}_n$ and $\mathcal{P}$ considered above.
    \item In section~\ref{sect:ope_main}, \ref{sect:harmonic_ensemble}, \ref{sect:usvt} we apply our results to three distincts examples.
    \begin{itemize}
        \item In section~\ref{sect:ope_main}, we revisit the coreset guarantees established in the example from \cite{bardenet2021determinantal} mentioned above for a different, easier to define discrete DPP: the so-called discrete multivariate orthogonal polynomial ensemble~\cite{tremblay2019determinantal}.
        Using our results to show that it is weakly coherent with a continuous process on $\mathcal{X} = [-1,1]^d$, we obtain another example of a DPP-based coreset that is strictly better than using independent sampling.
        \item In section~\ref{sect:harmonic_ensemble}, we introduce a novel DPP defined on a point cloud $X_n \subseteq \R^d$ living on a \emph{compact manifold}. This process is built from a graph associated to $X_n$, and similar in spirit to that of~\cite{tremblay2017graph}. We show that it is weakly coherent with the \emph{harmonic ensemble} associated to the manifold, a continuous DPP for which various better-than-independent guarantees have been established (\emph{e.g.},~\cite{levi2024linear,borda2024riesz}). To the best of our knowledge, this is the first process with this kind of property.
        \item Finally, in section \ref{sect:usvt}, we study a more exotic example, in which one observes a \emph{latent position random graph}~\cite{crane2018probabilistic}. That is, a graph with independent Bernoulli edges $a_{ij} \sim \text{Bernoulli}(\mathcal{K}(x_i,x_j))$, so that the adjacency matrix is a (very) noisy version of the Gram matrix and the $x_i$'s are not observed. We show that it is possible to build an estimate of $\mathcal{K}$ that satisfies condition \emph{ii)} of theorem \ref{th:intro}, so that sampling the nodes of the graph with the associated DPP is akin to sampling the underlying unobserved latent space. We remark that this example is more an illustrative proof-of-concept rather than a practical approach; in particular, determinantal sampling on graphs usually focuses on using \emph{roots of random forests} ~\cite{avena2018random,tremblay2017graph}, that are associated to a different kernel. Studying the potential infinite-graph limit of the latter is an important avenue for future work.
    
    \end{itemize}
    
\end{itemize}
Most of the technical arguments in sections~\ref{sect:awc_compact},~\ref{sect:ope_main},~\ref{sect:harmonic_ensemble} and~\ref{sect:usvt} are relegated to the supplementary material.

\subsection{Related work}
\label{sect:related_work}

Despite the large body of theoretical and applicative work on DPPs, and the natural questions that arise from subsampling iid data with discrete DPPs, discrete-to-continuous analyses of DPPs remain somewhat scarcely studied.
The work closest to our own is certainly~\cite{bardenet2021determinantal}, as described above. Nevertheless, several lines of work bear similarities with the present paper.

\subsubsection{Limits of DPP on the same space} Contrary to our discrete-to-continuous limit, comparison of DPPs defined on a \emph{same} space, continuous or discrete, is a more mature topic.\footnote{This statement may confuse the most careful readers, as the sample $X_n$ indeed belongs to $\mathcal{X}$. The crucial difference is that existing results consider DPPs defined with measures that are \emph{absolutely continuous with each other}, whereas we are interested in the convergence of processes defined respectively on the point sets of some finite (discrete) $X_n$ and, \emph{e.g.}, the Borel sets of $\mathcal{X} \subseteq \R^d$. See remark~\ref{rem:discrete_vs_continuous} for an extended discussion.}
    Limits of DPPs are classically framed as a convergence in law, or \emph{weak convergence}; for completeness, we review this notion in section~\ref{sect:weak_convergence} of the appendix. 
    The study of weak convergence is typically motivated by mathematical physics, where the interest lies in a sequence of processes defined on a same continuous space and whose \emph{sample-size} $m=|\mathcal{S}|$ grows to infinity, so as to statistically describe the position of a large number of particles (in the case of DPPs, fermions). This is different from our setting, where the \emph{size of the underlying discrete space} $X_n$ grows to infinity instead. For our purpose in particular, we stress that weak convergence is a purely \emph{asymptotic} framework.  Weak \emph{coherency}, on the other hand, provides a simple framework for which we are able to obtain \emph{non-asymptotic} results that pertain directly to statistics of typical interest to the practitioner.\footnote{We remark that, though the topology of the weak convergence of point processes over $\mathcal{X}$ may be theoretically metrizable, it is very cumbersome to manipulate. To the best of our knowledge, indeed, non-asymptotic results for weak convergence of DPPs have not appeared in the literature.}
    
    Sufficient conditions for weak convergence to occur have been derived in several seminal papers, for DPPs defined with respect to a same reference measure $\mu$~\cite{soshnikov2000determinantal,shirai2003random}. For instance, when $\mathcal{X} = \R^d$, weak convergence takes places for Hermitian kernels if $\mathcal{K}_n: \mathcal{X} \times \mathcal{X} \rightarrow \R$ converges to $\mathcal{K}$ in the weak operator topology and the trace of $\mathcal{K}_n$ converges to that of $\mathcal{K}$ over each compact of $\mathcal{X}$. For more general spaces (the same that we define in section~\ref{sect:background_DPP}) and kernels, it has been shown that weak convergence occurs as soon as $\mathcal{K}_n: \mathcal{X} \times \mathcal{X} \rightarrow \C$ converges to $\mathcal{K}$ uniformly over each compact of $\mathcal{X} \times \mathcal{X}$. Note that this convergence \emph{of kernels} is quite different from the purely discrete criteria we formulate in theorem~\ref{th:intro}. To the best of our knowledge, and even though limits of determinantal processes remain an active area of research (\emph{e.g.},~\cite{katori2022scaling}), these criteria have not been re-visited to accommodate discrete-to-continuous limits. In particular, our examples typically fall outside the scope of these existing results, as the underlying space $\mathcal{X}$ is different from the (random!) discrete space $X_n$, and the kernel $K_n$ and reference measure $\mu_n$ are \emph{not even defined} outside of $X_n$.

\subsubsection{Estimating DPPs} 
While our framework consists in building a discrete kernel $K_n$ that \emph{estimates the Gram matrix} $\mathcal{K}_{\vert X_n \times X_n}$ in order to ensure convergence of a discrete DPP on iid data towards a continuous one, many works directly observe discrete data \emph{drawn from a DPP}, and aim to learn the latter in a parametric or non-parametric way~\cite{kulesza_determinantal_2012}. In particular, the \emph{maximum-likelihood-estimation} (MLE) of discrete DPPs has received a lot of attention: it has been shown to be both $\mathrm{NP}$-hard in general, and hard to approximate in some regimes~\cite{grigorescu2022hardness}, but simple estimators  with theoretical guarantees have been derived in less general settings~\cite{urschel2017learning,gourieroux2025simple}. In the continuous setting, we can for instance mention~\cite{fanuel2021nonparametric} which leverages the  Reproducing Kernel Hilbert Spaces machinery to perform the estimation.

While these settings are different from our own, they might benefit from the tools we develop. 
For continuous DPPs, \cite{poinas2023asymptotic} aims to estimate a continuous DPP in a parametric family from a single \emph{infinite} sample observed \emph{through growing windows}, and provides \emph{asymptotic} consistency guarantees on the likelihood function as the number $n$ of observed datapoints grows. This is similar to our growing data-setting, and we expect that a variant of theorem~\ref{th:intro} we describe in section~\ref{sect:awc_compact} (propositions~\ref{prop:compact_uniform_kernel} and~\ref{prop:compact_mean_kernel}) could be useful to describe the concentration of the linear statistics of the MLE estimator in a \emph{non-asymptotic} manner.

\subsection{Notations}

We denote by $\R$ and $\C$ the fields of real and complex numbers, by $[n] = \{1,...,n\}$ the first $n$ integers and, for admissible space $\Gamma$ and measure $\gamma$, by $\Vert . \Vert_{L^p(\Gamma,\gamma)}$ the usual $p$-th Lebesgue norm with respect to the measure $\gamma$. Whenever the space is implicit from context and there is no ambiguity, we simply write $\Vert . \Vert_{L^p(\gamma)}$. Throughout the document, we denote by $\mathbb{P}$ a probability with respect to the \emph{randomness on point processes} (usually for \emph{determinantal} point processes, with respect to the objects $(X_n, K_n)$ described in the introduction), whereas $\prb$, $\mathbf{E}$ and $\mathbf{Var}$ denote probabilities, expectations and variances with respect to \emph{samples} of a given process.

\section{Background on (determinantal) point processes}
\label{sect:background_DPP}

As we deal with point processes defined over different types of (continuous or discrete) spaces, we provide a general measure-theoretical definition; our setting coincides with the most general one in which the theory of determinantal point processes has been explicitly studied~\cite{shirai2003random}. The reader less familiar with this precise vocabulary can keep in mind that this setting is flexible enough to deal with (determinantal) point processes over domains in $\R^d$, manifolds, or discrete spaces, and all common reference measures $\mu$ that appear in the literature. For a reference on measure-theoretical and topologogical notions, we refer to, \emph{e.g.},~\cite{dieudonne1960treatise}. 

\subsection{Correlation functions and determinantal point processes}
\label{sect:correlation_and_DPPs}

Let $\Gamma$ be a second-countable locally compact Hausdorff space. A \emph{(simple) point process} $\mathcal{P}$ is a probability distribution over \emph{locally finite subsets} of $\Gamma$: subsets $\mathcal{S} \subseteq \Gamma$ such that  
\begin{equation*}
    \#(\mathcal{S} \cap C) < \infty  
\end{equation*}
for all compact sets $C \subseteq \Gamma$~\cite{daley2003introduction}. 

\subsubsection{Linear statistics and correlation functions} Recall that, for any bounded measurable function $\varphi_r: \Gamma^r \rightarrow \R$ and $\mathcal{S} \subseteq \Gamma$, we denote their ($r$-point) linear statistics and their expectation by
\begin{equation}
     \Lambda^{(\varphi_r)}(\mathcal{S}):= \sum_{\substack{{x_1,...,x_r \in \mathcal{S}} \\ {x_i \neq x_j}}} \varphi_r(x_1,...,.x_r),\qquad \Phi^{(\varphi_r)}(\mathcal{P}) :=  \mathbf{E}_{\mathcal{S} \sim \mathcal{P}}\left[ \Lambda^{(\varphi_r)}(\mathcal{S})\right]
\end{equation}
In the remainder of the paper, we consider sets $\mathcal{S}$ sampled from a point process $\mathcal{P}$, and loosely refer to the $\Lambda^{(\varphi_r)}(\mathcal{S})$ as the linear statistics \emph{of the point process $\mathcal{P}$} (though we stress that those only depend on $\varphi_r$ and $\mathcal{S}$). 

Different properties of point processes can be described using their \emph{correlation functions} (also called product density functions, or joint intensities), provided they exist. Given a Radon measure $\gamma$ on $\Gamma$, a locally integrable function $\rho_r: \Gamma^r \rightarrow \R$ is called the \emph{$r$-point correlation function} of a point process, if for \emph{any} bounded measurable $\varphi_r$, it holds that:
\begin{equation}
    \label{eq:correlation}
    \Phi^{(\varphi_r)}(\mathcal{P}) = \int_{\Gamma^r} \varphi_r \rho_r d\gamma^{\otimes r} := \int_{\Gamma^r} \varphi_r(x_1,...,x_r) \rho_r(x_1,...,x_r) d\gamma(x_1)...d\gamma(x_r)  
\end{equation}
where $\int_{\Gamma^r} \varphi_r \rho_r d\gamma^{\otimes r}$ is a shorthand notation that we often use for convenience.\footnote{In order to define the integral in equation~\eqref{eq:correlation}, one endows $\Gamma$ with its Borel $\sigma$-algebra.} In other words, if they exist, correlation functions express the expectation of any linear statistics of the point process as an $L^2$-inner product with respect to a reference measure $\gamma$.

In many cases, it is actually possible to \emph{define} a point process by its correlation functions $(\rho_r)_r$, and general admissibility conditions under which those do define a (unique) point process have been characterized~\cite{lenard1973correlation,lenard1975states1,lenard1975states2,daley2008introduction}. 

\subsubsection{Determinantal point processes} DPPs are precisely a family of point processes defined by their correlation functions~\cite{macchi1975coincidence} (for DPPs defined over finite sets, this definition simplifies and there is no need to refer to correlation functions; see section~\ref{sect:finite_DPP} below). A point process is \emph{determinantal} if and only if there exists a kernel $K: \Gamma \times \Gamma \rightarrow \C$, called the \emph{correlation kernel} (also, marginal kernel) of the process, such that 
\begin{equation}
    \label{eq:det}
    \rho_r(x_1,...,x_r) = \det\left( \left[K(x_i,x_j)\right]_{i,j=1}^r \right),  
\end{equation}
where $\left[K(x_i,x_j)\right]_{i,j=1}^r$ denotes the $r \times r$ matrix with entries $K(x_i,x_j)$. Throughout the paper, we denote by $\mathrm{DPP}(K,\gamma)$ the determinantal point process with kernel $K$ with respect to a measure $\gamma$, and by $\rho_r[K]$ the corresponding $r$-point correlation function. Note that not all pairs $(K,\gamma)$ define a (unique) DPP. One way to ensure this is the following, classical theorem.

\subsubsection{The Macchi-Soshnikov theorem}
The existence condition is stated under the assumption that $K$ is continuous, Hermitian, and that its associated integral operator $T_K: L^2(\Gamma;\gamma) \rightarrow L^2(\Gamma;\gamma)$ is \emph{locally trace-class}. Here, $T_K$ is defined by  
\begin{equation}
    (T_Kf)(x) = \int_{\Gamma} f(x) K(x,y)d\gamma(y).  
\end{equation}
By Mercer's theorem for non-negative definite operators, the locally trace-class condition amounts to requiring that
\begin{equation}
    \mathrm{Tr}_\gamma(K) :=\int_{C} K(x,x) d\gamma(x) < \infty  
    \label{eq:trace-class}
\end{equation}
for all compact sets $C \subseteq \Gamma$.
Let us remark that, if $\Gamma$ is further compact, Mercer's theorem also ensures that $T_K$ is a Hilbert-Schmidt operator, and its Hilbert-Schmidt norm is identified with
\begin{equation}
    \Vert K \Vert_{L^2(\gamma^{\otimes 2})}^2 := \int_{\Gamma^2} \vert K(x,y) \vert^2 d\gamma(x) d\gamma(y) < \infty.
    \label{eq:hilbert-schmidt}
\end{equation}
The Macchi-Soshnikov theorem then asserts the following.

\begin{thm}[Macchi-Soshnikov theorem~\cite{macchi1975coincidence,soshnikov2000determinantal,shirai2003random}]
    Under the previous assumptions, ${K}$ defines a DPP if and only if the eigenvalues of $T_K$ lie in $[0,1]$. In that case, it is unique.
\end{thm}

Though it is the most common setting in the literature, we note that the kernel $K$ need not be Hermitian to define a DPP, but that the theory is less developed in that case. In the case of finite spaces $\Gamma$ (see section~\ref{sect:finite_DPP} below), we refer the reader to~\cite{arnaud2024determinantal} for a systematic study of DPPs with non-Hermitian kernels, including a characterization of those kernels that define a DPP. For more extensive background on DPPs see, \emph{e.g.},~\cite{daley2003introduction,hough_determinantal_2006,kulesza_determinantal_2012,katori2020determinantal,baccelli2024random,IntroDPP}.

For the remainder of the paper, \textbf{we assume that the kernel $K$ is continuous} for all the DPPs we consider, and note that this is \emph{always} the case for kernels defined over discrete spaces.\footnote{As those are endowed with the discrete topology, for which all maps are continuous.} Moreover, unless otherwise stated we assume that $\mathcal{X}$ is a \textbf{compact}, second-countable Hausdorff space.



\subsection{Determinantal point processes over finite spaces}
\label{sect:finite_DPP}

In this work, a key role is played by DPPs defined with respect to measures $\gamma_n$ with finite support $X_n$ of size $n$, usually called ``discrete'' DPPs. 
The previous definitions then significantly simplify, and one can show the following: a DPP with kernel $K_n: X_n \times X_n \rightarrow \C$ with respect to the measure $\gamma_n: X_n \rightarrow \R$ over $X_n$ is a probability distribution on $2^{X_n}$ such that  
\begin{equation*}
    \prb_{\mathcal{S}}(A \subseteq \mathcal{S}) =  \det\left( \left[K_n(x,y)\right]_{x,y \in A} \right) \prod_{x \in A} \gamma_n(x)  
\end{equation*}
for all $A \subseteq X_n$. In words, the probability of observing $A$ in the sample $\mathcal{S}$ is, up to a factor depending on the measure $\gamma_n$, given by the determinant of the matrix $\left[K_n(x,y)\right]_{x,y \in \mathcal{S}}$ restricted to its rows and columns indexed by $A$.
There are two typical choices of $\gamma_n$ in the literature, equivalent up to a re-normalization of the kernel:
\begin{itemize}
    \item the case of the counting measure $\gamma_n = \sum_{x \in X_n} \delta_x$,
    \begin{equation}
        \prb_{\mathcal{S}}(A \subseteq \mathcal{S}) =  \det\left( \left[K_n(x,y)\right]_{x,y \in A} \right),
        \label{eq:counting_measure}
    \end{equation}
    which is the most common definition of DPPs over finite spaces~\cite{kulesza_determinantal_2012}; 
    \item the case of the empirical measure $\gamma_n = \mu_n := \frac{1}{n}{\sum_{x \in X_n} \delta_x}$,
    \begin{equation}
        \prb_{\mathcal{S}}(A \subseteq \mathcal{S}) = \frac{1}{n^{\vert A \vert}} \det\left( \left[K_n(x,y)\right]_{x,y \in A} \right),
        \label{eq:empirical_measure}
    \end{equation}
    which is the convention that we adopt in this paper for convenience.
\end{itemize}

When choosing $\mu_n$, as we do in the rest of this paper, the Macchi-Soshnikov condition for Hermitian kernels is satisfied when \emph{the kernel matrix} $K_n$ has its eigenvalues in $[0,n]$. This translates the fact that the eigenvalues of \emph{the operator} $T_{K_n}$, defined with respect to $\mu_n$, are in $[0,1]$.\footnote{With respect to the counting measure, we would ask that these eigenvalues be in $[0,1]$.}

\begin{rem}
For our purpose, we stress that a \emph{discrete} DPP $\mathcal{P}_n = \textup{DPP}(K_n, \gamma_n)$ is entirely defined by a \emph{matrix} $K_n = \left[K_n(x,y)\right]_{x,y \in X_n}$ representing the kernel, and a \emph{vector} $\gamma_n = [\gamma_n(x)]_{x \in X_n}$ representing the reference measure, where we overload the notations $K_n$ and $\gamma_n$. Written like this, $K_n$ and $\gamma_n$ implicitly define an \emph{ordering} on $X_n = \{x_1,...,x_n\}$. In particular, given the kernel, the elements of $X_n$ need not even be known when sampling from $\mathcal{P}_n$, as a sample in $X_n$ corresponds to a subset of \emph{indices} in $[n]$. This is of particular importance for the example we consider in section~\ref{sect:usvt}, where the $x_i$'s are not observed and the matrix $K_n$ is constructed from some auxiliary available data. 
In this fashion, expectations of linear statistics for discrete DPPs can be conveniently expressed using only the matrix $K_n$ and the vector $\gamma_n$:
\begin{align}
    \int_{\Gamma^r} \varphi_r \rho_r[K_n] d\gamma_n^{\otimes r} & =  \int_{\Gamma^r} \varphi_r(x_1,...,x_r) \mathrm{det}\left([K_n(x_i,x_j)]_{i,j=1}^r]\right) d\gamma_n(x_1)...d\gamma_n(x_r) \notag \\
    & = \sum_{i_1,\ldots,i_r=1}^n \varphi_r(x_{i_1},...,x_{i_r}) \mathrm{det}\left([K_n]_{\{i_1,\ldots,i_r\}}\right) [\gamma_n]_{i_1}...[\gamma_n]_{i_r} \label{eq:corr_discrete}.
\end{align}
Whenever considered, such an ordering is always made explicit, and $K_n$ can refer to either the kernel or its matrix, depending on which notation is more convenient.
\end{rem}

\subsubsection{Random point process} Let us come now back to the question that motivated this work. Consider a subset $X_n$ randomly drawn in $\mathcal{X}$; for instance, iid according to a probability distribution $\mu$. A DPP $\mathcal{P}_n = \textup{DPP}(K_n, \mu_n)$ over $X_n$ is then a \textbf{random, discrete DPP}, where the randomness pertains to that of $X_n$ and, eventually, on other random quantities that might enter the computation of $K_n$ (\emph{e.g.}, as in section~\ref{sect:usvt}). In particular, we are interested in the concentration of the $\Phi^{(\varphi_r)}( \mathcal{P}_n)$ towards the $\Phi^{(\varphi_r)}(\mathcal{P})$ for some continuous DPP $\mathcal{P}$ \emph{with respect to the randomness on $\mathcal{P}_n$}. We stress that it must not be confused with randomness coming from the drawing of a point cloud $\mathcal{S}$ from either point processes $\mathcal{P}$ or $\mathcal{P}_n$. 

\begin{rem}
    \label{rem:discrete_vs_continuous}
    One could instead study an alternative, deterministic process: the \textbf{continuous} point process over $\mathcal{X}$ obtained by first drawing $X_n$ iid and then subsampling it with a discrete DPP. However this ``iid-then-DPP'' point process over $\mathcal{X}$ is \emph{not} determinantal in general. This is one reason we rather consider that $\mathcal{P}_n$ is a discrete, \emph{random} DPP. 
\end{rem}

\section{Linear statistics and weak coherency of point processes}
\label{sect:weak_consistency_general_section}

We begin by formally introducing the notion of \emph{weak coherency} for sequences of point processes, which guarantees the asymptotic concentration of their linear statistics. The material in this section is general, and also applies to point processes defined over any \emph{locally} compact space $\mathcal{X}$ such as, \emph{e.g.}, $\R^d$.

\subsection{Definition}

Two sequences of random point processes $\mathcal{P}_n$ and $\mathcal{Q}_n$ (that are not necessarily determinantal, discrete or continuous), are weakly coherent if the expectation of their linear statistics concentrate towards each other, as formulated in the following definition.

\begin{definition}[Weak coherency]
    \label{def:weak_coherence}
    Let $\mathcal{P}_n$ and $\mathcal{Q}_n$ be two sequences of \textbf{random} point processes, respectively over sequences of subspaces $\mathcal{X}_n \subset \mathcal{X}$ and $\mathcal{Y}_n \subset \mathcal{X}$ of $\mathcal{X}$.
    
    The sequence $\mathcal{P}_n$ is \emph{weakly coherent} with $\mathcal{Q}_n$ if, for all $r \in \N_{>0}$, compactly-supported, bounded and measurable functions $\varphi_r:\mathcal{X}^r \rightarrow \R$, all $\delta \in (0,1)$ and all $\eps >0$, there exists $N(\delta, \eps, \varphi_r)$ 
    such that
\begin{equation}
    n \geq N(\delta, \eps, \varphi_r) \Rightarrow \mathbb{P}\left( \left\vert \Phi^{(\varphi_r)}(\mathcal{P}_n) - \Phi^{(\varphi_r)}(\mathcal{Q}_n) \right\vert \geq \eps \right) \leq \delta.
    \label{eq:consistency_1}
\end{equation}
The rate $N(\delta, \eps, \varphi_r)$ is the \emph{weak-coherency concentration rate} of $\mathcal{P}_n$ towards $\mathcal{Q}_n$.
\end{definition}

Here, the probability $\mathbb{P}$ refers to the randomness of $\mathcal{P}_n$ and $\mathcal{Q}_n$. In other words, weak coherency holds if, for the concentration of the linear statistics $\left\vert \Phi^{(\varphi_r)}(\mathcal{P}_n) - \Phi^{(\varphi_r)}(\mathcal{Q}_n) \right\vert$, any small tolerance level $\eps$ and probability of failure $\delta$ can be reached for $n$ large enough. Note that the rate $N(\delta, \eps, \varphi_r)$ is allowed to depend on $\varphi_r$: it is not a uniform rate over those test functions. 
In section~\ref{sect:awc_compact}, our main result (theorem~\ref{th:detailed}) provides sufficient conditions for weak coherency to hold in the particular case where $\mathcal{P}_n$ is a discrete DPP over random iid data with respect to the empirical probability measure, and $\mathcal{Q}_n = \mathcal{P}$ is a fixed deterministic continuous DPP with respect to the probability measure $\mu$.

\begin{rem}[Flexibility of weak coherency]
As mentioned earlier, our definition of weak coherency also encompasses more general situations than discrete-to-continuous limits of DPPs, and allows to considerf point processes that are not necessarily DPPs, and can be random, discrete or continuous.
Moreover, while we focus on probability measures $\mu$ in this paper, the point processes may be defined with reference measures that are general Radon measures.\footnote{This last situation naturally occurs when the set $X_n$ is sampled in $\mathcal{X}$ according to a Poisson process defined with respect to a Radon measure $\mu$. Our results could actually be generalized to this case; see section~\ref{sect:discussion} for a brief discussion on graph-centric applications.}
In particular, it is flexible enough to
cover the case where the DPPs are \emph{both} discrete or both continuous; this setting is natural in, \emph{e.g.}, kernel estimation from empirical measurements~\cite{poinas2023asymptotic}. In the rest of the paper, our results will be formulated for the specific DPPs $\mathcal{P}_n$ and $\mathcal{P}$, but some technical results will be valid for general point processes.
\end{rem}

\begin{rem}[Approximation rate]\label{rem:epsn}
    Depending on context, it may be more convenient to reformulate probabilistic error bounds like equation~\eqref{eq:consistency_1} so as to emphasize the \emph{error rate} $\eps_n = \eps_{n}(\delta)$ instead. In particular, guarantees like the following are common in machine learning: for some $n$ and $\delta$, with probability at least $1-\delta$ the error $\left\vert \Phi^{(\varphi_r)}(\mathcal{P}_n) - \Phi^{(\varphi_r)}(\mathcal{Q}_n) \right\vert$ is bounded by $\eps_n$. We find this formulation to be slightly less flexible in some examples, and opted to express $n$ with respect to $\eps$ and $\delta$ instead. Of course, one may often freely switch between the two by inspecting $N(\cdot)$, and we indeed sometimes comment on achievable error rates $\eps_n$ along the way.
\end{rem}

\subsection{The moment mapping theorem}
\label{sect:lin_stats_moment_map}

Weak coherency guarantees the concentration of the expectation of linear statistics, and it is straightforward to notice that this concentration of expectations is preserved when considering, \emph{e.g.}, Lipschitz functions of linear statistics. In many situations, it is however the \emph{moments} of those linear statistics that are of primary importance:
for instance, their \emph{variance} has been used to quantify some advantages of DPPs over iid sampling~\cite{bardenet2024smallcoresetsnegativedependence}. In this section, we show that weak coherency as defined above indeed \emph{also} entails the concentration of moments.

Given $\varphi_r$ and a point process $\mathcal{P}$, we denote by $m_k^{\mathcal{P}} \left( \Lambda^{(\varphi_r)} \right)$, $\overline{m}_k^{\mathcal{P}} \left( \Lambda^{(\varphi_r)} \right)$ 
the $k$-th raw and central moments of $\Lambda^{(\varphi_r)}$ with respect to $\mathcal{P}$:
\begin{equation}
    m^{(\varphi_r)}_k(\mathcal{P})
    = \mathbf{E}_{\mathcal{S}\sim\mathcal{P}} \left[ \Lambda^{(\varphi_r)}(\mathcal{S})^k \right], \quad \overline{m}^{(\varphi_r)}_k(\mathcal{P}) = \mathbf{E}_{\mathcal{S}\sim\mathcal{P}}\left[ \left( \Lambda^{(\varphi_r)}(\mathcal{S}) - m_1^{(\varphi_r)}(\mathcal{P}) \right)^k \right]\, ,
\end{equation}
such that with our previous notations $\Phi = m_1$. By analogy with the so-called continuous mapping theorem for weak convergence, we call the following theorem the moment mapping theorem.

\begin{thm}[Moment mapping theorem]
    \label{th:moment_mapping}

    Let ${\varphi_r}: \mathcal{X}^r \rightarrow \R$ be a compactly-supported, bounded and measurable function. Let $\mathcal{P}_n, \mathcal{Q}_n$ be sequences of random point processes such that $\mathcal{P}_n$ is weakly coherent with $\mathcal{Q}_n$ with rate $N(\delta, \eps, \varphi_r)$. 

    Then, there exists a \emph{finite} family of compactly-supported, bounded and measurable functions $\left(\varphi^{i}\right)_{i=1}^L$ depending on $\varphi_r$, and a constant $s \in \R$ such that, for all $\delta \in (0,1)$ and $\eps > 0$,
    \begin{equation}
        \label{eq:sought_moment}
        n \geq \max_{i=1,\ldots, L} N\left(\frac{\delta}{L}, \frac{\eps}{s}, \varphi^i\right) ~\Rightarrow~ \mathbb{P}\left(\left\vert m_k^{(\varphi_r)} \left( \mathcal{P}_n \right) - m_k^{(\varphi_r)} \left( \mathcal{Q}_n \right)\right\vert \geq \eps \right) \leq \delta.
    \end{equation}

    In addition, if each of the moments $m_k^{(\varphi_r)}(\mathcal{Q}_n)$ are uniformly bounded in $n$ with probability $1$, the same result holds for the central moments with a different constant $\overline{s}$.
\end{thm}

This result hinges on a simple (but verbose) combinatorial expansion of $\left(\Lambda^{(\varphi_r)}\right)^k$. We state a detailed version of theorem~\ref{th:moment_mapping} with explicit constants in section~\ref{sect:proof_moment_mapping} of the supplementary material. We remark that the uniform-boundedness hypothesis necessary for central moments is trivially satisfied in our setting of interest, since $\mathcal{Q}_n = \mathcal{P}$ is a fixed, deterministic DPP. 

\begin{rem}\label{rem:epsn_moment}
    The moment mapping theorem translates a useful property: the rate for the moments, $N(\delta/L, \eps/s, \varphi_r)$, only changes its dependence in $\delta$ and $\eps$ \emph{by multiplicative constants} compared to the weak coherency rate $N(\delta, \eps, \varphi_r)$. In particular, when inverting the bound and examining the error rate $\eps_n$ for given $n,\delta$ instead (see remark~\ref{rem:epsn}), this generally results in the \emph{same} rate $\eps_n$ for the expectation of the linear statistics and their moments, up to multiplicative constants.
\end{rem}

\section{Weak coherency for discrete-to-continuous limits of DPPs}
\label{sect:awc_compact}

Our main result concerns weak coherency between a discrete DPP over iid data and a corresponding continuous DPP. Recall that we consider a probability measure $\mu$ over a compact $\mathcal{X}$, a set $X_n =\{x_1,...,x_n\} \subseteq \mathcal{X}$ of $n$ points sampled iid according to $\mu$. 

We are going to show the following detailed version of theorem~\ref{th:intro}, thus establishing two sufficient conditions for discrete-to-continuous weak coherency to take place.

\begin{thm}[Detailed rates for theorem~\ref{th:intro}]
    \label{th:detailed}
    Let $\mathcal{K}$ be kernel and $(K_n)_{n}$ a sequence of $n \times n$ matrices such that all pairs $(\mathcal{K},\mu)$ and $(K_n,\mu_n)$ satisfy the conditions of the Macchi-Soshnikov theorem. If for all $\delta \in (0,1)$ and $\eps >0$, there exists $N_K(\delta,\eps)$ such that either:
\begin{enumerate}[label=\roman*)]
    \item $\mathcal{K}$ and $K_n$ are complex-valued, and
    \begin{equation*}
        \label{eq:max_condition_kernel}
        n \geq N_K(\delta, \eps) ~\Rightarrow~ \mathbb{P}\left(\max_{i,j \in [n]} \left\vert [K_n]_{i,j} - \mathcal{K}(x_i,x_j) \right\vert \geq \eps\right) \leq \delta;
    \end{equation*}
    \item $\mathcal{K}$ and $K_n$ are real-valued, and
    \begin{equation*}
        \label{eq:mean_condition_kernel}
        n \geq N_K(\delta, \eps) ~\Rightarrow~ \mathbb{P}\left(\max\left(\left\Vert \tfrac{K_n - \mathcal{K}_{\vert X_n \times X_n}}{n} \right\Vert_F ,
            \left\vert \mathrm{tr}\left( \tfrac{K_n}{n} \right) - \mathrm{tr}\left( \tfrac{\mathcal{K}_{\vert X_n \times X_n}}{n} \right)\right\vert \right)\geq \eps\right) \leq \delta,
    \end{equation*}
\end{enumerate}
    then the sequence $\mathcal{P}_n = \textup{DPP}(K_n, \mu_n)$ is weakly coherent with $\mathcal{P} = \textup{DPP}(\mathcal{K}, \mu)$. 
In particular, for any measurable and bounded $\varphi_r: \mathcal{X}^r \rightarrow \R$, we have
    \begin{equation}
        \label{eq:mean_condition}
        N(\delta, \eps, \varphi_r) = \max\left(N_K\left(\frac{\delta}{4}, \mathcal{K}_\infty\right),~ N_K\left(\frac{\delta}{4}, \frac{\eps}{c} \right), ~ \frac{2^{r+5} \beta \log\left( \frac{4}{\delta} \right)}{\eps^2} , ~  \frac{C_r}{\eps} \right),
    \end{equation}
    with $c = 2^r (r\times r!) \|\varphi_r\|_{L^1(\mu_n^{\otimes r})} \mathcal{K}_\infty^{r-1}$ under hypothesis i), and $c = 2^r (r\times r!) \|\varphi_r\|_{L^\infty(\mu_n^{\otimes r})} \mathcal{K}_\infty^{r-1}$ under hypothesis ii).
In both cases, we have:
\begin{equation*}
    \beta = \max_{\{x_1,...,x_r\} \in \mathcal{X}^r}\det\left( \left[ \mathcal{K}(x_{j},x_{k}) \right]_{j,k=1}^r \right) \Vert \varphi_r \Vert_{\infty}, \ \ C_r = 2^{r+1} \left(\mu^{\otimes r}(\mathrm{supp}(\varphi_r)) \beta + \beta^2 \right),
\end{equation*}
where $\Vert \varphi_r \Vert_\infty = \max_{z \in \mathcal{X}^r} \vert \varphi_r (z) \vert$.
\end{thm}

Note that the approximation condition of $\mathcal{K}$ by $K_n$ is more relaxed under condition \emph{ii)}, but that we require the kernels to be real-valued. We believe a similar result should hold true for complex-valued kernels as well, but our proof techniques do not generalize to this case. We also remark that our proof does generalize to sequences of DPPs defined by \emph{non-Hermitian} kernels, provided those exist, and with exact same rates.

Note that, for the Gram matrix $K_n = \mathcal{K}_{\vert X_n \times X_n}$, weak coherency \emph{always} holds with rate $N_K(\cdot)=0$. Hence, in any circumstances, if given access to the kernel $\mathcal{K}(x_i,x_j)$, one can always construct a discrete DPP that is weakly coherent with $\text{DPP}(\mathcal{K},\mu)$. The meat of our theorem lies in the \emph{stability conditions} on $K_n$ that still allows for weak coherency to hold. As we will see in the examples of sections \ref{sect:ope_main}, \ref{sect:harmonic_ensemble} and \ref{sect:usvt}, there are indeed many cases of applications where the true kernel $\mathcal{K}$ is inaccessible, in which case controlling its approximation becomes the main goal of the practitioner.

\begin{rem}
    \label{rem:threshold_wc_rates}

    Following remark~\ref{rem:epsn} and examining the rate $\eps_n$ for given $n, \delta$, the term $\frac{2^{r+5} \beta \log\left( \frac{4}{\delta} \right)}{\eps_n^2}$ imposes a rate no faster than $\eps_n = O\left(\sqrt{\frac{\log\left( \frac{1}{\delta} \right)}{n}}\right)$. This is the typical convergence rate for most statistics over $n$ iid variables (as obtained by, \emph{e.g.}, McDiarmid's concentration inequality). This rate may be achieved (for instance, for the Gram matrix) or, depending on the example, the convergence rate $N_K$ of the kernel may incur a strictly slower rate $\eps_n$. 
\end{rem}

\begin{proof}[Proof outline for theorem~\ref{th:detailed}]
    Let $\varphi_r: \mathcal{X}^r \rightarrow \R$ be any bounded measurable function on the compact $\mathcal{X}^r$. Since we are dealing with DPPs, the linear statistics can be expressed using determinants, and we first need to show that, for any $\delta \in (0,1)$ and $\eps > 0$, there exists $N(\delta, \eps, \varphi_r)$ beyond which, with probability at least $1 - \delta$,
    \begin{equation*}
        \left\vert \int_{X_n^r} \varphi_r \rho_r[K_n] d\mu_n^{\otimes r} - \int_{\mathcal{X}^r} \varphi_r \rho_r[\mathcal{K}] d\mu^{\otimes r} \right\vert \leq \eps\, .
    \end{equation*}
    where we recall the shorthand notations in equations~\eqref{eq:correlation} and~\eqref{eq:corr_discrete}.
    
    Applying the triangle inequality, we obtain
    \begin{align*}
    \underbrace{\left\vert \int_{X_n^r} \varphi_r \rho_r[K_n] d\mu_n^{\otimes r} - \int_{\mathcal{X}^r} \varphi_r \rho_r[\mathcal{K}] d\mu^{\otimes r} \right\vert}_{E_n} & \leq \underbrace{\left\vert \int_{X_n^r} \varphi_r \left( \rho_r[K_n] - \rho_r[\mathcal{K}] \right) d\mu_n^{\otimes r} \right\vert}_{E^K_n}  \\
    & \hspace{0.4cm} + \underbrace{\left\vert \int_{X_n^r} \varphi_r \rho_r[\mathcal{K}] d\mu_n^{\otimes r} - \int_{\mathcal{X}^r} \varphi_r \rho_r[\mathcal{K}] d\mu^{\otimes r} \right\vert}_{E^\mu_n},
    \end{align*}
    so that we are left with establishing the concentration for each summand. 
    Classically, we will derive $N\left(\delta,\eps, \varphi_r\right)$ such that for any $n \geq N\left(\delta,\eps, \varphi_r\right)$,
    \begin{equation*}
        \mathbb{P}\left( E_n^K \leq \frac{\eps}{2} \right) \geq 1 - \frac{\delta}{2} \ \ \text{and} \ \ \mathbb{P}\left( E_n^\mu \leq \frac{\eps}{2} \right) \geq 1 - \frac{\delta}{2}.
    \end{equation*}
    The result then follows by a union bound.

\end{proof}

We call $E_n^K$ the \emph{kernel error} and $E_n^\mu$ the \emph{measure error}, whose concentration will be established in sections~\ref{sect:compact_kernel} and~\ref{sect:compact_measure} respectively. The former is due to the deviation from $K_n$ to the Gram matrix $\mathcal{K}_{\vert X_n \times X_n}$, and in particular would vanish for $K_n =\mathcal{K}_{\vert X_n \times X_n}$. It will be bounded through the use of novel stability bounds for determinants. The latter is due to the deviation from $\mu$ to the empirical measure $\mu_n$, and will be bounded by classical concentration inequalities for iid variables.

The rest of this section is devoted to these technical computations.
We then describe applications of our results in sections~\ref{sect:ope_main},~\ref{sect:harmonic_ensemble} and~\ref{sect:usvt}.

\subsection{Concentration bounds for the kernel error}
\label{sect:compact_kernel}

We bound in this section the kernel error $E_K^n$, in a very general setting. We consider the general case of two random sequences of DPPs $\mathcal{P}_n = \mathrm{DPP}(K_n^\mathcal{P},\nu_n)$ and $\mathcal{Q}_n = \mathrm{DPP}(K_n^\mathcal{Q},\nu_n)$ defined with respect to a \emph{same} Radon measure $\nu_n$, over a second-countable locally compact Hausdorff space $\mathcal{X}$. We stress that, here, $\mathcal{X}$ is not necessarily compact, and $\nu_n$ is not necessarily normalized.

\begin{rem}\label{rem:same_measure}
Interestingly, the results of this section directly yield general conditions for weak coherency of sequences of DPPs \emph{defined with respect to the same measure} (that may be random and/or depend on $n$), since in that case the measure error $E_\mu^n$ vanishes. 
Since they also apply to non-compact spaces $\mathcal{X}$ and non-probability measures $\nu_n$, they could be applied to sequences of DPPs defined with respect to the Lebesgue measure on $\R^d$, or with respect to the counting measure $\sum_{x \in X_n} \delta_x$; we briefly discuss applications of these more general settings in section~\ref{sect:discussion}.
\end{rem}

We fix a bounded, measurable and compactly-supported function $\varphi_r: \mathcal{X}^r \rightarrow \R$, and aim to establish concentration bounds on a general kernel error of the form
\begin{equation*}
    \mathcal{E}_n^K = \left\vert \int_{C} \varphi_r \left( \rho_r[K_n^{\mathcal{P}}] - \rho_r[K_n^{\mathcal{Q}}] \right) d\nu_n^{\otimes r} \right\vert.
\end{equation*}
The kernel error $E^K_n = \int_{X_n^r} \varphi_r \left( \rho_r[K_n] - \rho_r[\mathcal{K}] \right) d\mu_n^{\otimes r}$ in the proof of theorem \ref{th:detailed} is a particular case with $\nu_n = \mu_n$.

\subsubsection{Technical assumptions}
We state our results in the general setting, with corollaries for the particular case of theorem~\ref{th:detailed}. Dealing with the general case requires a few careful definitions, \emph{that can safely be ignored in the setting of theorem~\ref{th:detailed}}. In particular, some quantities appearing in our bounds depend on the support of $\varphi_r$ and on the support of the measure $\nu_n$: in order to avoid a number of technicalities in the definition of these terms, and without loss of generality, we assume that $\mathrm{supp}(\varphi_r) = C_{\varphi_r}^r$ for some $C_{\varphi_r} \subseteq \mathcal{X}$, and we denote by $C_{\varphi_r, n} = C_{\varphi_r} \cap \supp(\nu_n)$ the intersection of this domain with $\supp(\nu_n)$. For convenience, we further assume that $C_{\varphi_r, n}$ is compact. Under the hypotheses of theorem~\ref{th:detailed}, we simply have $\supp(\nu_n) = C_{\varphi_r,n}= X_n$, and this is always the case.

\subsubsection{Concentration from uniform approximation}

We begin with a bound on $\mathcal{E}_n^K$ that depends on the maximum difference between $K_n^\mathcal{P}$ and $K_n^\mathcal{Q}$, which corresponds to the first assumption of theorem~\ref{th:detailed}. The proof of this result is detailed at the end of the section, after discussing its consequences and introducing some necessary lemmas.

\begin{prop}
    \label{prop:compact_uniform_kernel}
    Under our running assumptions, it holds that
    \begin{equation}
        \mathcal{E}_n^K \leq (r \times r!)  \Vert \varphi_r \Vert_{L^1(\nu_n^{\otimes r})} a_n^{r-1} \max_{x,y \in C_{\varphi_r, n}} \left\vert K_n^{\mathcal{P}}(x,y) - K_n^{\mathcal{Q}}(x,y) \right\vert,
        \label{eq:ineq_kernel_compact_uniform}
    \end{equation}
    where $a_n = \max \left( \max_{x,y \in C_{\varphi_r, n}} \vert K_n^{\mathcal{P}}(x,y) \vert , \max_{x,y \in C_{\varphi_r, n}} \vert K_n^{\mathcal{Q}}(x,y) \vert \right)$.
\end{prop}

As a consequence of proposition~\ref{prop:compact_uniform_kernel}, the concentration in max-norm of the kernel $K_n^{\mathcal{P}}$ towards $K_n^{\mathcal{Q}}$ over all compacts of $\mathcal{X}$ is a sufficient condition for weak coherency of $\mathrm{DPP}(K_n^{\mathcal{P}},\nu_n)$ with $\mathrm{DPP}(K_n^{\mathcal{Q}},\nu_n)$, and equation~\eqref{eq:ineq_kernel_compact_uniform} provides concentration rates for the $r$-points linear statistics at each order $r$ depending explicitly on the concentration rates of the kernels. In our discrete-to-continuous setting, $\nu_n = \mu_n$, $C_{\varphi_r,n} = X_n$, $K_n^\mathcal{P} = K_n$ and $K_n^\mathcal{Q} = \mathcal{K}$, and we obtain the following.

\begin{cor}[Kernel error for theorem~\ref{th:detailed}, $i)$]
    \label{cor:kernel_error_unif}
    Let $\mathcal{K}_\infty = \max_{x,y \in X_n} |\mathcal{K}(x,y)|$. Under the hypotheses of theorem~\ref{th:detailed} $i)$, for any $\delta \in (0,1)$ and $\eps>0$,
    \begin{equation}
        n \geq \max\left(N_K\left(\frac{\delta}{4}, \mathcal{K}_\infty\right),~ N_K\left(\frac{\delta}{4}, \frac{\eps}{c} \right)\right) ~\Rightarrow~ \mathbb{P}\left(E_n^K \geq \eps/2\right) \leq \frac{\delta}{2},
    \end{equation}
    where
    $$
    c = 2^r (r\times r!) \|\varphi_r\|_{L^1(\mu_n^{\otimes r})} \mathcal{K}_\infty^{r-1}.
    $$
\end{cor}

\begin{proof}
    [Proof of corollary \ref{cor:kernel_error_unif}]
    We first notice that, for $n \geq N_K\left(\frac{\delta}{4}, \mathcal{K}_\infty\right)$ and with probability at least $1-\frac{\delta}{4}$,
    $$
        \max_{i,j \in [n]} |[K_n]_{i,j}|\leq 2\mathcal{K}_\infty,
    $$
    so that $E_n^K \leq 2^{r-1} (r \times r!) \Vert \varphi_r \Vert_{L^1(\mu_n^{\otimes r})} \mathcal{K}_\infty^{r-1} \left( \max_{i,j \in [n]} \vert [K_n]_{i,j} - \mathcal{K}(x_i,x_j) \vert \right)$.
    Taking $n \geq N_K\left(\frac{\delta}{4}, \frac{\eps}{c} \right)$ for
    \begin{equation*}
        c = 2^r (r\times r!) \|\varphi_r\|_{L^1(\mu_n^{\otimes r})} \mathcal{K}_\infty^{r-1}
    \end{equation*}
    then ensures, by proposition \ref{prop:compact_uniform_kernel} and with probability at least $1-\frac{\delta}{2}$, that $E_n^K \leq \frac{\eps}{2}$.  Applying a union bound to bound the probability of either condition not holding yields the result.
\end{proof}

The demonstration of proposition~\ref{prop:compact_uniform_kernel} in the case of $\nu_n = \mu_n$ hinges on the following estimate, the proof of which is deferred to section~\ref{sect:proof_uniform_concentration_compact} of the appendix.

\begin{prop}
    \label{prop:uniform_determinant_concentration_matrix}
    Let $A$ and $B$ be two $n \times n$ matrices with complex coefficients, and $I_r = \{i_1,...,i_r\} \subseteq [n]$ a set of $r$ distinct indices. Then
    \begin{equation}
        \left\vert \det(A_{I_r}) - \det(B_{I_r}) \right\vert \leq r! \sum_{j = 1}^r \left( \max_{k,l \in [n]} \vert A_{k,l} \vert \right)^{j-1} \left(\max_{k,l \in [n]} \vert A_{k,l} - B_{k,l} \vert \right) \left( \max_{k,l\in [n]} \vert B_{k,l} \vert\right)^{r-j},
    \end{equation}
    where $M_{I_r}$ denotes the 
    $r \times r$ 
    matrix restricted to the rows and columns of $M$ indexed by $I_r$.
\end{prop}

The proof generalizes to obtain proposition~\ref{prop:uniform_determinant_concentration}, which applies to any measure $\nu_n$; we detail the relevant differences in section~\ref{sect:proof_uniform_concentration_compact} of the appendix. 

\begin{prop}
    \label{prop:uniform_determinant_concentration}
    For any given $x_1,...,x_r \in C_{\varphi,n}$, denote by $b(x_1,...,x_r)$ the difference
    \begin{equation}
        b(x_1,...,x_r) = \left\vert \det\left( \left[ K_n^{\mathcal{P}}(x_{j},x_{k}) \right]_{j,k=1}^r \right) - \det\left( \left[ K_n^{\mathcal{Q}}(x_{j},x_{k}) \right]_{j,k=1}^r \right) \right\vert.
    \end{equation}
    Then, under our running hypotheses, it holds that
    \begin{align}
        b(x_1,...,x_r) \leq r! \sum_{j = 1}^r M_{\mathcal{P}}^{j-1} \left(\max_{x,y \in C_{\varphi_r,n}} \vert K_n^\mathcal{P}(x,y) - K_n^\mathcal{Q}(x,y) \vert \right) M_\mathcal{Q}^{r-j}, \nonumber
    \end{align}
    where $M_\mathcal{P} = \max_{x,y \in C_{\varphi_r,n}} \vert K_n^{\mathcal{P}}(x,y) \vert$ and $M_\mathcal{Q} = \max_{x,y \in C_{\varphi_r,n}} \vert K_n^{\mathcal{Q}}(x,y) \vert$.
\end{prop}

Using proposition~\ref{prop:uniform_determinant_concentration}, we can finally deduce proposition~\ref{prop:compact_uniform_kernel}.

\begin{proof}[Proof of proposition~\ref{prop:compact_uniform_kernel}]
    Using the triangle inequality, we obtain
    \begin{align*}
    \mathcal{E}^K_n & \leq \int_{\mathrm{supp}(\varphi_r)} \left\vert \varphi_r(x_{1},...,x_{r}) \right\vert b(x_1,...,x_r) d\nu_n^{\otimes r}(x_1,...,x_r) \\
    & \leq \Vert \varphi_r \Vert_{L^1(\nu_n^{\otimes r})} \max_{x_1,...,x_r \in C_{\varphi_r, n}} \left\vert \det\left( \left[ K_n^{\mathcal{P}}(x_{j},x_{k}) \right]_{j,k=1}^r \right) - \det\left( \left[ K_n^{\mathcal{Q}}(x_{j},x_{k}) \right]_{j,k=1}^r \right) \right\vert,
    \end{align*}
    where $b$ is defined as in proposition~\ref{prop:uniform_determinant_concentration}, that we can then apply. It follows that
    \begin{align*}
        \mathcal{E}^K_n & \leq r! \Vert \varphi_r \Vert_{L^1(\nu_n^{\otimes r})} \left( \sum_{j = 1}^r \left( \max_{x,y \in C_{\varphi_r, n}} \vert K_n^{\mathcal{P}}(x,y) \vert \right)^{j-1} \left( \max_{x,y \in C_{\varphi_r, n}} \vert K_n^{\mathcal{Q}}(x,y)\vert \right)^{r-j} \right) \\
        & \hspace{0.4cm} \times \max_{x,y \in C_{\varphi_r, n}} \left\vert K_n^{\mathcal{P}}(x,y) - K_n^{\mathcal{Q}}(x,y) \right\vert ,
    \end{align*}
    from which we obtain inequality~\eqref{eq:ineq_kernel_compact_uniform}. 
\end{proof}

\subsubsection{Concentration from on-average approximation}

For real-valued kernels, we can relax the uniform approximation  in proposition~\ref{prop:compact_uniform_kernel}. Our results are stated under the same technical assumptions as in the preceding section. 

\begin{prop}
    \label{prop:compact_mean_kernel}
    Under our running assumptions, it holds that
    \begin{align}
        \mathcal{E}_n^K \leq (r \times r!)  \Vert \varphi_r \Vert_{L^\infty(\nu_n^{\otimes r})} (a_n')^{r-1} \max\left( \left\Vert K_n^{\mathcal{P}} - K_n^{\mathcal{Q}} \right\Vert_{L^2(\nu_n^2)}, \left\vert \mathrm{Tr}_{\nu_n}\left( K_n^{\mathcal{P}} \right) - \mathrm{Tr}_{\nu_n}\left(K_n^{\mathcal{Q}} \right) \right\vert \right),
        \label{eq:ineq_kernel_compact_mean}
    \end{align}
    where $a_n' = \max\left( \left\Vert K_n^{\mathcal{P}} \right\Vert_{L^2(\nu_n^2)}, \left\Vert K_n^{\mathcal{Q}} \right\Vert_{L^2(\nu_n^2)}, \mathrm{Tr}_{\nu_n}\left( K_n^{\mathcal{P}} \right),  \mathrm{Tr}_{\nu_n}\left( K_n^{\mathcal{Q}} \right) \right)$, 
    and the trace $\mathrm{Tr}$ for a given measure and kernel is defined as in equation~\eqref{eq:trace-class}.
\end{prop}
Interestingly, the bound involves local traces and $L^2$ norms, that are also involved in the hypotheses of the Macchi-Soshnikov theorem. Further, as we previously stated, it is likely that a similar result holds true for complex-valued kernels as well, but our current proof does not carry over to this case. We first discuss its consequence for discrete-to-continuous limits and introduce a key lemma giving its proof.

For the discrete-to-continuous setting of theorem~\ref{th:detailed}, $C_{\varphi_r,n} = X_n$ and $\nu_n = \mu_n$. In that case, we recognize the usual Frobenius norm and matrix-trace differences
\begin{align*}
\left\Vert K_n - \mathcal{K}_{\vert X_n \times X_n} \right\Vert_{L^2(\mu_n^2)} & = \left\Vert \frac{K_n - \mathcal{K}_{\vert X_n \times X_n}}{n} \right\Vert_F, \\
\left\vert \mathrm{Tr}_{\mu_n}\left( K_n \right) - \mathrm{Tr}_{\mu_n}\left( \mathcal{K}_{\vert X_n \times X_n} \right) \right\vert & = \left\vert \mathrm{tr}\left( \frac{K_n}{n} \right) - \mathrm{tr}\left( \frac{\mathcal{K}_{\vert X_n \times X_n}}{n} \right)\right\vert
\end{align*}
and, similarly to corollary~\ref{cor:kernel_error_unif}, obtain the following bound on the kernel error.

\begin{cor}[Kernel error for theorem~\ref{th:detailed}, $ii)$]
    \label{cor:kernel_error_matrix}
    Let $\mathcal{K}_\infty = \max_{x,y} |\mathcal{K}(x,y)|$. Under the hypotheses of theorem~\ref{th:detailed}, ii), for any $\delta \in (0,1)$ and $\eps>0$,
    \begin{equation}
        n \geq \max\left(N_K(\delta, \mathcal{K}_\infty),~ N_K\left(\delta, \frac{\eps}{c}\right)\right) ~\Rightarrow~ \mathbb{P}\left(E_n^K \geq \eps/2\right) \leq \frac{\delta}{2},
    \end{equation}
    where
    $$
    c = 2^r (r\times r!) \|\varphi_r\|_{L^\infty(\mu_n^{\otimes r})} \mathcal{K}_\infty^{r-1}.
    $$
\end{cor}
 
The main ingredient in the proof of proposition~\ref{prop:compact_mean_kernel} is the following lemma~\ref{lem:concentration_determinant_matrices_main} concerning sums of determinants of sub-matrices, the proof of which is deferred to section~\ref{sect:proof_mean_concentration_compact} of the appendix. We believe this result might be of independent interest. 

\begin{lem}[Determinant concentration lemma]
    \label{lem:concentration_determinant_matrices_main}
    
     Let $A$ and $B$ be two $n \times n$ matrices with real coefficients. Then,
     \begin{equation*}
         \sum_{\substack{{I_r \subseteq [n]} \\ {\vert I_r \vert = r}}} \vert \det(A)_{I_r} - \det(B)_{I_r}\vert \leq (r \times r!) M_{A,B}^{r - 1} \max\left(\left\Vert A - B \right\Vert_F , \left\vert \mathrm{tr}\left(A \right) - \mathrm{tr}\left( B \right)\right\vert \right),
         \label{eq:concentration_dets}
     \end{equation*}
     where $M_{A,B} = \max\left( h(A,B), t(A,B) \right)$, for $h(A,B) = \max\left( \Vert A \Vert_F, \Vert B \Vert_F \right)$ and $t(A,B) = \max\left( \mathrm{tr}(A), \mathrm{tr}(B) \right)$.
\end{lem}

More generally, the following holds, as detailed in section~\ref{sect:proof_mean_concentration_compact} of the appendix.

\begin{lem}[Determinant concentration lemma (general)]
    \label{lem:concentration_determinant_general}
     Suppose that the kernels $K_n^{\mathcal{P}}$ and $K_n^{\mathcal{Q}}$ are real valued, and let
     \begin{equation}
         \label{eq:def_d}
         d = \int_{\mathrm{supp}(\varphi_r)} \left\vert \det\left( \left[ K_n^{\mathcal{P}}(x_j,x_k) \right]_{j,k=1}^r \right) -  \det\left( \left[K_n^{\mathcal{Q}}(x_j,x_k) \right]_{j,k=1}^r \right) \right\vert d\nu_n^{\otimes r}(x_1,...,x_r).
     \end{equation} Then, under our running assumptions, it holds that
     \begin{align}
         d \leq (r \times r!) M_{\mathcal{P},\mathcal{Q}}^{r - 1} \max\left(\left\Vert K_n^{\mathcal{P}} - K_n^{\mathcal{Q}} \right\Vert_{L^2(\nu_n^2)} , \left\vert \mathrm{Tr}_{\nu_n}\left( K_n^{\mathcal{P}} \right) - \mathrm{Tr}_{{\nu_n}}\left( K_n^{\mathcal{Q}} \right)\right\vert \right), \nonumber
     \end{align}
     where $M_{\mathcal{P},\mathcal{Q}} =  \max\left( h(K_n^{\mathcal{P}},K_n^{\mathcal{Q}}), t(K_n^{\mathcal{P}},K_n^{\mathcal{Q}}) \right)$, for which, given two kernels $K_1$ and $K_2$, $h(K_1,K_2) = \max\left( \Vert K_1 \Vert_{L^2(\nu_n^2)}, \Vert K_2 \Vert_{L^2(\nu_n^2)} \right)$ and $t(K_1,K_2) = \max\left( \mathrm{Tr}_{{\nu_n}}(K_1), \mathrm{Tr}_{{\nu_n}}(K_2) \right)$.
\end{lem}

We can now prove proposition~\ref{prop:compact_mean_kernel}.

\begin{proof}[Proof of proposition~\ref{prop:compact_mean_kernel}]
    From the triangle inequality, we have
    \begin{align*}
        \mathcal{E}_n^K \leq d \Vert \varphi_r \Vert_{L^\infty(\nu_n^{\otimes r})} ,
    \end{align*}
where $d$ is defined as in equation~\eqref{eq:def_d}.
The result is then obtained by applying lemma~\ref{lem:concentration_determinant_matrices_main}, similarly to the proof of proposition~\ref{prop:compact_uniform_kernel}.
\end{proof}

\subsection{Concentration bounds for the measure error}
\label{sect:compact_measure}

We now focus on the concentration towards $0$ of the measure error $E^\mu_n$. 
We begin with a general result, probabilistically bounding the difference between an integral with respect to the measures $\mu_n^{\otimes r}$ and $\mu^{\otimes r}$. It is proved in section~\ref{sect:proof_convergence_empirical_measure} of the appendix, based on careful applications of classical concentration inequalities.

\begin{prop}
    \label{prop:compact_measure}
    Let $f: \mathcal{X} \rightarrow \R$ be a bounded and measurable function, such that $\vert f(x)\vert \leq \beta$ for all $x \in \mathcal{X}^r$. Then, for any $\widetilde{\eps} > 0$, with probability at least $1 - 2 \exp\left( \frac{- 2 \widetilde{\eps}^2}{n b_n} \right)$ over $X_n$,
    \begin{equation}
        \left\vert \int_{X_n^r} f d\mu_n^{\otimes r} - \int_{\mathcal{X}^r} f d\mu^{\otimes r} \right\vert \leq \eps + \frac{M}{n},
    \end{equation}
    where $M = 2^{r-1} \left( \beta + \Vert f \Vert_{L^1(\mu^{\otimes r})}\right)$ is a constant, and
    \begin{equation}
            b_n = \left( \beta \frac{\sum_{l = 1}^{r} \binom{r}{l} n^{l-1}}{n^r} \right)^2 \sim \frac{\beta^2}{n^2}.
    \end{equation}
\end{prop}

We note that the $\frac{M}{n}$ term is deterministic, and the probabilistic bound only concerns the $\widetilde{\eps}$ term.

In order to apply this result to the measure error $E^\mu_n$, we consider the function 
\begin{equation*}f(x_{i_1},...,x_{i_r}) = \varphi_r(x_{i_1},...,x_{i_r}) \det\left( \left[ \mathcal{K}(x_{i_j},x_{i_k}) \right]_{i,k=1}^r \right).
\end{equation*}
Under our running assumptions that $\varphi_r$ is bounded and $\mathcal{K}$ is continuous (hence bounded on the compact $\mathcal{X} \times \mathcal{X}$), $f$ is bounded as well. We further denote by $\varphi_{min},\varphi_{max} \in \R$ the lower and upper bounds of $\varphi_r$, so that
\begin{equation}
    \varphi_{min} \leq \varphi_r(x) \leq \varphi_{max}
\end{equation}
for all $x \in \mathrm{supp}(\varphi_r)$. We then readily obtain the following, as detailed in section~\ref{sect:preuve_corollaire_mesure} of the appendix. 

\begin{cor}
    \label{cor:erreur_mesure}
    Let $\mathcal{K}: \mathcal{X} \times \mathcal{X} \rightarrow \C$ be a bounded and measurable kernel, and
    \begin{equation}
        \label{eq:def_beta}
        \beta = \max_{\{x_1,...,x_r\} \in \mathcal{X}^r}\det\left( \left[ \mathcal{K}(x_{j},x_{k}) \right]_{j,k=1}^r \right) \max(-\varphi_{min},\varphi_{max}).
    \end{equation}
    Then, there is a constant $C_r$ such that, for any $\delta \in (0,1)$ and $\eps > 0$,
    \begin{equation}
        n \geq \max\left(\frac{2^{r+5} \beta \log\left( \frac{4}{\delta} \right)}{\eps^2}, \frac{C_r}{\eps} \right) \Rightarrow \mathbb{P}\left(E_n^\mu \geq \eps/2\right) \leq \frac{\delta}{2},
    \end{equation}
    where $C_r = 2^{r+1} \left( \mu^{\otimes r}(C) \beta + \beta^2 \right)$.
\end{cor}

\noindent Putting everything together, this concludes the proof of theorem~\ref{th:detailed}.

\section{Better-than-Poisson variance over $[-1,1]^d$}
\label{sect:ope_main}

We are now going to apply our results to several examples. In the first example, we revisit the use of the so-called \emph{multivariate orthogonal polynomial ensembles} for coreset construction, a family of DPPs that can be defined over both continuous and discrete spaces. There exists better-than-independent guarantees that have been obtained for instantiations of \emph{continuous} DPPs over $\mathcal{X} = [-1,1]^d$, some of which have been used to propose a surrogate \emph{discrete} DPP with strict better-than-independent coreset guarantees~\cite{bardenet2020monte,bardenet2021determinantal,bardenet2024smallcoresetsnegativedependence}. We show in this section that the same kind of guarantees can be obtained using the discrete multivariate orthogonal polynomial ensemble over an iid sample $X_n \subseteq \mathcal{X}$~\cite{tremblay2023extended}, \emph{without resorting to a surrogate process}. For the remainder of this section, {$m$ denotes the sample-size of these DPPs}, and will remain fixed. We take $\mathcal{X} = [-1,1]^d$, and note that our development would remain valid over any compact domain of $\R^d$. 

\subsection{Multivariate orthogonal polynomial ensembles}
\label{sect:def_ope}

In order to define the kernels of multivariate polynomial ensembles, we first define orthogonal polynomials associated to both continuously and finitely-supported measures. We consider a probability measure $\mu$ on $\mathcal{X}$ that is absolutely continuous with respect to the Lebesgue measure, and a subset $X_n = \{x_1,...,x_n\} \subseteq \mathcal{X}$ of $n$ points drawn iid according to $\mu$, with associated empirical measure $\mu_n$.

\subsubsection{Orthogonal polynomials} 
For families of degrees $(\beta_1,...,\beta_d)$, consider the set of monomial-evaluation functions on $\R^d$ of the form $x \mapsto x(1)^{\beta_1} ... x(d)^{\beta_d}$, where $x = (x(1),...,x(d)) \in \R^d$, and the ordering on these functions given by the graded lexical order on the degrees $\beta_i$'s. We denote by $(\mathcal{M}_i)_i$ this sequence of functions, and $M_i$ their restriction to the points of $X_n$.\footnote{That is, the $M_i$'s are functions from $X_n$ to $\R$.} Applying the Gram-Schmidt orthogonalization process to the first $m$ of these functions in $L^2(\mathcal{X};\mu)$ (resp. $L^2(X_n;\mu_n)$) yields a sequence of \emph{multivariate orthogonal polynomials} $(\mathcal{P}_i)_{i=1}^m$ (resp. $(P_i)_{i=1}^m$). More explicitly, we respectively have the following:
\begin{equation}
    \begin{cases}
        \mathcal{P}_1 = \frac{\mathcal{M}_1}{\Vert \mathcal{M}_1 \Vert_\mu} \\
        {\mathcal{P}'}_{k+1} = \mathcal{M}_{k+1} - \frac{\langle \mathcal{M}_{k+1}, \mathcal{P}_k\rangle_\mu}{\langle \mathcal{P}_k, \mathcal{P}_k \rangle_\mu} \mathcal{P}_k \\
        \mathcal{P}_{k+1} = \frac{{\mathcal{P}'}_{{k+1}}}{\Vert \widetilde{\mathcal{P}}_{k+1} \Vert_\mu},
    \end{cases}
\end{equation}
\begin{equation}
    \begin{cases}
        {P}_1 = \frac{{M}_1}{\Vert {M}_1 \Vert_{\mu_n}} \\
        {{P'}}_{k+1} = {M}_{k+1} - \frac{\langle {M}_{k+1}, \mathcal{P}_k\rangle_{\mu_n}}{\langle {P}_k, {P}_k \rangle_{\mu_n}} {P}_k \\
        {P}_{k+1} = \frac{{{P'}}_{{k+1}}}{\Vert {{P}}_{k+1} \Vert_{\mu_n}},
    \end{cases}
\end{equation}
where we used the notations
\begin{equation}
    \langle \mathcal{P}, \mathcal{Q} \rangle_\mu = \int_\mathcal{X} \mathcal{P}(x) \mathcal{Q}(x) d\mu(x), \ \ \Vert \mathcal{P} \Vert_\mu = \langle \mathcal{P}, \mathcal{P} \rangle_\mu,
\end{equation}
\begin{equation}
    \langle P, Q \rangle_{\mu_n} = \sum_{x \in X_n}  \frac{1}{n} P(x) Q(x), \ \ \Vert P \Vert_\mu = \langle P, P \rangle_{\mu_n}.
\end{equation}
Under very mild admissibility conditions ---\emph{e.g.}, that $\mu(A) > 0$ for some open subset $A \subseteq \mathcal{X}$---, these procedures do define (unique) sequences of orthogonal polynomials. We always make this assumption in the following.

\subsubsection{Multivariate orthogonal polynomial ensembles}
The continuous and discrete multivariate orthogonal polynomial ensembles are DPPs defined for the measures $\mu$ and $\mu_n$, with kernels $\mathcal{K}$ and $K_n$ parametrized by the first $m$ orthogonal polynomials with respect to these measures:
\begin{equation}
    \mathcal{K}(x,y) = \sum_{i = 1}^m \mathcal{P}_i(x)\mathcal{P}_i(y) \ \ \forall x,y \in \mathcal{X},
\end{equation}
\begin{equation}
    [K_n]_{j,k} = \sum_{i = 1}^m {P}_i(x_j){P}_i(x_k) \ \ \forall j,k \in [n],
\end{equation}
where $m$ is some positive integer. We note that those DPPs exist, as we are under the hypotheses of the Macchi-Soshnikov theorem.\footnote{By construction, the orthogonal polynomials form an orthonormal basis with respect to the inner product $\langle . , . \rangle_\mu$ (resp. $\langle . , . \rangle_{\mu_n}$), so that the eigenvalues of $T_\mathcal{K}$ (resp. $K_n$) are in $\{0,1\}$ (resp. $\{0,n\}$).} Moreover, those are so-called projection DPPs~\cite{hough_determinantal_2006}, so that samples $\mathcal{S}$ from $\mathrm{DPP}(\mathcal{K},\mu)$ or $\mathrm{DPP}(K_n,\mu_n)$ have cardinality $m$ almost surely.  

We can now begin to show the weak consistency of these two multivariate orthogonal polynomial ensembles. We stress that \emph{$m$ is fixed} throughout our analysis, whereas we take $n \rightarrow \infty$.

\subsection{Weak consistency of orthogonal polynomial ensembles}

The first step in our development is to notice that orthogonal polynomials with respect to $\mu_n$ concentrate towards their continuous counterparts.

\begin{prop}
    \label{prop:cv_orthogonal_polynomials}
    Under the setting of section~\ref{sect:def_ope}, there exist three constants $A,B,C >0$ such that, for any $\widetilde{\delta} \in (0,1)$ and $\widetilde{\eps} > 0$,
    \begin{equation}
        \label{eq:condition_concentration_poly}
        n \geq \frac{A^2 B^2 \log\left(\frac{C}{\widetilde{\delta}}\right)}{\widetilde{\eps}^2} ~ \Rightarrow ~ \mathbb{P}\left( \max_{i \in [m]} \max_{x \in X_n} \vert \mathcal{P}_i(x) - P_i(x) \vert \geq \widetilde{\eps} \right) < \widetilde{\delta}.
    \end{equation}
\end{prop}

It is intuitive that this concentration should take place but, to the best of our knowledge, results of this type have not appeared in previous literature.\footnote{Despite a growing interest in asymptotics related to orthogonal polynomials with respect to $\mu_n$ and, in particular, in the so-called Christoffel function, the inverse of the diagonal of $K_n$, that has found applications in, \emph{e.g.}, support estimation for probability density functions on $\R^d$~\cite{lasserre2022christoffel}.} Our proof is deferred to section~\ref{sect:proof_ope} of the supplementary material, where the rates are derived from applications of McDiarmid's inequality.

It is then evident that $K_n$ concentrates towards $\mathcal{K}$, and we are able to apply theorem~\ref{th:detailed}, $i)$.

\begin{thm}[Weak coherency of multivariate orthogonal polynomial ensembles]
    \label{th:cv_polynomial_ensembles}
    Under the setting of section~\ref{sect:def_ope},  $\mathrm{DPP}(K_n,\mu_n)$ is weakly coherent with $\mathrm{DPP}(\mathcal{K},\mu)$, with rate given by
    \begin{equation*}
        N_K(\delta,\eps) = \frac{A^2 B^2 m^2 M^2 \log\left( \frac{C}{\delta} \right)}{\eps^2},
    \end{equation*}
    where $A,B$ and $C$ are the same constants as in proposition~\ref{prop:cv_orthogonal_polynomials}, and 
    \begin{equation*}
        M = 2 \max_{i \in [m]} \max_{x \in \mathcal{X}} \vert \mathcal{P}_i(x) \vert.
    \end{equation*}
\end{thm}

The precise weak coherency rate $N_K$ is derived in section~\ref{sect:rates_cv_ope} of the supplementary material. Following remark~\ref{rem:threshold_wc_rates}, we note that this yields an error rate of $\eps_n = \mathcal{O}\left( m \sqrt{\frac{\log\left( \frac{1}{\delta} \right)}{n}}  \right)$.

\subsection{Better-than-independent guarantees}

For this example, weak coherency can actually be used to translate properties of $\mathrm{DPP}(\mathcal{K},\mu)$ to $\mathrm{DPP}(K_n,\mu_n)$, through the moment mapping theorem \ref{th:moment_mapping}. Indeed, the variance of the $1$-point linear statistics of $\mathrm{DPP}(\mathcal{K},\mu)$ has been studied by~\cite{bardenet2021determinantal}, and the moment mapping theorem allows to relate it to that of those associated to $\mathrm{DPP}(K_n,\mu_n)$. This variance turns out to be strictly smaller than for independent sampling and, as it controls the concentration of linear statistics towards their expectation, a lower variance translates to smaller coresets (see remark~\ref{rem:ope_coreset} below). Let us first recall a result of~\cite{bardenet2021determinantal}, concerning continuous multivariate orthogonal polynomial ensembles.

This result hinges on the remarkable analytical properties of $\mathcal{K}$, and is stated under the mild technical assumptions that $\varphi$ is regular enough and that $\mu$ is Nevai-class.\footnote{This is a common regularity condition: when the density $p$ of $\mu$ decomposes as $p(x_1,...,x_d) = p_1(x_1)...p_d(x_d)$, it for instance suffices that each $p_i$ be positive on $[-1,1]$; see,\emph{e.g.},~\cite{bardenet2020monte} for a more extensive discussion.}

\begin{thm}[\cite{bardenet2021determinantal}]
    \label{th:bardenet_ghosh}
    Suppose that we are under the setting of section~\ref{sect:def_ope} and that $m = p^d$ for some $p \in \N_{>0}$.\footnote{This last assumption can be relaxed, at the cost of a more technical statement.} Suppose further that $\mu$ is Nevai-class, and consider a bounded and measurable function $\widetilde{\varphi}: \mathcal{X} \rightarrow \R$ such that $\varphi: x \mapsto \frac{\widetilde{\varphi}(x)}{\mathcal{K}(x,x)}$ is Lipschitz continuous. Then,
    \begin{equation}
        \mathbf{Var}_{\mathcal{S} \sim \mathrm{DPP}(\mathcal{K},\mu)}\left[ \Lambda^{(\varphi)}(\mathcal{S}) \right] \in \mathcal{O}\left( \frac{1}{m^{1 + 1/d}} \right),
    \end{equation}
    where $\Lambda^{(\varphi)}(\mathcal{S})$ denotes the $1$-point linear statistic of $\mathrm{DPP}(\mathcal{K},\mu)$ with respect to $\varphi$.
\end{thm}

As a consequence of the moment mapping theorem (theorem~\ref{th:moment_mapping}), this translates to the following.

\begin{cor}
    \label{cor:variance_discrete_ope}
    Under the assumptions of theorem~\ref{th:bardenet_ghosh}, for any $\delta \in (0,1)$ and $\eps >0$ small enough, it holds with probability at least $1 - \delta$ that
    \begin{equation}
        \mathbf{Var}_{\mathcal{S} \sim \mathrm{DPP}(K_n,\mu_n)}\left[ \Lambda^{(\varphi)}(\mathcal{S}) \right] \leq \mathbf{Var}_{\mathcal{S} \sim \mathrm{DPP}(\mathcal{K},\mu)}\left[ \Lambda^{(\varphi)}(\mathcal{S}) \right] + \eps
    \end{equation}
    as soon as $n \geq N(\delta,\eps,\varphi)$, with $N$ as in equation~\eqref{eq:mean_condition} and for $N_K$ as in theorem~\ref{th:cv_polynomial_ensembles}.
\end{cor}
In particular, we can instantiate the rate $\eps_n = \mathcal{O}\left( m \sqrt{\frac{\log\left( \frac{1}{\delta} \right)}{n}}  \right)$.
For the sake of comparison, the corresponding variance for Poisson point processes over $X_n$ is in $\mathcal{O}\left( \frac{1}{m}\right) + \mathcal{O}_P\left( \frac{1}{\sqrt{n}} \right)$ (see, \emph{e.g.}, Appendix S1 in~\cite{bardenet2021determinantal}). This means that, for fixed $m$ large enough and as $n \rightarrow \infty$, the variance of linear statistics of $\mathrm{DPP}(K_n,\mu_n)$ is \emph{strictly} better than that for independent sampling.

\begin{rem}
    \label{rem:ope_coreset}
    A concrete consequence of having lower variance than independent sampling is that $\mathrm{DPP}(K_n,\mu_n)$ is \emph{provably} better than independent sampling when used to build coresets. In particular, this results in a faster decreasing rate of the difference between $L_\mathcal{S}$ and $L(\theta)$ (equation~\eqref{eq:coreset_statistics}) as $m$ increases, and we refer to~\cite{bardenet2024smallcoresetsnegativedependence} for a comprehensive description of this phenomenon. This theoretical result is further backed by extensive empirical evidence~\cite{tremblay2023extended,bardenet2024smallcoresetsnegativedependence}.
    In fact, as we discussed in the introduction, another DPP on $X_n$ satisfying a guarantee of this type has been proposed in~\cite{bardenet2021determinantal,bardenet2024smallcoresetsnegativedependence}, where it is also obtained by showing that the variances of its linear statistics approach those of $\mathrm{DPP}(\mathcal{K},\mu)$.\footnote{It is likely that a stronger weak-consistency result could be obtained for this process as well.} The main drawback of this second example is that the construction of the kernel of the process is much more involved, and we discuss the differences between the guarantees for both approaches in section~\ref{sect:comparaison_bardenet} of the supplementary material. In terms of practical performance, the empirical results in~\cite{bardenet2024smallcoresetsnegativedependence} suggest that the discrete orthogonal polynomial ensemble actually performs better.
\end{rem}

\subsubsection{Illustration: discrete orthogonal polynomial ensemble v.s. iid for coresets}

\begin{wrapfigure}{r}{0.5\linewidth}
    \centering
    \scalebox{0.4}{

\begin{tikzpicture}[/tikz/background rectangle/.style={fill={rgb,1:red,1.0;green,1.0;blue,1.0}, fill opacity={1.0}, draw opacity={1.0}}, show background rectangle]
\begin{axis}[point meta max={nan}, point meta min={nan}, legend cell align={left}, legend columns={1}, title={}, title style={at={{(0.5,1)}}, anchor={south}, font={{\fontsize{14 pt}{18.2 pt}\selectfont}}, color={rgb,1:red,0.0;green,0.0;blue,0.0}, draw opacity={1.0}, rotate={0.0}, align={center}}, legend style={color={rgb,1:red,0.0;green,0.0;blue,0.0}, draw opacity={1.0}, line width={1}, solid, fill={rgb,1:red,1.0;green,1.0;blue,1.0}, fill opacity={1.0}, text opacity={1.0}, font={{\fontsize{24 pt}{10.4 pt}\selectfont}}, text={rgb,1:red,0.0;green,0.0;blue,0.0}, cells={anchor={center}}, at={(0.98, 0.98)}, anchor={north east}}, axis background/.style={fill={rgb,1:red,1.0;green,1.0;blue,1.0}, opacity={1.0}}, anchor={north west}, xshift={1.0mm}, yshift={-1.0mm}, width={150.4mm}, height={99.6mm}, scaled x ticks={false}, xlabel={$m$}, x tick style={color={rgb,1:red,0.0;green,0.0;blue,0.0}, opacity={1.0}}, x tick label style={color={rgb,1:red,0.0;green,0.0;blue,0.0}, opacity={1.0}, rotate={0}}, xlabel style={at={(ticklabel cs:0.5)}, anchor=near ticklabel, at={{(ticklabel cs:0.5)}}, anchor={near ticklabel}, font={{\fontsize{24 pt}{14.3 pt}\selectfont}}, color={rgb,1:red,0.0;green,0.0;blue,0.0}, draw opacity={1.0}, rotate={0.0}}, xmode={log}, log basis x={10}, xmajorgrids={true}, xmin={0.8467453123625269}, xmax={302.33412132595987}, xticklabels={{$10^{0}$,$10^{1}$,$10^{2}$}}, xtick={{1.0,10.0,100.0}}, xtick align={inside}, xticklabel style={font={{\fontsize{24 pt}{10.4 pt}\selectfont}}, color={rgb,1:red,0.0;green,0.0;blue,0.0}, draw opacity={1.0}, rotate={0.0}}, x grid style={color={rgb,1:red,0.0;green,0.0;blue,0.0}, draw opacity={0.1}, line width={2}, solid}, axis x line*={left}, x axis line style={color={rgb,1:red,0.0;green,0.0;blue,0.0}, draw opacity={1.0}, line width={1}, solid}, scaled y ticks={false}, ylabel={90\%-quantile relative error}, y tick style={color={rgb,1:red,0.0;green,0.0;blue,0.0}, opacity={1.0}}, y tick label style={color={rgb,1:red,0.0;green,0.0;blue,0.0}, opacity={1.0}, rotate={0}}, ylabel style={at={(ticklabel cs:0.5)}, anchor=near ticklabel, at={{(ticklabel cs:0.5)}}, anchor={near ticklabel}, font={{\fontsize{24 pt}{14.3 pt}\selectfont}}, color={rgb,1:red,0.0;green,0.0;blue,0.0}, draw opacity={1.0}, rotate={0.0}}, ymode={log}, log basis y={10}, ymajorgrids={true}, ymin={0.03240546100098253}, ymax={0.6804562481584964}, yticklabels={{$10^{-1.0}$,$10^{-0.5}$}}, ytick={{0.1,0.31622776601683794}}, ytick align={inside}, yticklabel style={font={{\fontsize{24 pt}{10.4 pt}\selectfont}}, color={rgb,1:red,0.0;green,0.0;blue,0.0}, draw opacity={1.0}, rotate={0.0}}, y grid style={color={rgb,1:red,0.0;green,0.0;blue,0.0}, draw opacity={0.1}, line width={2}, solid}, axis y line*={left}, y axis line style={color={rgb,1:red,0.0;green,0.0;blue,0.0}, draw opacity={1.0}, line width={1}, solid}, colorbar={false}]
    \addplot[color={rgb,1:red,0.0;green,0.6056;blue,0.9787}, name path={5}, draw opacity={1.0}, line width={2}, solid, mark={diamond*}, mark size={3.75 pt}, mark repeat={1}, mark options={color={rgb,1:red,0.0;green,0.6056;blue,0.9787}, draw opacity={1.0}, fill={rgb,1:red,0.0;green,0.6056;blue,0.9787}, fill opacity={1.0}, line width={0.75}, rotate={0}, solid}]
        table[row sep={\\}]
        {
            \\
            1.0  0.5130929454532965  \\
            2.0  0.3373979407314181  \\
            4.0  0.2377513041105987  \\
            8.0  0.15785881978473965  \\
            16.0  0.10853753529856842  \\
            32.0  0.07645835314384723  \\
            64.0  0.06316809301967229  \\
            128.0  0.05610614528789941  \\
            256.0  0.03532144395974686  \\
        }
        ;
    \addlegendentry {$\mathrm{DPP}(K_n,\mu_n)$}
    \addplot[color={rgb,1:red,0.8889;green,0.4356;blue,0.2781}, name path={6}, draw opacity={1.0}, line width={2}, solid, mark={*}, mark size={3.75 pt}, mark repeat={1}, mark options={color={rgb,1:red,0.8889;green,0.4356;blue,0.2781}, draw opacity={1.0}, fill={rgb,1:red,0.8889;green,0.4356;blue,0.2781}, fill opacity={1.0}, line width={0.75}, rotate={0}, solid}]
        table[row sep={\\}]
        {
            \\
            1.0  0.6242807750924425  \\
            2.0  0.5043167268051353  \\
            4.0  0.37212012134418204  \\
            8.0  0.29137783000412215  \\
            16.0  0.2204668872110536  \\
            32.0  0.1764516598001598  \\
            64.0  0.1450875813627058  \\
            128.0  0.12738090238825234  \\
            256.0  0.11221536450803268  \\
        }
        ;
    \addlegendentry {iid}
\end{axis}
\end{tikzpicture}}
    \caption{\centering $90\%$-quantile relative error as a function of $m$ in log-log scale.}
    \label{fig:coreset}
\end{wrapfigure}
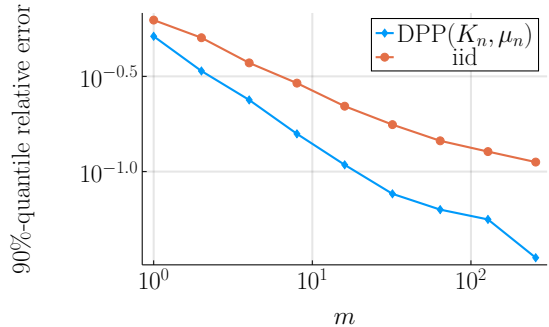

Following remark~\ref{rem:ope_coreset}, we illustrate the performance of $\mathrm{DPP}(K_n,\mu_n)$ for coreset construction on a controlled example, as compared to iid sampling. For more extensive evaluation, we refer to~\cite{tremblay2023extended,bardenet2024smallcoresetsnegativedependence}. We draw $n$ points $X_n = \{x_1,...,x_n\}$ uniformly in $[-1,1]^2$, and consider the $1$-means problem, \emph{i.e.},
\begin{equation*}
    \argmin_{\theta \in \R^2} \underbrace{\sum_{i=1}^n \Vert x_i - \theta \Vert^2}_{L(\theta)}.
\end{equation*}
It is easy to compute the optimal solution for this problem. For coresets, we stress that our goal is rather to build a subset $\mathcal{S} \subseteq X_n$ and an estimator $L_\mathcal{S}$ of $L$ such that $L(\theta)$ and $L_\mathcal{S}(\theta)$ are close \emph{for all values of $\theta$}. 
We are going to compare the performance of the two following estimators of $L(\theta)$:
\begin{equation*}
    L_\mathcal{S}^{\mathrm{iid}}(\theta) = \sum_{i = 1}^{n} \frac{\Vert x_i - \theta \Vert^2 \epsilon_i}{m p_i}, \quad L_\mathcal{S}^{\mathrm{DPP}}(\theta) = \sum_{i = 1}^{n} \frac{\Vert x_i - \theta \Vert^2 \epsilon_i}{[K_n]_{i,i} / n}.
\end{equation*}
For the iid estimate, a subsample of $m$ points $\mathcal{S}$ in $X_n$ is drawn iid with replacement according to the following sensitivity scores~\cite{tremblay2019determinantal}:
\begin{equation*}
    p_i \propto \frac{1}{n} \left( 1 + \frac{\Vert x_i \Vert^2}{v} \right), \ \ v = \frac{1}{n} \sum_{i = 1}^n \Vert x_i \Vert^2.
\end{equation*}
This is the optimal estimate for iid data~\cite{langberg2010universal,bachem1703practical}. For the DPP-based estimate, $\mathcal{S}$ is drawn from $\mathrm{DPP}(K_n,\mu_n)$, and the $[K_n]_{i,i}/n$ denominator is the probability that $x_i$ belongs to $\mathcal{S}$. In both cases, $\epsilon_i$ denotes the number of occurrences of $x_i$ in the sample (which is always in $\{0,1\}$ for the DPP samples, but may be larger for iid since we draw with replacement). That $L_\mathcal{S}^{\mathrm{iid}}(\theta)$ is an unbiased estimate of $L(\theta)$ is classical~\cite{bachem1703practical}, and that $L_\mathcal{S}^{\mathrm{DPP}}(\theta)$ is is obtained from inspecting equations~\eqref{eq:correlation} and~\eqref{eq:det}. For our Julia implementation,\footnote{Available at \url{https://gitlab.irisa.fr/hjaquard/discrete-to-continuous-dpps}.} we rely on the exact, fast DPP sampler from~\cite{barthelme2023faster}.\footnote{As implemented in~\url{https://github.com/dahtah/Determinantal.jl}.}

For a given $\mathcal{S}$ and estimator $L_\mathcal{S}$, we compute the worst values of the relative error $\frac{\vert L_\mathcal{S}(\theta) - L(\theta) \vert}{L(\theta)}$, as obtained from $n_\theta = 100$ random values of $\theta$ selected uniformly in $[-1,1]^d$. This operation is repeated for $n_S = 100$ draws of $\mathcal{S}$. We then compute the $90\%$-quantile of the relative error over those $n_\mathcal{S}$ draws; that is, the value $\epsilon$ such that $90\%$ of the sampled coresets $\mathcal{S}$ result in a relative error of at most $\epsilon$. We report these $90\%$ quantiles for a range of sample-sizes $m = \vert \mathcal{S} \vert \in \{1,2,4,8,16,32,64,128,256\}$ in Figure~\ref{fig:coreset}, as averaged over $100$ realizations of $X_n$. Sampling from $\mathrm{DPP}(K_n,\mu_n)$ indeed results in faster rates.

\section{Repulsive sampling on unknown manifolds through graphs}
\label{sect:harmonic_ensemble}

Our second example pertains to random sampling on manifolds using DPPs, and in particular to the so-called the \emph{harmonic ensemble}.  For some families of manifolds, strict better-than-independent guarantees have been derived for this process (\emph{e.g.}~\cite{levi2024linear,borda2024riesz} and references therein) and, as one may expect, its definition is quite involved and hinges on geometric objects associated to the manifold $\mathcal{X}$. We introduce in this section a DPP over points $X_n$ sampled from $\mathcal{X}$ and show that it is weakly coherent with the harmonic ensemble, allowing for provable repulsive sampling on $\mathcal{X}$ \emph{even when both this manifold and the sampling density are unknown}, as long as $n$ is sufficiently large. Our result is simply based on an application of theorem~\ref{th:detailed} which, together with recent advances on graph-manifold approximations~\cite{dunson2021spectral} and density estimation on manifolds~\cite{wu2022strong}, abstracts away most of the difficulty.

\subsection{Harmonic ensembles on graphs and manifolds}
\label{sect:setting_harmonic}

For the remainder of this section, we consider $\mathcal{X}$ a smooth, connected, compact and orientable $d_{\mathcal{X}}$-dimensional Riemannian submanifold of $\R^d$ without boundary.\footnote{Note that, when $\mathcal{X}$ is endowed with its usual topology, it is a second-countable locally compact Hausdorff space.} 
All definitions and properties that we use on Riemannian manifolds can be found in, \emph{e.g.}, the classical textbook
~\cite{gallot_riemannian}.

\subsubsection{Harmonic ensemble on a manifold}
    Compact, connected and orientable Riemannian manifolds admit a canonical measure $\omega$ defined by their volume form, with respect to which the harmonic ensemble is defined. In addition, those manifolds admit a so-called \emph{Laplace-Beltrami operator} $\mathcal{L}$, which is defined for any smooth function $f: \mathcal{X} \rightarrow \R$ by
    \begin{equation}
        \mathcal{L}f = \mathrm{div}(\nabla f),
    \end{equation}
    where $\mathrm{div}$ and $\nabla$ denote respectively the (Riemannian) divergence and gradient. The operator $\mathcal{L}$ admits a pure-point spectrum and, as it is further non-positive definite, this implies: 1/ that the eigenvalues $\lambda_i$ of $- \mathcal{L}$ form a discrete sequence that we can order as
    \begin{equation}
        0 \leq \lambda_1 \leq \lambda_2 \leq ...
    \end{equation}
    and, 2/ that we can define an orthonormal basis of $L^2(\mathcal{X};\omega)$ made up of \emph{smooth} normalized eigenfunctions $\phi_i \in L^2(\mathcal{X};\omega)$ associated to these eigenvalues. For simplicity of exposition, \emph{we assume that all of the eigenvalues of $\mathcal{L}$ are simple}: while this hypothesis is (at least in theory) restrictive, it can be lifted by considering eigenprojectors in place of eigenfunctions in the following.\footnote{A further motivation for this (common) restriction is that we rely on results stated under this same assumption~\cite{dunson2021spectral}. As explained in~\cite{dunson2021spectral}, those results can be generalized as well.} All of our results would generalize to this case as well. For any $m \in \N_{>0}$, we can then define the \emph{harmonic ensemble} of order $m$ over $\mathcal{X}$, which is the DPP defined with respect to the measure $\omega$ by the kernel
    \begin{equation}
        \widetilde{\mathcal{K}}(x,y) = \sum_{i = 1}^m \phi_i(x)\phi_i(y) \ \ \forall x,y \in \mathcal{X}.
    \end{equation}
    This indeed defines a DPP by virtue of the Macchi-Soshnikov theorem, and is the usual way the harmonic ensemble is introduced. For our purpose, it is more convenient to define it in a slightly different way. Let $\mu$ be a probability measure that is absolutely continuous with respect to $\omega$, with positive density $p$. Then, the harmonic ensemble is equivalently defined as the DPP with respect to the measure $\mu$ and kernel
    \begin{equation}
        {\mathcal{K}}(x,y) = \frac{1}{\sqrt{p(x)}} \widetilde{\mathcal{K}}(x,y) \frac{1}{\sqrt{p(y)}}.
    \end{equation}
    That this defines the same DPP is readily obtained by inspecting equations~\eqref{eq:correlation} and~\eqref{eq:det}, and discussed in section~\ref{sect:app_harmonic_ensemble} of the supplementary material.

Going forward, we proceed under the following technical assumption.

\begin{asmp}
    \label{asmp:density_regularity}
    The measure $\mu$ is absolutely continuous with respect to the measure $\omega$, with density $p: \mathcal{X} \rightarrow [p_{min},p_{max}]$ for some \emph{positive} $p_{min} > 0$ that is Hölder of order $\kappa \in (0,1]$. 
\end{asmp}

Let us now describe a DPP on a set $X_n$ sampled from $\mathcal{X}$ for which we show below that it is weakly coherent with the harmonic ensemble on $\mathcal{X}$. Its construction is based on graphs, that are commonly employed as discrete approximations of manifolds~\cite{belkin2006convergence},

\subsubsection{Harmonic ensemble associated to a graph sampled from a manifold}
    We consider a probability measure $\mu$ over $\mathcal{X}$ that is absolutely continuous with respect to $\omega$, and $X_n = \{x_1,...,x_n\}$ sampled iid from $\mu$. Then, define the complete graph $G_n$ with nodes $X_n$ and edges $E_n = \{\{x,y\} \ ; x,y \in X_n \}$, with weights 
    \begin{equation}
        w_n(x,y) = \exp\left( \frac{- \Vert x - y \Vert^2_{\R^d}}{4 h_1(n)^2} \right)
    \end{equation}
    for any $x,y \in X_n \subseteq \mathcal{X}$ and some positive $h_1(n)$, and where $\Vert . \Vert_{\R^d}$ denotes the euclidean norm in $\R^d$. Here, $h_1(n)$ is a bandwidth parameter, that should scale appropriately with $n$. A natural idea to approximate $\mathcal{K}$ is be to build an approximation of $- \mathcal{L}$ using the \emph{graph Laplacian} of $G_n = (X_n,E_n,w_n)$ and compute its first $m$ eigenvectors to build a kernel $K_n$, and then consider the DPP associated with $\mu_n$. It turns out that there are two issues with this construction. 
    \begin{enumerate}
    \item As one may expect, and this would indeed be the case, the graph Laplacian and its eigenvectors converge to objects that depend on $\mu$ as $n \rightarrow \infty$. On the other hand, a remarkable line of work initiated by~\cite{coifman2006diffusion} showed that a proper renormalization of the graph Laplacian allows to asymptotically recover the operator $\mathcal{L}$ (instead of a density-dependent variant). In our case, the construction is due to~\cite{dunson2021spectral}, and goes as follows. First, we consider a $n \times n$ normalized weighted adjacency matrix
    \begin{equation}
        W_{i,j} = \frac{w_n(x_i,x_j)}{d_n(x_i)d_n(x_j)}
    \end{equation}
    where $d_n(x) = \sum_{i=1}^n w_n(x,x_i)$, and define its associated  diagonal degree matrix with entries $D_{i,i} = \sum_{j = 1}^n W_{i,j}$. We then define the normalized Laplacian
    \begin{equation*}
        L_n = \frac{I - D^{-1}W}{h_1(n)^2},
    \end{equation*}
    which is semi-definite positive, and denote by $(\widetilde{u_i})_{i=1}^m$ its first $m$ eigenvectors associated to its smallest eigenvalues.\footnote{Notice that $L_n$ is not symmetric, but is similar to a symmetric matrix, so that it shares the same spectrum: $L_n = D^{-1/2} \left( \frac{I - D^{-1/2} W D^{-1/2}}{h_1(n)^2} \right) D^{1/2}$. This is the reason we speak of it having eigenvalues and being semi-definite positive.} We further consider the family of re-weighted eigenvectors given by
    \begin{equation}
        {u_i}(x) = \frac{\widetilde{u_i}(x)^2}{\sqrt{\frac{\vert \mathbf{S}^{d-1} \vert h_1(n)^d}{d} \sum_{j = 1}^n \frac{\widetilde{u_i}(x_j)}{\#(B(x_i,h_1(n)) \cap X_n)}}},
        \label{eq:renormalization_harmonic}
    \end{equation}
    where $\vert \mathbf{S}^{d-1} \vert$ denotes the volume of the $(d-1)$-dimensional Euclidean sphere, and we denote by $\#(B(x_i,h_1(n)) \cap X_n)$ the number of datapoints in the $h_1(n)$-Euclidean neighborhood of $x_i$.
    The result of~\cite{dunson2021spectral} then states that there exists some constant $h_1(n) >0$ such that, if $h_1(n) \leq H_1$ is small enough and decreases slowly enough with $n$ that $h_1(n) \geq \left( \frac{\log(n)}{n} \right)^{\frac{1}{4 d_\mathcal{X} + 8}}$, the $u_i$'s concentrate towards the eigenfunctions $\phi_i$ of $\mathcal{L}$ in the $L^\infty$ sense over $X_n$. We recall their precise result in section~\ref{sect:app_harmonic_ensemble} of the supplementary material.
    \item The kernel $\mathcal{K}$ of the harmonic ensemble over $\mathcal{X}$ is defined using the density $p$, which is unknown. This second issue can be taken care of by introducing a density estimator~\cite{pelletier2005kernel}: for $h_2(n) > 0$, we then consider a density estimator $e[p]:\mathcal{X} \rightarrow \R$ of $p$ defined by 
    \begin{equation}
        \label{eq:kernel_density_estimator}
        e[p](x) = \frac{1}{n h_2(n)^{d_\mathcal{X}}} \sum_{i = 1}^n l\left( \frac{\Vert x_i - x \Vert_{\R^d}}{h_2(n)} \right),
    \end{equation}
    for some function $l: \R_{\geq 0} \rightarrow \R$ satisfying mild integrability and vanishing at infinity-conditions (see assumption~\ref{asmp:kernel_density} below). Here, $h_2(n)$ is a second bandwidth parameter, that should scale appropriately with $n$. In particular, it was shown in~\cite{wu2022strong} that there exists a constant $h_2(n) > 0$ such that, if $h_2(n) \leq H_2$ and $h_2(n)$ decreases slowly enough with $n$ that $\left( \frac{\log(n)}{nh_2(n)^{d_\mathcal{X}}}^{\kappa/2} \right)$ goes to $0$, the density estimator $e[p]$ approaches $p$ in the $L^\infty$ sense as $n \rightarrow \infty$. We recall their precise result in section~\ref{sect:app_harmonic_ensemble} of the supplementary material.
    \end{enumerate}

Both constants $H_1$ and $H_2$ depend on the probability density $p$ and the geometry of $\mathcal{X}$. Concerning $l$, we more specifically assume the following.

\begin{asmp}
    \label{asmp:kernel_density}
    The function $l$ in equation~\eqref{eq:kernel_density_estimator} should be defined such that:
    \begin{enumerate}
        \item it is be bounded on $\R_{\geq 0}$;
        \item it is Riemann integrable on any compact of the form $[0,\alpha]$ for $\alpha > 0$;
        \item there exists a $t_0$ such that for any $t \geq t_0$ we have $l(t) \leq \frac{1}{t^\beta}$ for $\beta > d_{\mathcal{X}}$;
        \item it is such that $\int_{\R^{d_\mathcal{X}}} l(\Vert v \Vert_{\R^{d_\mathcal{X}}}) dv = 1$.
    \end{enumerate}
\end{asmp}
    
    We are almost done with the construction, and need to take one last precautionary step to ensure that the DPP we define exists. Consider the measure $\omega_n = \frac{1}{n} \sum_{x \in X_n} \frac{1}{e[p](x)} \delta_x$, which is an approximation of $\omega$ on $X_n$, and the vectors $v_1,...,v_m \in \R^n$ obtained by applying the Gram-Schmidt orthogonalization process to the vectors ${u_i}$ with respect to the inner product $\langle f,g \rangle_{\omega_n} = \sum_{x \in X_n} \omega_n(x) f(x) g(x)$. We finally define the kernel
    \begin{equation}
        [K_n]_{a,b} = \frac{1}{\sqrt{{e[p](x_a)}}} \left( \sum_{i = 1}^m v_i(x_a) v_i(x_b) \right) \frac{1}{\sqrt{{e[p](x_b)}}} \ \ \forall a,b \in [n],
    \end{equation}
    and call the DPP defined by $K_n$ with respect to the measure $\mu_n$ the \emph{discrete harmonic ensemble}.\footnote{Note that parts of this construction could likely be modified without impacting our later results. For instance, the point-wise re-normalization in equation~\eqref{eq:renormalization_harmonic} serves as a density estimator on $X_n$, which could be replaced with the kernel density estimator $e[p]$, for which the results of~\cite{dunson2021spectral} could likely be adapted. One could also consider a sparser $\eps$-graph or $k$-nearest-neighbors graph but, to the best of our knowledge, no variant of the result of~\cite{dunson2021spectral} have appeared in the existing literature for those cases.} For clarity, the construction of $K_n$ is summarized in algorithm~\ref{alg:he}. Let us stress that, unlike the kernel of the multivariate orthogonal polynomial ensemble, the kernel here is \textbf{parametrized by additional bandwidth parameters} $h_1(n)$ and $h_2(n)$. In particular, to claim the weak coherency of a sequence of DPPs, we always specify a corresponding sequence of bandwidth parameters.

\begin{algorithm}[t]
\caption{Kernel of the discrete harmonic ensemble}
\label{alg:he}
\begin{algorithmic}[1]
\Require $m \geq 0$, $X_n = \{x_1,...,x_n\}$, $h_1(n),h_2(n) > 0$, $l$ a kernel density function
\State Build the complete graph with weights $w_n(x_i,x_j) = \exp\left( \frac{- \Vert x_i - x_j \Vert^2_{\R^d}}{4 h_1(n)^2} \right) \ \ \forall i,j \in [n]$
\State Compute the degrees $d_n(x) = \sum_{i=1}^n w_n(x,x_i) \ \ \forall x \in X_n$
\State Compute the surrogate normalized adjacency matrix $W_{i,j} = \frac{w_n(x_i,x_j)}{d_n(x_i)d_n(x_j)}$ 
\State Compute its normalized Laplacian $L_n = \frac{I - D^{-1}W}{h_1(n)^2}$ 
\State Compute the first $k$ eigenvectors $u_i$ of $L_n$
\State Compute the re-normalized vectors ${u_i}(x) = \frac{\widetilde{u_i}(x)^2}{\sqrt{\frac{\vert \mathbf{S}^{d-1} \vert h_1(n)^d}{d} \sum_{j = 1}^n \frac{\widetilde{u_i}(x_j)}{\#(B(x_i,\eps_n) \cap X_n)}}} \ \ \forall x \in X_n$
\State Compute the kernel density estimator $e[p](x) = \frac{1}{n h_2(n)^{d_\mathcal{X}}} \sum_{i = 1}^n l\left( \frac{\Vert x_i - x \Vert_{\R^d}}{h_2(n)} \right) \ \ \forall x \in X_n$
\State Compute the $v_i$'s by orthonormalizing the $u_i$'s with respect to $\omega_n = \frac{1}{n} \sum_{x \in X_n} \frac{1}{e[p](x)} \delta_x$
\For{$1 \leq a,b \leq n$}
\State $[K_n]_{a,b} \gets \frac{1}{\sqrt{e[p](x_a) e[p](x_b)}} \sum_{i = 1}^k v_i(x_a) v_i(x_b)$
\EndFor
\end{algorithmic}
\end{algorithm}

\subsection{Weak coherency of harmonic ensembles}

We are going to show that, when $h_1(n)$ and $h_2(n)$ scale appropriately, the discrete harmonic ensemble is weakly coherent with the harmonic ensemble on $\mathcal{X}$. The following preliminary result is proved in section~\ref{sect:app_harmonic_ensemble} of the supplementary material.

\begin{prop}
    \label{prop:harmonic_noyaux_auxiliaires}
    Consider the kernel $\widetilde{K_n}$ defined by
    \begin{equation*}
        \left[\widetilde{K_n}\right]_{a,b} = \sum_{i = 1}^m v_i(x_a) v_i(x_b) \ \ \forall a,b \in [n].
    \end{equation*}
    Then, $\mathrm{DPP}(K_n,\mu_n)$ and $\mathrm{DPP}(\widetilde{K_n},\omega_n)$ define the same DPP. 
    Let 
    \begin{equation}
        \label{eq:alpha_beta}
        e = \widetilde{\eps} + \alpha + \beta, \ \ \text{for} \ \ \alpha = h_1(n)^{1/2} \ \ \text{and} \ \ \beta =\left(\frac{\log(n)}{n h_2(n)^{d_\mathcal{X}}}\right)^{\kappa/2},
    \end{equation}
    and assume that $h_1(n) \leq H_1$ and $h_1(n) \geq \left( \frac{\log(n)}{n} \right)^{\frac{1}{4 d_\mathcal{X} + 8}}$, $\alpha < 1$, $h_2(n) \leq H_2$ and $\beta < \min\left(1, \frac{p_{min}}{2}\right)$.
    Assume further that assumptions~\ref{asmp:density_regularity} and~\ref{asmp:kernel_density} are satisfied.
    Then, there exist $\widetilde{A},\widetilde{B},\widetilde{C} > 0$ such that, for any $\widetilde{\eps} >0$ and with probability at least $1 - \widetilde{C} \left( \frac{1}{n^2} + \exp\left( \frac{-2 \widetilde{\eps}^2 n}{\widetilde{B}^2} \right) \right)$,
    \begin{equation*}
        \max_{a,b \in [n]} \left\vert \left[\widetilde{K_n}\right]_{a,b} - \widetilde{\mathcal{K}}(x_a,x_b) \right\vert \leq m \widetilde{A} e \left( \widetilde{A} e + \widetilde{M} \right),
    \end{equation*}
    where
    \begin{equation*}
        \widetilde{M} = \max_{i \in [m]}\left( 2 \max_{x \in \mathcal{X}} \vert \phi_i(x) \vert\right),
    \end{equation*}
    The constants $\widetilde{A}$ and $\widetilde{B}$ depend on the density $p$ and the geometry of $\mathcal{X}$, and $\widetilde{C}$ on $m$.
\end{prop}

The probabilistic bound on the kernel-difference relies on the results of~\cite{dunson2021spectral} and~\cite{wu2022strong} and, when $h_2(n)$ decays slowly enough with $n$, establishes their concentration. Let us comment on two aspects of proposition~\ref{prop:harmonic_noyaux_auxiliaires}.
\begin{enumerate}
    \item Since $\mathrm{DPP}(\widetilde{K_n},\omega_n)$ exists by the Macchi-Soshnikov theorem, so does $\mathrm{DPP}(K_n,\mu_n)$.
    \item The additional hypotheses on $\alpha$ and $\beta$ further constraint the choice of bandwidths parameters $h_1(n)$ and $h_2(n)$, on top of the conditions from~\cite{dunson2021spectral} and~\cite{wu2022strong}.
\end{enumerate}
As the concentration bound in proposition~\ref{prop:harmonic_noyaux_auxiliaires} depends on $e$, hence on both $h_1(n)$ and $h_2(n)$, it is then necessary to specify admissible sequences of bandwidth parameters in order to ensure that $e$ decreases appropriately with $n$.

\begin{definition}[Admissible bandwidths]
    Two sequences of bandwidths $h_1(n) \leq H_1$ and $h_2(n)~\leq~H_2$ are $C$-\emph{admissible} if they satisfy that, for any $\eps > 0$, there exists $N_h(\eps,C)$ such that
    \begin{equation}
        \label{eq:hypothesis_bandwidths}
        n \geq N_h(\eps,C) ~ \Rightarrow ~ \begin{cases}
            \min\left( h_1(n)^{1/2}, \left(\frac{\log(n)}{n h_2(n)^{d_\mathcal{X}}}\right)^{\kappa/2}\right) \leq \max\left(\frac{\eps}{C} \right), \\
             h_1(n)^{1/2} < 1/3, \\
             \left(\frac{\log(n)}{n h_2(n)^{d_\mathcal{X}}}\right)^{\kappa/2} < \min\left( \frac{1}{3}, \frac{p_{min}}{2} \right).
        \end{cases}
    \end{equation}
\end{definition}
In particular, the topmost condition in the rhs of equation~\eqref{eq:hypothesis_bandwidths} ensures that $\alpha + \beta$ goes to $0$ as $n \rightarrow \infty$. The two bottom conditions are stricter variants of the hypotheses of proposition~\ref{prop:harmonic_noyaux_auxiliaires}.

Proposition~\ref{prop:harmonic_noyaux_auxiliaires} can then be used to to show the concentration of $K_n$ towards $\mathcal{K}$ for admissible bandwidth parameters. The constants below are obtained from (very) rough bounds, and we refer to the proof in section~\ref{sect:proof_harmonic_theorem} of the supplementary material for a more precise concentration result.

\begin{thm}[Weak coherency of harmonic ensembles]
    \label{th:harmonic_coherency}
    Suppose that assumptions~\ref{asmp:density_regularity} and~\ref{asmp:kernel_density} are satisfied. Then, there exists some $A > 0$ such that, for $3mA$-admissible bandwidth parameters $h_1(n)$ and $h_2(n)$, $\mathrm{DPP}(K_n,\mu_n)$ is weakly coherent with $\mathrm{DPP}(\mathcal{K},\mu)$, with weak-coherency rate given by
    \begin{equation}
        N_K(\delta,\eps) = \max\left( \sqrt{\frac{2\widetilde{C}^2}{\delta}}, \frac{9 A^2 \widetilde{B}^2 m^2 \log\left( \frac{2\widetilde{C}}{\delta}\right)}{\eps^2},  N_h(\eps,C) \right).
    \end{equation}
    The constant $A$ depends on $\widetilde{A}$ (hence, on the geometry of $\mathcal{X}$), the density $p$ and the kernel $\mathcal{K}$.
\end{thm}

\begin{rem}
\label{rem:better_than_iid_harmonic}
Quantitative better-than-independent results for the harmonic ensemble $\mathrm{DPP}(\mathcal{K},\mu)$ on some families of manifolds $\mathcal{X}$ have been obtained, but are quite technical, and we briefly discuss in section~\ref{sect:discussion} how some of them could be translated to the discrete harmonic ensemble $\mathrm{DPP}(K_n,\mu_n)$. In place of those quantitative results, we emphasize a simple methodological takeaway. On the one hand, $\mathrm{DPP}(K_n,\mu_n)$ is a DPP on a finite set of points $X_n$ sampled from an \emph{unknown} manifold $\mathcal{X}$ according to an \emph{unknown} probability distribution $\mu$, that is weakly coherent with $\mathrm{DPP}(\mathcal{K},\mu)$, and can easily be sampled from using standard discrete-DPP sampling algorithms. On the other hand, even building the kernel of $\mathrm{DPP}(\mathcal{K},\mu)$ requires knowledge of the eigenstructure of the Laplace-Beltrami operator $\mathcal{L}$, which is unrealistic in most practical scenarii. For applications relying on, \emph{e.g.}, Monte-Carlo integration on manifolds, $\mathrm{DPP}(K_n,\mu_n)$ provides a feasible and likely provably-advantageous alternative to independent sampling. A deeper study of these guarantees is an important path for future work.
\end{rem}

\subsubsection{Illustration: discrete harmonic ensemble v.s. iid for Monte-Carlo on $\mathcal{S}^2$}
\label{sect:harmonic_monte-carlo}

To substantiate remark~\ref{rem:better_than_iid_harmonic}, we illustrate the behavior of $\mathrm{DPP}(K_n,\mu_n)$ for Monte-Carlo integration on the sphere $\mathcal{S}^2 \subseteq \R^3$, as compared to iid sampling. Our goal is to estimate
\begin{equation*}
    I := \int_{\mathcal{S}^2} f d\omega \ \left(= \frac{4 \pi}{3}\right), \ \ \text{where} \ \ f(x,y,z) = z^2,
\end{equation*}
from the knowledge only of a sample of points $X_n = \{x_1,...,x_n\}$ drawn independently in $\mathcal{S}^2$ according to some probability measure $\mu$ with density $p$. In practice, we draw the $x_i$'s uniformly, \emph{but do not rely on this information when estimating $I$}. Similarly, we do not assume knowledge of the manifold $\mathcal{X} = \mathcal{S}^2$. We stress that this setting is crucially different from the usual numerical integration setup on manifolds, which typically assumes knowledge of the manifold; for a reference on Monte-Carlo integration in this setting, see, \emph{e.g.},~\cite{ehler2019optimal}, and~\cite{levi2024linear,lemoine2024monte} for DPP-based techniques. To proceed, we are going to compute Monte-Carlo estimates of
\begin{equation*}
    I_n := \int_{\mathcal{S}^2} f d\omega_n,
\end{equation*}
where we recall that $\omega_n$ is the approximation of $\omega$ computed from $e[p]$. We stress that, \emph{for a given $X_n$}, $I_n$ is a fixed, deterministic quantity. Under the setting of proposition~\ref{prop:harmonic_noyaux_auxiliaires}, $\vert I - I_n \vert \leq A_f \left( \widetilde{\eps} + \beta \right)$ with probability at least $C_f \left( \frac{1}{n^2} + \exp\left( \frac{2- \widetilde{\eps}^2 n}{ B_f^2} \right) \right)$ for some constants $A_f,B_f,C_f > 0$ (we refer to the proof of proposition~\ref{prop:harmonic_noyaux_auxiliaires} for more details on this phenomenon), so that estimates of $I_n$ indeed leads to estimates of $I$.

In our experiments, we take $n = 3000$ points. The density estimate $e[p]$ is constructed using the normalized indicator function $l = \frac{\mathbf{1}_{[0,1]}}{\pi}$ and $h_2(n) = \left( \frac{\log(n)}{n} \right)^{1/4}$. To compute the approximate eigenfunctions $v_i$, we take $h_1(n) = \left( \frac{\log(n)}{n} \right)^{1/16}$. 

As a baseline, we consider a simple strategy where the $m$ points of a subsample $\mathcal{S}$ of $X_n$ are drawn independently and with replacement, with probability $p_i$ of sampling $x_i$ proportional to $e[p](x_i)$, and the estimator
\begin{equation}
    I_\mathcal{S}^{\mathrm{iid}} = \sum_{i = 1}^n \frac{f(x_i) \epsilon_i}{n m p_i e[p](x_i)}.
\end{equation}
Noting that $\mathbf{E}_\mathcal{S}(\epsilon_i) = m p_i$, it is obtained from a straightforward calculation that this is an unbiased estimator of $I_n$. We compare this baseline strategy with a DPP-based strategy, for which the points $X_n$ are sampled from the discrete harmonic ensemble. The corresponding estimator is defined as
\begin{equation}
    I_\mathcal{S}^{\mathrm{DPP}} = \sum_{i = 1}^n \frac{f(x_i) \epsilon_i}{n e[p](x_i) [K_n]_{i,i} / n},
\end{equation}
where we keep the $n/n$ factor in the denominator to highlight the similarity with the definition of $I_\mathcal{S}^{\mathrm{iid}}$. Using equations~\eqref{eq:correlation} and~\eqref{eq:det}, it is readily obtained that $I_\mathcal{S}^{\mathrm{DPP}}$ is an unbiased estimator of $I_n$.

For a given sample $\mathcal{S}$ and estimator $I_\mathcal{S}$, we evaluate the relative error $\frac{\vert I_\mathcal{S} - I \vert}{I}$ for both estimators. We compute these errors for $n_\mathcal{S} = 1000$ draws of the sample $\mathcal{S}$ and record in Figure~\ref{fig:mc_sphere} the average of this error for both estimators for different values of $m \in \{1,2,4,8,16,32,64,128\}$, as averaged over $10$ realizations of $X_n$. The DPP-based estimator is qualitatively more sensitive to the value of $m$ than its iid counterpart, and we observe two different regimes. 1/ For small $m$'s, the relative error decreases at a faster \emph{rate} for $I_\mathcal{S}^{\mathrm{DPP}}$. 2/ For higher $m$'s, this is not the case anymore, and the relative error for $I_\mathcal{S}^{\mathrm{DPP}}$ actually increases. In this regime, $m$ is no longer small enough with respect to $n$ that the first $m$ eigenfunctions $\phi_i$ of $-\mathcal{L}$ can be faithfully approximated by the $v_i$'s, and $\mathrm{DPP}(K_n,\mu_n)$ no longer behaves similarly to $\mathrm{DPP}(\mathcal{K},\mu)$

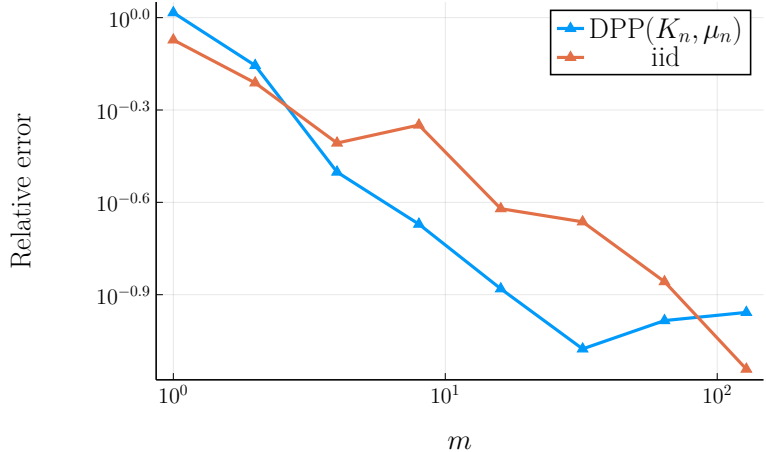
\begin{figure}
    \centering
    \scalebox{0.6}{

\begin{tikzpicture}[/tikz/background rectangle/.style={fill={rgb,1:red,1.0;green,1.0;blue,1.0}, fill opacity={1.0}, draw opacity={1.0}}, show background rectangle]
\begin{axis}[point meta max={nan}, point meta min={nan}, legend cell align={left}, legend columns={1}, title={}, title style={at={{(0.5,1)}}, anchor={south}, font={{\fontsize{14 pt}{18.2 pt}\selectfont}}, color={rgb,1:red,0.0;green,0.0;blue,0.0}, draw opacity={1.0}, rotate={0.0}, align={center}}, legend style={color={rgb,1:red,0.0;green,0.0;blue,0.0}, draw opacity={1.0}, line width={1}, solid, fill={rgb,1:red,1.0;green,1.0;blue,1.0}, fill opacity={1.0}, text opacity={1.0}, font={{\fontsize{18 pt}{10.4 pt}\selectfont}}, text={rgb,1:red,0.0;green,0.0;blue,0.0}, cells={anchor={center}}, at={(0.98, 0.98)}, anchor={north east}}, axis background/.style={fill={rgb,1:red,1.0;green,1.0;blue,1.0}, opacity={1.0}}, anchor={north west}, xshift={1.0mm}, yshift={-1.0mm}, width={150.4mm}, height={99.6mm}, scaled x ticks={false}, xlabel={$m$}, x tick style={color={rgb,1:red,0.0;green,0.0;blue,0.0}, opacity={1.0}}, x tick label style={color={rgb,1:red,0.0;green,0.0;blue,0.0}, opacity={1.0}, rotate={0}}, xlabel style={at={(ticklabel cs:0.5)}, anchor=near ticklabel, at={{(ticklabel cs:0.5)}}, anchor={near ticklabel}, font={{\fontsize{18 pt}{14.3 pt}\selectfont}}, color={rgb,1:red,0.0;green,0.0;blue,0.0}, draw opacity={1.0}, rotate={0.0}}, xmode={log}, log basis x={10}, xmajorgrids={true}, xmin={0.8645372313078652}, xmax={148.05608753987678}, xticklabels={{$10^{0}$,$10^{1}$,$10^{2}$}}, xtick={{1.0,10.0,100.0}}, xtick align={inside}, xticklabel style={font={{\fontsize{14 pt}{10.4 pt}\selectfont}}, color={rgb,1:red,0.0;green,0.0;blue,0.0}, draw opacity={1.0}, rotate={0.0}}, x grid style={color={rgb,1:red,0.0;green,0.0;blue,0.0}, draw opacity={0.1}, line width={0.5}, solid}, axis x line*={left}, x axis line style={color={rgb,1:red,0.0;green,0.0;blue,0.0}, draw opacity={1.0}, line width={1}, solid}, scaled y ticks={false}, ylabel={Relative error}, y tick style={color={rgb,1:red,0.0;green,0.0;blue,0.0}, opacity={1.0}}, y tick label style={color={rgb,1:red,0.0;green,0.0;blue,0.0}, opacity={1.0}, rotate={0}}, ylabel style={at={(ticklabel cs:0.5)}, anchor=near ticklabel, at={{(ticklabel cs:0.5)}}, anchor={near ticklabel}, font={{\fontsize{18 pt}{14.3 pt}\selectfont}}, color={rgb,1:red,0.0;green,0.0;blue,0.0}, draw opacity={1.0}, rotate={0.0}}, ymode={log}, log basis y={10}, ymajorgrids={true}, ymin={0.06659449895441653}, ymax={1.1242765986857066}, yticklabels={{$10^{-0.9}$,$10^{-0.6}$,$10^{-0.3}$,$10^{0.0}$}}, ytick={{0.12589254117941667,0.25118864315095796,0.5011872336272722,1.0}}, ytick align={inside}, yticklabel style={font={{\fontsize{14 pt}{10.4 pt}\selectfont}}, color={rgb,1:red,0.0;green,0.0;blue,0.0}, draw opacity={1.0}, rotate={0.0}}, y grid style={color={rgb,1:red,0.0;green,0.0;blue,0.0}, draw opacity={0.1}, line width={0.5}, solid}, axis y line*={left}, y axis line style={color={rgb,1:red,0.0;green,0.0;blue,0.0}, draw opacity={1.0}, line width={1}, solid}, colorbar={false}]
    \addplot[color={rgb,1:red,0.0;green,0.6056;blue,0.9787}, name path={1}, draw opacity={1.0}, line width={2}, solid, mark={triangle*}, mark size={3.75 pt}, mark repeat={1}, mark options={color={rgb,1:red,0.0;green,0.6056;blue,0.9787}, draw opacity={1.0}, fill={rgb,1:red,0.0;green,0.6056;blue,0.9787}, fill opacity={1.0}, line width={0.75}, rotate={0}, solid}]
        table[row sep={\\}]
        {
            \\
            1.0  1.0378496660652605  \\
            2.0  0.6991856692217173  \\
            4.0  0.31502035587241245  \\
            8.0  0.21335494433916408  \\
            16.0  0.13157499109909854  \\
            32.0  0.08392276269863781  \\
            64.0  0.1037335231324736  \\
            128.0  0.11027887537004333  \\
        }
        ;
    \addlegendentry {$\mathrm{DPP}(K_n,\mu_n)$}
    \addplot[color={rgb,1:red,0.8889;green,0.4356;blue,0.2781}, name path={2}, draw opacity={1.0}, line width={2}, solid, mark={triangle*}, mark size={3.75 pt}, mark repeat={1}, mark options={color={rgb,1:red,0.8889;green,0.4356;blue,0.2781}, draw opacity={1.0}, fill={rgb,1:red,0.8889;green,0.4356;blue,0.2781}, fill opacity={1.0}, line width={0.75}, rotate={0}, solid}]
        table[row sep={\\}]
        {
            \\
            1.0  0.8465071648720806  \\
            2.0  0.6138689495482618  \\
            4.0  0.391477596990573  \\
            8.0  0.44780818814247353  \\
            16.0  0.23957888844459846  \\
            32.0  0.2171263368070116  \\
            64.0  0.13878971089496295  \\
            128.0  0.0721401559625711  \\
        }
        ;
    \addlegendentry {iid}
\end{axis}
\end{tikzpicture}}
    \caption{\centering Relative error as a function of $m$ in log-log scale.}
    \label{fig:mc_sphere}
\end{figure}

\begin{rem}
    For our implementation in Julia,\footnote{Available at \url{https://gitlab.irisa.fr/hjaquard/discrete-to-continuous-dpps}.} we use a naive implementation of the Gram-Schmidt process to obtain the $v_i$'s from the $u_i$'s. This computation is subject numerical instability, and the computed $u_i$'s may not be orthonormal, especially for larger values of $m$. In particular, the eigenvalues of $K_n$ may exceed $1$. To mitigate this issue, we further divide $K_n$ by its largest eigenvalue in our implementation, so that we can 1/ ensure that our DPPS always exist and 2/ rely on existing DPP samplers. We remark that this issue could be resolved with finer engineering, but has little impact on our \emph{qualitative} observation.

    In the computation of $e[p]$, we need to know the dimension of the manifold $\mathcal{X}$ (here, $d_\mathcal{X} = 2$). For simplicity, we assume in our pipeline that it is known, but we stress that this value could be estimated based on the point cloud $X_n$ itself; we refer to, \emph{e.g.},~\cite{bi2025manifold} for a recent review of this classical issue. In particular, for well-behaved manifolds, the minimax risk for dimension estimation decreases exponentially with $n$~\cite{kim2016minimax}.
\end{rem}

\section{Discrete-to-continuous limits of DPPs on random graphs}
\label{sect:usvt}

Our last example pertains to latent position random graphs~\cite{crane2018probabilistic}, for which the positions of the points in $X_n$ are not available, but implicitly represented by random edges between those points.
We are going to show that weak-coherency can be used to construct discrete-to-continuous limits even in this challenging setting. Compared to the previous examples, there are two important technical differences: first, the kernel $K_n$ has an additional layer of randomness on top of $X_n$ (the random edges). Second, more crucially, here $[K_n]_{ij}$ will \emph{not} converge to $\mathcal{K}(x_i,x_j)$ as in the two previous examples, but we will rather rely on the second condition of theorem~\ref{th:detailed}, with concentration in Frobenius norm and trace.

Note that, in terms of applications, the results in this section should be understood as a proof-of-concept rather than a practical way to sample (random) graphs with DPPs. Indeed, in the literature, graph sampling with DPPs is rather performed with the so-called root process of random forests (\emph{e.g.},~\cite{avena2018two,pilavci_graph_2021}), which unfortunately does not seem to be directly covered by the results of this paper. Studying this more involved process is a major path for future work.

\subsection{Latent position random graphs and the connectivity kernel}

In the most common setting, latent position random graphs are defined from a so-called \emph{connectivity kernel} $\mathcal{W}:\mathcal{X} \times \mathcal{X} \rightarrow [0,1]$ for some \emph{compact} $\mathcal{X} \subseteq \R^d$. Given such a kernel, a \emph{sparsity level} $\alpha_n \leq 1$ and a probability distribution $\mu$ over $\mathcal{X}$, a latent position random graph with $n$ nodes is drawn by first sampling $n$ points $x_1,...,x_n \subseteq \mathcal{X}$ iid according to $\mu$, and then drawing edges 
\begin{equation}
    \label{eq:lpm}
    a_{i,j} \sim \mathrm{Ber}(\alpha_n\mathcal{W}(x_i,x_j)).
\end{equation}
Here, $a_{i,j} = a_{j,i}$ are the entries of the adjacency matrix $A$ of the graph, and we stress that this adjacency matrix is typically the only information available: unlike the two previous examples, \emph{the positions $x_i$ are not known}. 
The coefficient $\alpha_n$ controls the sparsity of the graph, and there are $\mathcal{O}(\alpha_n n^2)$ edges in expectation. It is common to differentiate between different density regimes: the so-called dense graphs when $\alpha_n \in \mathcal{O}(1)$, sparse graphs when $\alpha_n \in \mathcal{O}\left( \frac{1}{n} \right)$, and \emph{relatively sparse} graphs for any rate in-between. In general, non-asymptotic convergence results can be obtained when $\alpha_n \gtrsim \frac{\log(n)}{n}$, and this is the regime we consider.

Assuming that the kernel $\mathcal{K} = \mathcal{W}$ is a valid kernel for a DPP ---namely, that $\mathcal{W}$ is continuous, symmetric, and that the integral operator $T_\mathcal{W}$ associated to $\mathcal{W}$ and $\mu$ has its eigenvalues in $[0,1]$, so that $\mathrm{DPP}(\mathcal{K},\mu)$ exists by the Macchi-Soshnikov theorem--- we examine the following question: by observing only the adjacency matrix $A$, do we have enough information to approximate a DPP with respect to $\mathcal{K}$? Or, in our context, can we construct solely from $A$ a matrix $K_n$ such that $\textrm{DPP}(K_n, \mu_n)$ is weakly consistent with $\textrm{DPP}(\mathcal{K}, \mu)$?\footnote{Recall here that the $x_i$ are not known. Hence sampling the $x_i$'s rather consists in sampling indices among $1,\ldots, n$, as described in section~\ref{sect:background_DPP}.}

\subsection{Weak coherency from Universal Singular Value Thresholding}

A first observation is that
\begin{equation*}
    \mathbf{E}\left[ \frac{A}{\alpha_n} \ \vert \ X_n \right] = \mathcal{K}_{\vert X_n \times X_n}.
\end{equation*}
so that $A$ is indeed a good candidate to estimate the Gram matrix. In fact, one is able to prove that convergence between the two holds \emph{in operator norm}~\cite{lei2015consistency}. However, this is not enough to apply theorem~\ref{th:intro}: this clearly does not converge entry-wise and, moreover, this concentration does not hold either in Frobenius norm (which would just estimate the sum of the variance of the Bernoulli variables). 

As it turns out, it is then possible to recover both the Frobenius-norm and trace concentration by relying on a powerful tool from the matrix-estimation literature, the so-called universal singular value thresholding (USVT)~\cite{chatterjee2015matrix}, which proceeds by considering a modified, edge-weighted graph\footnote{It \emph{may} be possible to obtain weak-coherency results using different tools. Using USVT approximations comes with two convenient features: 1/ we can rely on existing, powerful concentration results; 2/ USVT estimates are efficient to compute (see, \emph{e.g.}, the introduction of~\cite{luo2024computational} for a discussion on this topic).}. The full result will still require substantial modification compared to the vanilla USVT.

In the rest of the section, we assume that $\mathcal{K}(x,x) = c$ for all $x \in \mathcal{X}$, where $c \leq 1$ is some constant: this is for instance satisfied when $\mathcal{W}$ is a radial kernel, which is a classical assumption for latent position random graphs. For simplicity, we also assume that $c$ is known.

\subsubsection{USVT-based kernel}
Denoting by $A = \sum_{i = 1}^n \lambda_i u_i u_i^t$ the eigendecomposition of $A$ and by $\gamma_n \geq 0$ some non-negative threshold, the USVT estimator is defined by
\begin{equation*}
    \widetilde{A}_{\gamma_n} = \frac{1}{\alpha_n} \sum_{\lambda_i \geq \gamma_n} \lambda_i u_i u_i^t.
\end{equation*}
It has been proven that this estimate suffices to recover the Frobenius-norm concentration~\cite{nicoOT} (see section~\ref{sect:proof_usvt} of the supplementary material), but it is not enough to ensure the trace-concentration condition. Further, there is no guarantee that the corresponding DPP even exists. To resolve these issues, we to consider a modified USVT estimate. The first step is to construct a diagonal perturbation of $\widetilde{A}_{\gamma_n}$:
\begin{equation}
    \bar A_{\gamma_n} = \widetilde{A}_{\gamma_n} + C(\widetilde{A}_{\gamma_n}) I, \ \ \text{with} \ \ C(\widetilde{A}_{\gamma_n}) = \max\left( c - \frac{\mathrm{tr}\left( \widetilde{A}_{\gamma_n} \right)}{n},~ 0 \right),
\end{equation}
where we recall that $c$ is the value of $\mathcal{K}(x,x)$, and therefore of $\mathrm{tr}(\mathcal{K}_{X_n\times X_n}/n)$. As the proof of proposition~\ref{prop:usvt} below shows, the trace condition is always satisfied for $\overline{A}_{\gamma_n}$, but the associated DPP may not exist. To satisfy the assumptions of the Macchi-Soshnikov theorem, we consider the kernel
\begin{equation}
    \qquad K_n = C'(\bar{A}_{\gamma_n}) \bar A_{\gamma_n},  \ \ \text{where} \ \ C'(\bar{A}_{\gamma_n}) = \min\left( \frac{n}{\lambda_{\max}(\bar A_{\gamma_n})},~ \left(1+ \frac{1}{(\alpha_n n)^{1/4}}\right)^{-1}\right),
\end{equation}
where $\lambda_{\max}(A)$ is the largest eigenvalue of $A$.
Since $C(\widetilde{A}_{\gamma_n}) \geq 0$, the eigenvalues of $\overline{A}_{\gamma_n}$ are non-negative and, since $0 \leq C'(\bar{A}_{\gamma_n}) \leq \frac{n}{\lambda_{\max}(\bar A_{\gamma_n})}$, we obtain that the eigenvalues of $K_n$ are in $[0,n]$. As $K_n$ is indeed symmetric, $\mathrm{DPP}(K_n,\mu_n)$ exists.
Then, we are able to show the Frobenius norm and trace-concentration of $K_n$ towards $\mathcal{K}$.

\begin{prop}
    \label{prop:usvt}
    Assume that $\alpha_n \gtrsim \log n / n$. For any $q>0$, there exist three constants $\rho_q, b_q, B_q\in \R_{>0}$ such that, for any $\delta \in (0,1)$, $\eps > 0$ and for $\gamma_n = \rho_q (\alpha_n n)^{3/4}$, 
    \begin{equation*}
        n \geq N_K(\delta,\eps) ~ \Rightarrow ~ \mathbb{P}\left( \max\left( \left\Vert \tfrac{K_n - \mathcal{K}_{\vert X_n \times X_n}}{n} \right\Vert_F, \left\vert \mathrm{tr}\left( \tfrac{K_n}{n} \right) - \mathrm{tr}\left( \tfrac{\mathcal{K}_{\vert X_n \times X_n}}{n} \right) \right\vert \right) \geq \eps \right) \leq \delta,
    \end{equation*}
    where
    \begin{equation}
        \label{eq:NK_usvt}
        N_K(\delta,\eps) = \max\left( \frac{1}{\delta^q}, \frac{b_q}{\alpha_n \eps^8}, \frac{B_q}{\alpha_n \eps^4} \right).
    \end{equation}
\end{prop}

Here, depending on the desired guarantee, the parameter $q$ can be tweaked to increase the probability of the bound being satisfied, at the price of larger multiplicative constants. The proof of proposition~\ref{prop:usvt} is detailed in section~\ref{sect:proof_usvt} of the supplementary material, and is a generalization of the result from~\cite{nicoOT}. It relies on a concentration result from~\cite{lei2015consistency} and on spectral inequalities, and we stress that it is only valid for $\alpha_n \gtrsim \frac{\log(n)}{n}$.

\begin{rem}
    Achievable rates for $\eps = \eps_n$ directly depend on the sparsity level $\alpha_n$. For relatively sparse graphs with $\alpha_n \gtrsim \frac{\log n}{n}$, the $\frac{b_q}{\alpha_n \eps^8}$ term in particular imposes a rate $\eps_n$ no faster $\mathcal{O}\left( \frac{1}{\log(n)^{1/8}}\right)$. In general, the sparser the graph is, the slower the achievable rates.
\end{rem}

We can now use theorem~\ref{th:detailed}, $ii)$, to obtain the following.

\begin{cor}
    Under the setting of proposition~\ref{prop:usvt}, $\mathrm{DPP}(K_n,\mu_n)$ is weakly coherent with $\mathrm{DPP}(\mathcal{K},\mu)$. For a given $q$, the weak coherency rate is given by the $N_K$ in equation~\eqref{eq:NK_usvt}.
\end{cor}

\begin{rem}
    There are practical obstacles to actually implementing our estimator $K_n$: first, the sparsity level $\alpha_n$ is not known, and has to be estimated; second, the hyperparameter $\gamma_n$ is only known up to a multiplicative constant and must also be adjusted by validation procedures; third, the diagonal-value $c$ may not always be known. As mentioned above, the estimator $K_n$ presented here is mostly a theoretical proof-of-concept. In practice, sampling graphs with DPPs is generally done with random forests, which we leave for future work.
\end{rem}

\section{Discussion and perspectives}
\label{sect:discussion}

With the exception of very structured problems like linear regression (\emph{e.g.},~\cite{derezinsky_reverse_2018,derezinski2021determinantal}), establishing {strict}, quantifiable better-than-independent guarantees for discrete DPPs is a notoriously challenging problem. On the other hand, the practical statistical advantage of DPPs over iid sampling on different tasks is supported by many experimental results, though at an increased computational cost. In order to quantify this advantage, we proposed a systematic methodology that can be used to translate statistical guarantees from  continuous DPPs to discrete DPPs. In particular, the concentration of \emph{all the moments} of linear statistics is controlled by the concentration rate \emph{of the kernel}, either in max norm or in Frobenius norm and trace. These criteria allow to deal both with fixed-size, projection DPPs (for which the max-norm concentration is already a powerful tool), and with full-rank, very noisy estimates of the continuous kernel. The approximation rate for these statistics is lower-bounded by $\mathcal{O}\left( \sqrt{\frac{\log\left( \frac{1}{\delta} \right)}{n}} \right)$, which matches the classical concentration rate for $n$ iid variables; on that front, we expect our result to be optimal.

While our presentation is focused on DPPs defined over compact spaces, our analysis of the kernel error in section~\ref{sect:compact_kernel} is very general. The analysis of the measure error in section~\ref{sect:compact_measure}, on the other hand, heavily depends on the compactness hypothesis, but could be generalized to sets $X_n$ sampled according to, \emph{e.g.}, a Poisson process of increasing intensity $r(n) \mu$ for some Radon measure $\mu$ and $r(n)$ increasing with $n$, by relying on concentration results for Poisson samples (see, \emph{e.g.},~\cite{bachmann2016concentration}).

In terms of extensions, our results would warrant an extensive empirical study for each of the three DPPs we study in sections~\ref{sect:ope_main},~\ref{sect:harmonic_ensemble} and~\ref{sect:usvt}. Let us now discuss a few other perspectives.

\subsubsection{Central limit theorems} Faster-than-iid central limit theorems (CLTs) for some classes of functions defined over $\mathcal{X}$ have been obtained for the continuous multivariate orthogonal polynomial ensembles and harmonic ensembles~\cite{bardenet2020monte,levi2024linear}, thus establishing strong guarantees for DPP-based Monte-Carlo integration. Could CLTs then be obtained for functions \emph{defined over $\mathcal{X}$} but \emph{when sampling from $X_n$} instead of $\mathcal{X}$, using a discrete DPP? In particular, following a criterion established in~\cite{soshnikov2002gaussian}, CLTs for DPPs can be established, essentially, as soon as an the variance of ($1$-point) linear statistics grows sufficiently fast with respect to their expectation. As weak coherency allows to control both quantities when $n$ is large enough, it is likely that existing CLTs could be translated to discrete DPPs. We stress that, beyond the rate, results of this type would provide useful guarantees in challenging statistical settings, where the interest lies in functions defined over the latent space $\mathcal{X}$, that is not observed (\emph{e.g.}, an unknown manifold, like in section~\ref{sect:harmonic_ensemble}).

\subsubsection{L-ensembles} In practice, it is often more convenient for the practitioner to define discrete DPPs in terms of (extended) L-ensembles~\cite{tremblay2023extended}, instead of their marginal probabilities as in equations~\eqref{eq:counting_measure} and~\eqref{eq:empirical_measure}. L-ensembles describe the \emph{sample probabilities} of each possible $\mathcal{X}$, rather than the co-incidence probability of given points in the sample $\mathcal{S}$, using an auxiliary kernel $L$. In particular, the kernel $L$ is a Lipschitz function of the marginal kernel $K$, and we can expect the concentration bounds of theorem~\ref{th:detailed} to translate to this setting (the Frobenius-norm bound can for instance be translated using results from matrix analysis~\cite{kittaneh1985lipschitz}). While the idea is quite simple, the continuous counterparts of $L$-ensembles are described using the Janossy densities of the process, and this extension would require an additional technical overhead.

\subsubsection{Roots of random forests}
As we mentioned in section~\ref{sect:usvt}, an important motivation for our work is the study of the continuous limit of the root process of random forests~\cite{avena2018two}. This is a DPP with kernel $q(L + qI)^{-1}$, for some $q> 0$ and $L$ the usual graph Laplacian, that is \emph{defined with respect to the counting measure $\sum_{x \in X_n} \delta_x$}. Working with this unnormalized measure brings about an important difference: this amounts to replacing $K_n$ (resp. $\mathcal{K}$) by $n K_n$ (resp. $n \mathcal{K}$) in theorem~\ref{th:detailed}, which imposes a much more stringent condition on the concentration of the kernels. In particular, it is unlikely that we could satisfy this condition for random graphs such as those we consider in section~\ref{sect:usvt}. 

On the other hand, our criteria for weak-coherency over \emph{non-compact} spaces $\mathcal{X}$ (section~\ref{sect:compact_kernel}) may still be satisfied in this setting: this is because, for those results to apply, we only require the kernels to concentrate over the $X_n \cap C$ for all compacts $C$ of $\mathcal{X}$, whereas the kernel $K_n$ may be constructed from the entire data set $X_n$. As it turns out, important classes of random graphs over non-compact latent spaces have been studied in the literature, such as the so-called graphexes~\cite{veitch2019sampling,borgs2019sampling,borgs2020identifiability}, that are most well-known as an appropriate framework for limits of \emph{sparse} graphs, or hyperbolic random graphs~\cite{krioukov2010hyperbolic} and other related sparse graph models (see, \emph{e.g.},~\cite{van2025projective}). Investigating local USVT-like concentration results for these models will be the object of future work.

\subsubsection{Single-sample DPP identification.} Let us finally briefly go back to the DPP-estimation setting from a single infinite sample of~\cite{poinas2023asymptotic}, that we briefly discussed in section~\ref{sect:related_work}. Consider a DPP on $\R^d$, with kernel $\mathcal{K}$ and defined with respect to the Lebesgue measure $\lambda^{\otimes r}$. Given a single, infinite sample $\mathcal{S}$ from $\mathrm{DPP}(\mathcal{K},\lambda^{\otimes r})$, the aim is to estimate the kernel $\mathcal{K}$ from a parametric family $(\mathcal{K}_\theta)_\theta$ using a maximum-likelihood criterion for a given observation window $W \subseteq \R^d$. Computing the likelihood $l_W(\theta)$ of $\mathcal{S} \cap W$ in practice is unfeasible, and~\cite{poinas2023asymptotic} proposes a surrogate likelihood $\tilde{l}_W(\theta)$ to approximate it. Without going into technicalities, they obtain a series of results: first, they show that DPPs defined on an arbitrarily fine grid $\eta \Z^d$ ($\eta > 0$), $\tilde{l}_W(\theta)$ converges almost surely to $l_W(\theta)$ as the size of the window $W$ increases (here, $\eta$ remains fixed); second, they show that an appropriately defined DPP on $\eta Z^d$ converges weakly to a DPP on $\R^d$ as $\eta \rightarrow 0$. 

This leaves open a number of questions: 1/ does $\tilde{l}$ converge to $l$ for DPPs defined over $\R^d$? 2/ does the maximum likelihood \emph{estimator} (rather than just the likelihood), that is, a kernel on $\eta \Z^d$, converge to $\mathcal{K}$? 3/ how does this kernel estimation affect downstream processing tasks? We expect that weak coherency (over non-compact spaces) may provide an appropriate framework to answer some of these questions.

%
%

\begin{appendix}

\section{Weak convergence for point processes}
\label{sect:weak_convergence}

Weak convergence generalizes the convergence in law of real-valued random variables to measures defined over general spaces called Polish spaces. Our goal in this section is to briefly recall this notion, and in particular its instantiation as a notion of limit for point processes. In short, it requires that the expectation of a large family of statistics converges. For a thorough treatment, the reader can refer to~\cite{van1996weak,daley2008introduction}.

Consider a sequence of measures $\lambda_n$ defined over a Polish space $\Omega$.\footnote{We endow this space with its Borel $\sigma$-algebra.} We say that $\lambda_n$ converges weakly to a measure $\lambda$ over $\Omega$ if
\begin{equation*}
    \int_{\Omega} h d\lambda_n \xrightarrow[]{n \rightarrow \infty} \int_{\Omega} h d\lambda
\end{equation*}
for all bounded and continuous functions $h:\Omega \rightarrow \R$. In words, the integral of any admissible function $f$ with respect to $\lambda_n$ must converge to its integral with respect to $\lambda$.

To cover the case of point processes, we need to identify locally finite subsets of $\Gamma$ with $\{0,1\}$-valued locally-finite measures, the so-called \emph{configurations} in $\Gamma$: 
\begin{equation*}
    \mathrm{Conf}(\Gamma) = \left\{ \mathcal{S} = \sum_{x \in B} \delta_x \ ;~ B \subseteq \Gamma, \text{ such that } \mathcal{S}(C) < \infty \text{ for all compact sets } C \subseteq \Gamma \right\}.  
\end{equation*}
Configurations of points in $\mathrm{Conf}(\Gamma)$ are in one-to-one correspondence with locally finite subsets of $\Gamma$, and a point process also be regarded as a 
$\mathrm{Conf}(\Gamma)$-valued random variable. 
Let us then fix a second-countable locally compact Hausdorff space $\Gamma$. Consider its space of configurations $\mathrm{Conf}(\Gamma)$, and let us endow it with the vague topology, \emph{i.e.}, the topology such that $\mathcal{S}_n \xrightarrow[]{n \rightarrow \infty} \mathcal{S}$ when
\begin{equation*}
    \int_{\Gamma} f d\mathcal{S}_n \xrightarrow[]{n \rightarrow \infty} \int_{\Gamma} f d\mathcal{S}
\end{equation*}
for all continuous $f:\Gamma \rightarrow \R$ with compact support. In particular, this makes $\mathrm{Conf}(\Gamma)$ a Polish space, upon which weak convergence can be defined.

We can now a sequence of point processes $\mathcal{P}_n$ over $\Gamma$. Here, we regard a point process not as a random locally finite subset but as a $\mathrm{Conf}(\Gamma)$-valued random variable, so that its associated distribution is nothing but a measure over $\mathrm{Conf}(\Gamma)$. We then say that the sequence $\mathcal{P}_n$ converges weakly to a point process $\mathcal{P}$ when the sequence of measures of its associated distributions converges to that of $\mathcal{Q}$, \emph{i.e.}, when
\begin{equation*}
    \mathbf{E}_{\mathcal{S}_n \sim \mathcal{P}_n}[h(\mathcal{S}_n)] \xrightarrow[]{n \rightarrow \infty} \mathbf{E}_{\mathcal{S} \sim \mathcal{Q}}[h(\mathcal{S})]
\end{equation*}
for all bounded functions $h: \mathrm{Conf}(\Gamma) \rightarrow \R$ continuous with respect to the vague topology. 
From a practical perspective, this means that, for any admissible statistic $h$ to be evaluated on a random configuration, its expectation over $\mathcal{P}_n$ should converge to that over $\mathcal{Q}$. 

In fact, weak convergence satisfies a so-called continuous mapping theorem, which is a stability result similar to the moment mapping theorem of section~\ref{sect:lin_stats_moment_map}: if $\lambda_n$ converges weakly to $\lambda$, then $\int_{\Omega} (h \circ g) d\lambda_n \rightarrow \int_{\Omega} (h \circ g) d\lambda$ for any continuous function $g$.
The two results are different in that the moment mapping theorem pertains to moments of statistics computed over \emph{random} point processes, and comes with an additional guarantee on the concentration rate.

\section{Missing proofs from section~\ref{sect:awc_compact}}
\label{sect:preuve_detailed}

For readability, we only write the proofs for the variants of propositions~\ref{prop:uniform_determinant_concentration_matrix} and~\ref{lem:concentration_determinant_matrices_main} pertaining to matrices. As our proofs only rely on the triangle and Cauchy-Schwarz inequalities, the versions for general kernels are obtained by applying the exact same arguments to general kernels, and replacing the corresponding sums by integrals over $\mathrm{supp}(\varphi_r)$ with respect to $\nu_n$ (see remarks~\ref{rem:generalisation_max_determinant} and~\ref{rem:generalisation_mean_determinant} below for details).

Let us first recall a useful identity.

\begin{prop}
    \label{eq:product_identity}
    Let $a_1,...,a_r \in \mathbf{C}$ and $b_1,...,b_r \in \C$. Then,
    \begin{equation}
    \prod_{i=1}^r a_i - \prod_{i = 1}^r b_i = \sum_{j = 1}^r \left( \prod_{k = 1}^{j-1} a_k \right) \left( a_j - b_j \right) \left( \prod_{k = j+1}^r b_k \right).
    \end{equation}
\end{prop}

\subsection{Proof of proposition~\ref{prop:uniform_determinant_concentration_matrix}}
\label{sect:proof_uniform_concentration_compact}

In order to simplify a number of book-keeping operations, it will be convenient to tweak our formalism a little. Let us consider a finite set $I_r = \{i_1,...,i_r\}$, understood as a set of finite indices, and a value function $\val: I_r \rightarrow [n]$ associating a value to each of the indices. Our aim is to show that, for $b(\val) = \vert \det(A)_{\val(I_r)} - \det(B)_{\val(I_r)} \vert$,
\begin{align*}
    b(\val) \leq r! \sum_{j = 1}^r \left( \max_{k,l \in [n]} \vert A_{k,l} \vert \right)^{j-1} \left(\max_{k,l \in [n]} \vert A_{k,l} - B_{k,l} \vert \right) \left( \max_{k,l\in [n]} \vert B_{k,l} \vert\right)^{r-j},
\end{align*}
where we adopt the convention that $\det(M)_{\val(I_r)} = 0$ whenever $\val$ is not injective. The reader can check that this is equivalent to the statement of proposition~\ref{prop:uniform_determinant_concentration_matrix}.
Let us denote by $\mathfrak{S}_{I_r}$ the set of permutations over the indices $I_r$ (let us stress that these permutations act on $I_r$ itself, and not on $\val(I_r)$). Applying successively the triangle inequality and equation~\eqref{eq:product_identity}, we obtain
\begin{align*}
    b(\val) & = \left\vert \sum_{\sigma \in \mathfrak{S}_{I_r}} \mathrm{sgn}(\sigma)\prod_{j = 1}^r A_{\val(i_j),\val(\sigma(i_j))} - \sum_{\sigma \in \mathfrak{S}_{I_r}} \mathrm{sgn}(\sigma) \prod_{j = 1}^r B_{\val(i_j),\val(\sigma(i_j))} \right\vert \\
    & \leq \sum_{\sigma \in \mathfrak{S}_{i_r}} \left\vert \prod_{j = 1}^r A_{\val(i_j),\val(\sigma(i_j))} - \prod_{j = 1}^r B_{\val(i_j),\val(\sigma(i_j))} \right\vert \\
    & = \sum_{\sigma \in \mathfrak{S}_{i_r}} \left\vert \sum_{j = 1}^r \left( \prod_{k = 1}^{j-1} A_{\val(i_k),\val(\sigma(i_k))} \right) \left( A_{\val(i_j),\val(\sigma(i_j))} - B_{\val(i_j),\val(\sigma(i_j))} \right) \left( \prod_{k=j+1}^r B_{\val(i_k),\val(\sigma(i_k))} \right) \right\vert \\
    & \leq  r! \sum_{j = 1}^r \left( \max_{k,l \in [n]} \vert A_{k,l} \vert \right)^{j-1} \left(\max_{k,l \in [n]} \vert A_{k,l} - B_{k,l} \vert \right) \left( \max_{k,l\in [n]} \vert B_{k,l} \vert\right)^{r-j},
\end{align*}
which proves proposition~\ref{prop:uniform_determinant_concentration_matrix}.

\begin{rem}
    \label{rem:generalisation_max_determinant}
    To prove proposition~\ref{prop:uniform_determinant_concentration}, we can repeat the exact same argument, replacing the matrices $A$ and $B$ by kernels $K_n^\mathcal{P}$ and $K_n^\mathcal{Q}$. In that case, we can no longer consider valuations $\val:I_r \rightarrow [n]$, and expressions of the form $A_{\val(i_j),\val(\sigma(i_j))}$ are replaced by $K_n^\mathcal{P}(x_j,x_{\sigma(j)})$ for a given permutation $\sigma$ of $[r]$.
\end{rem}

\subsection{Proof of lemma~\ref{lem:concentration_determinant_matrices_main}}
\label{sect:proof_mean_concentration_compact}

We once again adopt the formalism from section~\ref{sect:proof_uniform_concentration_compact}, and consider a finite set of indices $I_r$ and a value function $\val:I_r \rightarrow \R$. We are going to show that
\begin{equation*}
    \sum_{\val:I_r \rightarrow [n]} \vert \det(A)_{\val(I_r)} - \det(B)_{\val(I_r)}\vert \leq (r \times r!) M_{A,B}^{r - 1} \max\left(\left\Vert A - B \right\Vert_F , \left\vert \mathrm{tr}\left(A \right) - \mathrm{tr}\left( B \right)\right\vert \right),
\end{equation*}
where $M_{A,B} = \max\left( h(A,B), t(A,B) \right)$, for $h(A,B) = \max\left( \Vert A \Vert_F, \Vert B \Vert_F \right)$ and $t(A,B) = \max\left( \mathrm{tr}(A), \mathrm{tr}(B) \right)$. This is nothing but a reformulation of lemma~\ref{lem:concentration_determinant_matrices_main}.

We begin with an application of the triangle inequality, which yields
\begin{equation*}
    \sum_{\val:I_r \rightarrow [n]} \vert \det(A)_{\val(I_r)} - \det(B)_{\val(I_r)}\vert \leq \sum_{\sigma \in \mathfrak{S}_{I_r}} \left\vert \sum_{\val:I_r \rightarrow [n]} \left( \prod_{j = 1}^r A_{\val(i_j),\val(\sigma(i_j))} - \prod_{j=1}^r B_{\val(i_j),\val(\sigma(i_j))} \right) \right\vert
\end{equation*}
so that, using equation~\eqref{eq:product_identity} on $E_\sigma = \left\vert \sum_{\val:I_r \rightarrow [n]} \left( \prod_{j = 1}^r A_{\val(i_j),\val(\sigma(i_j))} - \prod_{j=1}^r B_{\val(i_j),\val(\sigma(i_j))} \right) \right\vert$,
\begin{equation*}
    E_\sigma \leq \sum_{j = 1}^r \left\vert \sum_{\val:I_r \rightarrow [n]}\left( \prod_{k = 1}^{j-1} A_{\val(i_k),\val(\sigma(i_k))} \right) \left( A_{\val(i_j),\val(\sigma(i_j))} - B_{\val(i_j),\val(\sigma(i_j))} \right) \left( \prod_{k=j+1}^r B_{\val(i_k),\val(\sigma(i_k))} \right) \right\vert.
\end{equation*}

To simplify the upcoming equations, let us introduce another piece of notation, and re-write each of the terms
\begin{equation*}
E_{\sigma,j} = \left\vert \sum_{\val:I_r \rightarrow [n]}\left( \prod_{k = 1}^{j-1} A_{\val(i_k),\val(\sigma(i_k))} \right) \left( A_{\val(i_j),\val(\sigma(i_j))} - B_{\val(i_j),\val(\sigma(i_j))} \right) \left( \prod_{k=j+1}^r B_{\val(i_k),\val(\sigma(i_k))} \right) \right\vert
\end{equation*}
as
\begin{equation*}
    E_{\sigma,j} = \left\vert \sum_{\val:I_r \rightarrow [n]} \prod_{k = 1}^r c(\sigma,j,\val,k) \right\vert,
\end{equation*}
where
\begin{equation*}
    c(\sigma,j,\val,k) = \begin{cases}
        A_{\val(i_k),\val(\sigma(i_k))} & \text{ if $k < j$}, \\
        A_{\val(i_j),\val(\sigma(i_j))} - B_{\val(i_j),\val(\sigma(i_j))} & \text{ if $k = j$}, \\
         B_{\val(i_k),\val(\sigma(i_k))} & \text{ if $k > j$}.
    \end{cases}
\end{equation*}

Let us now consider the unique decomposition $\sigma = \gamma_1 \circ ... \circ \gamma_c$ of $\sigma$ into $c$ \emph{disjoint} cycles. Then, we can write
\begin{align*}
    E_{\sigma,j} & = \left\vert \sum_{\val:I_r \rightarrow [n]} \prod_{l = 1}^c \left( \prod_{\substack{{k \in [r]} \\ {i_k \in \mathrm{supp}(\gamma_l)}}}^r c(\sigma,j,\val,k) \right) \right\vert \\
    & = \prod_{l = 1}^c \left\vert \sum_{\val:I_r \rightarrow [n]} \left( \prod_{\substack{{k \in [r]} \\ {i_k \in \mathrm{supp}(\gamma_l)}}}^r c(\sigma,j,\val,k)\right) \right\vert ,
\end{align*}
where the second equality holds because all of the cycles are disjoint, and we recognize an identity ``of the form $\sum_i \sum_j a_i b_j = \left( \sum_i a_i  \right) \left( \sum_j b_j \right)$''.

In the following, we write 
\begin{equation*}
    E_{\sigma,j,l} = \left\vert \sum_{\val:I_r \rightarrow [n]} \left( \prod_{\substack{{k \in [r]} \\ {i_k \in \mathrm{supp}(\gamma_l)}}}^r c(\sigma,j,\val,k)\right) \right\vert
\end{equation*}
each term of the product above, and reason by exhaustion on the length of the cycles.

If $\gamma_l$ has length $l = 1$, it is the equal to the identity on its support $\mathrm{supp}(\gamma_l) = \{i_k\}$, and we have:
\begin{itemize}
    \item $E_{\sigma,j,l} = \left\vert \sum_{i = 1}^n A_{i,i} \right\vert = \vert \mathrm{tr}(A)\vert$ if $k < j $;
    \item $E_{\sigma,j,l} = \left\vert \sum_{i = 1}^n B_{i,i} \right\vert = \vert \mathrm{tr}(B)\vert$ if $k > j$;
    \item $E_{\sigma,j,l} = \left\vert \sum_{i = 1}^n A_{i,i} - B_{i,i} \right\vert = \vert \mathrm{tr}(A) - \mathrm{tr}(B) \vert$ if $k = j$.
\end{itemize}

When the length of $\gamma_l$ is at least $2$, we are going to show in proposition~\ref{prop:submultiplicativity} below that
\begin{equation*}
    E_{\sigma,j,l} \leq \prod_{\substack{{k \in [r]} \\ {i_k \in \mathrm{supp}(\gamma_l)}}} \Vert C(j,k) \Vert_F,
\end{equation*}
where 
\begin{equation*}
    C(j,k) = \begin{cases}
        A & \text{ if $k < j$}, \\
        A - B & \text{ if $k = j$}, \\
        B & \text{ if $k > j$},
    \end{cases}
\end{equation*}
from which we obtain the desired inequality.

Since we only need to deal with one cycle $\gamma_l$ at a time, let us assume without loss of generality that $\gamma_l(i_j) = i_{j+1}$ (where, by definition, $i_{r_l + 1} = i_1$). Then, we have 
\begin{equation*}
        E_{\sigma,j,l} = \sum_{i_1,...,i_{r_l}=1}^n \prod_{k = 1}^{r_l} (C(j,k))_{i_k,i_{k+1}},
\end{equation*}
and we are left with showing the following.

\begin{prop}
    Let $M_1,...,M_s$ denote $s$ $n \times n$ matrices with real coefficients. Then, we have
    \begin{equation*}
        \left\vert \sum_{i_1,...,i_{s}=1}^n \prod_{k = 1}^{s} (M_k)_{i_k,i_{k+1}} \right\vert \leq \prod_{k=1}^{s} \Vert M_k \Vert_F.
    \end{equation*}
    \label{prop:submultiplicativity}
\end{prop}

\begin{proof}
    In the case that $s = 2$, the following is readily obtained by applying the Cauchy-Schwarz inequality:
    \begin{equation*}
        \left\vert \sum_{i_1,i_2 = 1}^n (M_1)_{i_1,i_2} (M_2)_{i_2,i_1} \right\vert \leq \Vert M_1 \Vert_F \Vert M_2 \Vert_F.
    \end{equation*}

    If $s \geq 3$, the Cauchy-Schwarz inequality yields
    \begin{equation*}
        \left\vert \sum_{i_1,...,i_{r_l}=1}^n \prod_{k=1}^n (M_k)_{i_k,i_{k+1}} \right\vert \leq \Vert M_1 \Vert_{F} \left\Vert \left[ \sum_{i_3,...,i_{r_l} = 1}^n  \prod_{k=2}^n (M_k)_{i_k,i_{k+1}}  \right]_{i_1,i_2=1}^n \right\Vert_F,
    \end{equation*}
    in which case we will show by induction that
    \begin{equation*}
        \Vert M_1 \Vert_{F} \left\Vert \left[ \sum_{i_3,...,i_{r_l} = 1}^n  \prod_{k=2}^n (M_k)_{i_k,i_{k+1}}  \right]_{i_1,i_2=1}^n \right\Vert_F^2 \leq \prod_{k = 1}^s \Vert M_k \Vert_F^2,
    \end{equation*}
    which yields the result.

    We now proceed with the induction. If $s = 3$,
    \begin{align*}
        \left\Vert \left[ \sum_{i_3 = 1}^n (M_2)_{i_2,i_3} (M_3)_{i_3,i_1} \right]_{i_1,i_2=1}^n \right\Vert_F^2 & = \sum_{i_1,i_2 = 1}^n \left\vert \sum_{i_3=1}^n (M_2)_{i_2,i_3} (M_3)_{i_3,i_1} \right\vert^2 \\
        & \leq \sum_{i_1,i_2=1}^n \Vert (M_2)_{i_2,:} \Vert_2^2 \Vert(M_3)_{:,i_1}\Vert_2^2 \\
        & = \sum_{i_1,i_2 = 1}^n \left( \sum_{i_3 = 1}^n \vert (M_2)_{i_2,i_3} \vert^2 \right) \left( \sum_{i_3' = 1}^n \vert (M_3)_{i_3,i_1} \vert^2 \right) \\
        & = \Vert M_2 \Vert_F^2 \Vert M_3 \Vert_F^2,
    \end{align*}
    where the second line is obtained by applying the Cauchy-Schwarz inequality.

    \noindent When $s > 3$, we reproduce a classical sub-multiplicativy argument and, for 
    \begin{equation*}
        C= \left\Vert \left[ \sum_{i_3,...,i_{r_l} = 1}^n  \prod_{k=2}^n (M_k)_{i_k,i_{k+1}}  \right]_{i_1,i_2=1}^n \right\Vert_F^2,
    \end{equation*}
    we finally obtain that:
    \begin{align*}
        C & = \sum_{i_1,i_2=1}^n \left\vert \sum_{i_3=1}^n  \left( (M_2)_{i_2,i_3} \sum_{i_4,...,i_{r_l}=1}^n \prod_{k=3}^{r_l} (M_k)_{i_k,i_{k+1}} \right) \right\vert^2 \\
        & \leq \sum_{i_1,i_2 = 1}^n \left\Vert (M_2)_{i_2,:} \right\Vert_2^2 \left\Vert \sum_{i_4 = 1}^n (M_3)_{:,i_4} \sum_{i_5,...,i_{r_l - 1}=1}^n \prod_{k=4}^{r_l} (M_k)_{i_k,i_{k+1}} \right\Vert_2^2 \\
        & = \sum_{i_1,i_2=1}^n \left( \sum_{i_3=1}^n \vert (M_3)_{i_2,i_3} \vert^2 \right) \left( \sum_{i_3,...,i_{r_l}=1}^n \left\vert \prod_{k=3}^{r_l} (M_k)_{i_k,i_{k+1}} \right\vert^2 \right) \\
        & = \Vert M_3 \Vert_F^2 \left\Vert \left[ \sum_{i_4,...,i_{r_l} = 1}^n \prod_{k=3}^{r_l} (M_k)_{i_k,i_{k+1}} \right]_{i_3,i_1} \right\Vert^2 \\
        & \leq \prod_{k=1}^s \Vert M_k \Vert_F^2,
    \end{align*}
    where the inequality is yet again obtained from the Cauchy-Schwarz inequality, and the last line follows from the induction hypothesis.
\end{proof}

Applying proposition~\ref{prop:submultiplicativity} to $E_{\sigma,j,l}$ completes the proof. We note that we only use the fact that the matrices are real-valued in the proof of proposition~\ref{prop:submultiplicativity}. 

\begin{rem}
    \label{rem:generalisation_mean_determinant}
    Similarly to remark~\ref{rem:generalisation_max_determinant}, we can apply the exact same arguments to prove lemma~\ref{lem:concentration_determinant_general}. Instead of considering sums over valuations though, those are replaced with multiple-integrals over the variables $(x_1,...,x_r)$ in $\mathrm{supp}(\varphi_r) = C_{\varphi_r}^r$. For proposition~\ref{prop:submultiplicativity} in particular, the same arguments show that, for real-valued kernels $K_1,...,K_s:\mathcal{X} \times \mathcal{X} \rightarrow \R$,
    \begin{equation*}
        \left\vert \underbrace{\int_{C_{\varphi_r}} ... \int_{C_{\varphi_r}}}_{k \ \text{times}} \prod_{k = 1}^s K_k(x_{k},x_{k+1}) d\nu_n(x_1)...d\nu_n(x_s) \right\vert \leq \prod_{k = 1}^s \Vert K_k \Vert_{L^2(\nu_n^{\otimes 2})}.
    \end{equation*}
    We stress that $C_{\varphi_r}$ is compact, and all relevant integrals can indeed be permuted in the course of the proof.
\end{rem}

\subsection{Proof of proposition~\ref{prop:compact_measure}}
\label{sect:proof_convergence_empirical_measure}

From the triangle inequality, we have
\begin{equation*}
    E^\mu_n \leq \left\vert \int_{X_n^r} f d\mu_n^{\otimes r} - \int_{\mathcal{X}^n} \int_{X_n^r} f d\mu_n^{\otimes r}  d\mu^{\otimes n}(X_n) \right\vert + \left\vert \int_{\mathcal{X}^n} \int_{X_n^r} f d\mu_n^{\otimes r}  d\mu^{\otimes n}(X_n) - \int_{\mathcal{X}^r} f d\mu^{\otimes r} \right\vert.
\end{equation*}
Let us denote by 
\begin{equation*}
    M_1 = \left\vert \int_{X_n^r} f d\mu_n^{\otimes r} - \int_{\mathcal{X}^n} \int_{X_n^r} f d\mu_n^{\otimes r}  d\mu^{\otimes n}(X_n) \right\vert,
\end{equation*}
\begin{equation*}
    M_2 = \left\vert \int_{\mathcal{X}^n} \int_{X_n^r} f d\mu_n^{\otimes r}  d\mu^{\otimes n}(X_n) - \int_{\mathcal{X}^r} f d\mu^{\otimes r} \right\vert
\end{equation*}
the two summands of the rhs.

We consider yet again a finite set of indices $I_r = \{i_1,...,i_r\}$ and valuations $\val: I_r \rightarrow \R$. Let us then introduce the following splitting:
\begin{align*}
    \int_{\mathcal{X}^n} \int_{X_n^r} f d\mu_n^{\otimes r}  d\mu^{\otimes n}(X_n) & = \int_{\mathcal{X}^n} \sum_{\val: I_r \rightarrow [n]} \frac{1}{n^r} f(x_{\val(i_1)},...,x_{\val(i_r)}) d\mu^{\otimes n}(X_n) \\
    & = \int_{\mathcal{X}^n} \sum_{\substack{{\val: I_r \rightarrow [n]} \\ {\val \text{ injective}}}} \frac{1}{n^r} f(x_{\val(i_1)},...,x_{\val(i_r)}) d\mu^{\otimes n}(X_n) \\
    & \hspace{0.4cm} +  \int_{\mathcal{X}^n} \sum_{\substack{{\val: I_r \rightarrow [n]} \\ {\val \text{ non injective}}}} \frac{1}{n^r} f(x_{\val(i_1)},...,x_{\val(i_r)}) d\mu^{\otimes n}(X_n).
\end{align*}
We start with $M_2$, and we denote by $T_{inj}$ (resp. $T_{\lnot inj}$) the integral over the sum of injective (resp. non-injective) valuations, and $c_{min}, c_{max} \in \R$ the lower and upper bounds on $f$.
Then, we have
\begin{equation*}
    \sum_{\substack{{\val: I_r \rightarrow [n]} \\ {\val \text{ non injective}}}} \frac{1}{n^r} c_{min} \leq T_{\lnot inj} \leq \sum_{\substack{{\val: I_r \rightarrow [n]} \\ {\val \text{ non injective}}}} \frac{1}{n^r} c_{max}.
\end{equation*}
Counting the number of non-injective value functions, we obtain
\begin{equation*}
    \#\{ \val: I_r \rightarrow [n] \text{ non injective}\} = \sum_{l = 1}^{r-1}\binom{r-1}{l} n^l \leq 2^{r-1} n^{r-1},
\end{equation*}
so that $\vert T_{\lnot inj} \vert \leq \frac{2^{r-1} n^{r-1}}{n^r} \max(-c_{min},c_{max}) = \frac{2^{r-1} n^{r-1}}{n^r} \beta$.

\noindent The crux of the argument comes into play in the computation of $T_{inj}$:
\begin{align*}
    T_{inj} & = \sum_{\substack{{\val: I_r \rightarrow [n]} \\ {\val \text{ injective}}}} \frac{1}{n^r} \int_{\mathcal{X}^n} f(x_{\val(i_1)},...,x_{\val(i_r)}) d\mu^{\otimes n}(X_n) \\
    & = \sum_{\substack{{\val: I_r \rightarrow [n]} \\ {\val \text{ injective}}}} \frac{1}{n^r} \int_{\mathcal{X}^r} f d\mu^{\otimes r} \\
    & = \frac{n^r - \sum_{l = 1}^{r-1} \binom{r-1}{l} n^l}{n^r} \int_{\mathcal{X}^r} f d\mu^{\otimes r},
\end{align*}
where the second equality is crucially obtained because the integration is over $r$ \emph{different} variables ($\val$ being injective). 

\noindent Plugging both estimates into $M_2$, we finally obtain
\begin{align*}
    M_2 & \leq \frac{2^{r-1} n^{r-1}}{n^r} \left( \left\vert\int_{\mathcal{X}^r} f d\mu^{\otimes r} \right\vert + \beta \right) \\
    & \leq \frac{M}{n}.
\end{align*}

We now turn to $M_1$, which we are going to bound using McDiarmid's inequality~\cite{concentration_boucheron}. In order to simplify the upcoming equations, we write
\begin{equation*}
    g(x_1,...,x_n) = \sum_{\val: I_r \rightarrow [n]} \frac{1}{n^r} f(x_{\val(i_1)},...,x_{\val(i_r)})
\end{equation*}
and, for an ordered set of elements $X_n = (x_1,...,x_j,...,x_n) \in \mathcal{X}^n $ and some $x' \in \mathcal{X}$, we define $Y_n = (y_1,...,y_n) := (x_1,...,x',...,x_n)$.

\noindent Then, we have
\begin{equation*}
    \vert g(X_n) - g(Y_n) \vert  = \left\vert \sum_{\substack{{\val:I_r \rightarrow [n]} \\ {j \in \val(I_r)}}} \frac{1}{n^r} f(x_{\val(i_1)},...,x_{\val(i_r)}) - f(y_{\val(i_1)},...,y_{\val(i_r)}) \right\vert
\end{equation*}
and, since $\#\{ \val : I_r \rightarrow [n] \ ; \ j \in \val(I_r)\} = \sum_{l=1}^{r}\binom{r}{l} n^{l-1}$, 
\begin{equation*}
    \vert g(X_n) - g(Y_n) \vert \leq \max(c_{min},c_{max}) \frac{\sum_{l=1}^{r} \binom{r}{l} n^{l-1}}{n^r}.
\end{equation*}
The conclusion is reached by applying McDiarmid's inequality.

\subsection{Proof of corollary~\ref{cor:erreur_mesure}}
\label{sect:preuve_corollaire_mesure}

We are going to show that, for $n \geq \max\left(\frac{2^{r+5} \beta \log\left( \frac{2}{\delta} \right)}{\eps^2}, \frac{C_r}{\eps} \right)$ and with probability at least $1 - \frac{\delta}{2}$, $$E^\mu_n \leq \frac{\eps}{2}.$$

Recall that, using the notations of proposition~\ref{prop:compact_measure} $E_n^\mu \leq \widetilde{\eps} + \frac{M}{n}$. Hence, it suffices to show that each term in the rhs is smaller than $\frac{\eps}{4}$. 

For the deterministic term, $\frac{M}{n} \leq \frac{\eps}{4}$ as soon as $n \geq \frac{4M}{\eps}$. Making $M$ explicit, this is satisfied as soon as
\begin{equation*}
    n \geq 4 \times 2^{r-1} \left( \mu^{\otimes r}(C) \beta + \beta^2 \right),
\end{equation*}
where $\beta$ is defined in equation~\eqref{eq:def_beta}.

Hence, $E_n^\mu \leq \frac{\eps}{2}$ whenever $\widetilde{\eps} \leq \frac{\eps}{4}$ which, according to proposition~\ref{prop:compact_measure} is the case with probability at least $1 - \frac{\delta}{2}$ as soon as $\frac{\delta}{2} \geq 2 \exp\left( \frac{- \left( \frac{\eps}{4} \right)^2 n}{2 b_n} \right)$, where $b_n$ is defined by
    \begin{align}
        \label{eq:bn_mesure}
        b_n & = n^2\left( \beta \frac{\sum_{l = 1}^{r} \binom{r}{l} n^{l-1}}{n^r} \right)^2 .
    \end{align}
    
Equivalently, this means that
\begin{equation*}
    n \geq \frac{32 b_n \log\left( \frac{4}{\delta} \right)}{\eps^2}
\end{equation*}
and, noting that $b_n \leq 2^r \beta^2$, it suffices that
\begin{equation*}
    n \geq \frac{2^{r+5} \beta \log\left( \frac{4}{\delta} \right)}{\eps^2}.
\end{equation*}

\end{appendix}

\bibliographystyle{imsart-number} 
\bibliography{ref}

\newpage

{ \centering \Large \textbf{Supplement to "Statistical Consistency of Discrete-to-Continuous Limits of Determinantal Point Processes}}

\section{Proof of theorem~\ref{th:moment_mapping}}
\label{sect:proof_moment_mapping}

We are going to prove more explicit version of theorem~\ref{th:moment_mapping}. For the sake of readability, let us introduce a bit of notation. For a fixed compactly-supported, bounded and measurable $\varphi_r: \mathcal{X} \rightarrow \R$, we write $\Lambda_\mathcal{P} := \Lambda^{(\varphi_r)}(\mathcal{S}_{\mathcal{P}_n})$ (resp. $\Lambda_\mathcal{Q} := \Lambda^{(\varphi_r)}(\mathcal{S}_{\mathcal{Q}_n})$) for the linear statistics associated to the point process $\mathcal{P}_n$ (resp. $\mathcal{Q}_n)$, and further drop the subscripts for the expectations whenever it is clear from context so that, for instance, $\mathbf{E}[\Lambda_\mathcal{P}] := \mathbf{E}_{\mathcal{S}_\mathcal{P} \sim \mathcal{P}_n} [\Lambda^{(\varphi_r)}(\mathcal{S}_\mathcal{P})]$ (resp. $\mathbf{E}[\Lambda_\mathcal{Q}] := \mathbf{E}_{\mathcal{S}_{\mathcal{Q}} \sim \mathcal{Q}_n} \left[\Lambda^{(\varphi_r)}(\mathcal{S}_{\mathcal{Q}})\right]$). While we stress that $\Lambda_\mathcal{P}$ \emph{depends on $\mathcal{S}_{\mathcal{P}_n}$} rather than $\mathcal{P}$, this term only appears below when taking expectations of the form $\mathbf{E}[\Lambda_\mathcal{P}]$ or $\mathbf{E}[\Lambda_\mathcal{P}^j]$ for some exponent $j$, and it is never used ambiguously. Using these notations, the expressions for the raw and central moments read:
    \begin{equation*}
            m_k^{(\varphi_r)}( \mathcal{P}_n) = \mathbf{E}\left[ \Lambda_\mathcal{P}^k \right], \ \ \overline{m}_k^{(\varphi_r)}( \mathcal{P}_n) = \mathbf{E}\left[ \left( \Lambda_\mathcal{P} - \mathbf{E}\left[ \Lambda_P \right] \right)^k \right],
    \end{equation*}
    \begin{equation*}
            m_k^{(\varphi_r)}( \mathcal{Q}_n) = \mathbf{E}\left[ \Lambda_\mathcal{Q}^k \right], \ \ \overline{m}_k^{(\varphi_r)}( \mathcal{Q}_n) = \mathbf{E}\left[ \left( \Lambda_\mathcal{Q} - \mathbf{E}\left[ \Lambda_Q \right] \right)^k \right].
    \end{equation*}

Recall that we work under the uniform-boundedness assumption on the moments that, with probability $1$, there exist some constants $c_j \in \R_{>0}$ such that
\begin{equation*}
    \left\vert\mathbf{E}\left[ \Lambda_\mathcal{Q}^j \right] \right\vert \leq c_j
\end{equation*}
for all admissible $\varphi_r$. This is for instance always satisfied in the setting of theorem~\ref{th:detailed}.

We are going to show the following.

\begin{thm}
    \label{th:precise_moment_mapping}
    Suppose that, for all $r$, compactly-supported, bounded and measurable functions $\widetilde{\varphi_r}: \mathcal{X}^r \rightarrow \R$, there exists some $\eps^{(\widetilde{\varphi_r})} > 0$ and $\widetilde{\delta} \in (0,1)$ such that, with probability at least $1 - \widetilde{\delta}$,

\begin{equation}
    \label{eq:hypothese}
    \left\vert \Phi^{(\widetilde{\varphi_r})}(\mathcal{P}_n) - \Phi^{(\widetilde{\varphi_r})}(\mathcal{Q}_n) \right\vert \leq \eps^{(\widetilde{\varphi_r)}}.
\end{equation}

    Then, for a given $\varphi_r$, with probability at least $1 - \widetilde{\delta}$,
    \begin{equation}
        \label{eq:precise_raw_moment}
            \left\vert m_k^{(\varphi_r)}( \mathcal{P}_n) - m_k^{(\varphi_r)}( \mathcal{Q}_n)  \right\vert \leq \sum_{l = r}^{kr} \sum_{\substack{{{m_1,...,m_{\left\lceil \frac{l}{r} \right\rceil} \geq 1}} \\ {m_1 + ... + m_{\left\lceil \frac{l}{r} \right\rceil} = k}}}\eps^{\left(\varphi_r^{\left(m_1,...,m_{\left\lceil \frac{l}{r} \right\rceil}\right)}\right)}_n
    \end{equation}
    and, with probability at least $1 - \widetilde{\delta}$ again,
    \begin{align}
    \label{eq:precise_central_moment}
            \left\vert \overline{m}_k^{(\varphi_r)}( \mathcal{P}_n) - \overline{m}_k^{(\varphi_r)}( \mathcal{Q}_n) \right\vert & \leq \sum_{q = 0}^k \binom{k}{q}  \sum_{l = r
            }^{kr} \sum_{\substack{{{m_1,...,m_{\left\lceil \frac{l}{r} \right\rceil} \geq 1}} \\ {m_1 + ... + m_{\left\lceil \frac{l}{r} \right\rceil} = k}}} c_{k-q} \left( 1 + c_1 (k-q) \left(c_1 + \eps_n^{(\varphi_r)} \right) \eps^{(\varphi_r)}_n, \right)   \eps^{\left(\varphi_r^{\left(m_1,...,m_{\left\lceil \frac{l}{r} \right\rceil}\right)}\right)}_n \nonumber \\
            & \hspace{0.4cm} + \sum_{q = 0}^k \binom{k}{q} (k-q) c_1 c_{k-q}\left(c_1 + \eps_n^{(\varphi_r)} \right) \eps^{(\varphi_r)}_n, 
    \end{align}
    where
    \begin{equation*}
            \varphi_r^{\left(m_1,...,m_{\left\lceil \frac{l}{r} \right\rceil}\right)}(x_1,...,x_l) = \left( \prod_{i = 1}^{\left\lfloor \frac{l}{r} \right\rfloor} \varphi_r(x_{(i-1)r + 1},...,x_{(i-1)r + r})^{m_i} \right) \varphi_r(x_{\left\lfloor \frac{r}{l} \right\rfloor + 1,...,x_l,x_1,...,x_{r - (l \mod r)}})^{m_{\left\lceil \frac{l}{r} \right\rceil}}.
    \end{equation*}
    
\end{thm}

\noindent The functions $\varphi_r^{\left(m_1,...,m_{\left\lceil \frac{l}{r} \right\rceil}\right)}$, along with $\varphi_r$, form the family $(\varphi^i)$ in theorem~\ref{th:moment_mapping}. The proof of theorem~\ref{th:moment_mapping} is thus completed by first noting that:
\begin{enumerate}
    \item for weakly coherent processes, the condition of equation~\eqref{eq:hypothese} is satisfied with $\eps^{(\widetilde{\varphi_r})} = \widetilde{\eps}$ when $n > n(\widetilde{\delta},\widetilde{\eps},\widetilde{\varphi_r})$;
    \item the right-hand sides of equation~\eqref{eq:precise_raw_moment} (resp.~\eqref{eq:precise_central_moment}) is a polynomial in the $\eps^{(\varphi^i)}$'s and, taking $s$ (resp. $\overline{s}$) the sum of its coefficients and $\widetilde{\eps} = \frac{\eps}{s}$ (resp. $\widetilde{\eps} = \frac{\eps}{\overline{s}}$), it is bounded by $\eps$ whenever $\eps$ (resp. $\overline{s}$) is small enough that $\frac{\eps}{s} < 1$ (resp. $\frac{\eps}{\overline{s}} < 1$). 
\end{enumerate}
Taking $\widetilde{\delta} = \frac{\delta}{L}$ and applying a union bound, we recover the sought condition of equation~\eqref{eq:sought_moment} \emph{when $\eps/s < 1$ (resp. $\eps/\overline{s} < 1$)}. This restriction can finally be discarded according to the remark below.

\begin{rem}
\label{rem:small_eps}
For any positive real-valued variable $X$ and $\eps < \eta$, 
\begin{equation*}
    \mathbb{P}(X > \eta) < \mathbb{P}(X > \eps).
\end{equation*}
\end{rem}

We now move on to the proof of theorem~\ref{th:precise_moment_mapping}, which hinges on the following combinatorial expansion. For any locally finite subset $\mathcal{S}$ of $\mathcal{X}$,
\begin{equation*}
    \left( \sum_{\substack{{x_1,...,x_r \in \mathcal{S}} \\ {x_i \neq x_j}}} \varphi_r(x_1,...,x_r) \right)^k = \sum_{l = r}^{kr} \sum_{\substack{{{m_1,...,m_{\left\lceil \frac{l}{r} \right\rceil} \geq 1}} \\ {m_1 + ... + m_{\left\lceil \frac{l}{r} \right\rceil} = k}}} \sum_{\substack{ {x_1,...,x_l \in \mathcal{S}} \\ {x_i \neq x_j} } }   \varphi_r^{\left(m_1,...,m_{\left\lceil \frac{l}{r} \right\rceil}\right)}(x_1,...,x_l),
\end{equation*}
which is obtained by inspecting the development of $\left( \sum_{\substack{{x_i,...,x_r} \\ {x_i \neq x_j}}} \varphi_r(x_1,...,x_r) \right)^k$. The precise form of the two outermost sums and functions $\varphi_r^{\left(m_1,...,m_{\left\lceil \frac{l}{r} \right\rceil}\right)}$ has little importance here, but it is crucial that the innermost sum is over $x_1,...,x_l$ that are pairwise different.

That way, taking expectations, we respectively find
\begin{equation*}
    \begin{cases}
        \mathbf{E}\left[ \Lambda_\mathcal{P}^k \right] = \sum_{l = r}^{kr} \sum_{\substack{{{m_1,...,m_{\left\lceil \frac{l}{r} \right\rceil} \geq 1}} \\ {m_1 + ... + m_{\left\lceil \frac{l}{r} \right\rceil} = k}}} \mathbf{E}_{\mathcal{S}_\mathcal{P} \sim \mathcal{P}} \left[ \sum_{\substack{ {x_1,...,x_l \in \mathcal{S}_{\mathcal{P}}} \\ {x_i \neq x_j} } }   \varphi_r^{\left(m_1,...,m_{\left\lceil \frac{l}{r} \right\rceil}\right)}(x_1,...,x_l) \right], \\
        \mathbf{E}\left[ \Lambda_\mathcal{Q}^k \right] = \sum_{l = r}^{kr} \sum_{\substack{{{m_1,...,m_{\left\lceil \frac{l}{r} \right\rceil} \geq 1}} \\ {m_1 + ... + m_{\left\lceil \frac{l}{r} \right\rceil} = k}}} \mathbf{E}_{\mathcal{S}_{\mathcal{Q}_n} \sim \mathcal{Q}_n} \left[ \sum_{\substack{ {x_1,...,x_l \in \mathcal{S}_{\mathcal{Q}_n}} \\ {x_i \neq x_j} } }   \varphi_r^{\left(m_1,...,m_{\left\lceil \frac{l}{r} \right\rceil}\right)}(x_1,...,x_l) \right],
    \end{cases}
\end{equation*}
so that, by the triangle inequality,
\begin{equation*}
    \hspace{-1cm} \left\vert m_k^{(\varphi_r)}( \mathcal{P}_n) - m_k^{(\varphi_r)}( \mathcal{Q}_n)  \right\vert \leq \sum_{l = r}^{kr} \sum_{\substack{{{m_1,...,m_{\left\lceil \frac{l}{r} \right\rceil} \geq 1}} \\ {m_1 + ... + m_{\left\lceil \frac{l}{r} \right\rceil} = k}}} \left\vert \mathbf{E}_{\mathcal{S}_\mathcal{P} \sim \mathcal{P}_n} \left[ \Lambda^{\varphi_r^{\left(m_1,...,m_{\left\lceil \frac{l}{r} \right\rceil}\right)}}(\mathcal{S}_\mathcal{P}) \right] - \mathbf{E}_{\mathcal{S}_{\mathcal{Q}_n} \sim \mathcal{Q}_n} \left[ \Lambda^{\varphi_r^{\left(m_1,...,m_{\left\lceil \frac{l}{r} \right\rceil}\right)}}(\mathcal{S}_{\mathcal{Q}_n}) \right] \right\vert.
\end{equation*}

\noindent Making explicit the concentration rates from equation~\eqref{eq:hypothese} immediately yields inequality~\eqref{eq:precise_raw_moment}.

To obtain the bound on the central moments, we make use of the identify from equation~\eqref{eq:product_identity}. We begin with an application of the binomial theorem to obtain:
\begin{equation*}
    (\Lambda - \mathbf{E}[\Lambda])^k = \sum_{q = 0}^k \binom{k}{q} \Lambda^q (- \mathbf{E}[\Lambda_\mathcal{Q}])^{k-q},    
\end{equation*}
\begin{equation*}
    (\Lambda_n - \mathbf{E}[\Lambda_n])^k = \sum_{q = 0}^k \binom{k}{q} \Lambda_n^q (- \mathbf{E}[\Lambda_\mathcal{P}])^{k-q}.    
\end{equation*}
Taking the difference between the expectation of the two expressions, and then applying the triangle inequality, yields
\begin{align*}
    \left\vert \overline{m}_k^{(\varphi_r)}( \mathcal{P}_n) - \overline{m}_k^{(\varphi_r)}( \mathcal{Q}_n) \right\vert & = \left\vert \sum_{q = 0}^k \binom{k}{q} \left( \mathbf{E}[\Lambda_\mathcal{Q}^q] (- \mathbf{E}[\Lambda_\mathcal{Q}])^{k-q} - \mathbf{E}[\Lambda_\mathcal{P}^q] (- \mathbf{E}[\Lambda_\mathcal{P}])^{k-q}  \right) \right\vert \\
    & \leq \sum_{q = 0}^k \binom{k}{q} \left\vert \mathbf{E}[\Lambda_\mathcal{Q}^q] (- \mathbf{E}[\Lambda_\mathcal{Q}])^{k-q} - \mathbf{E}[\Lambda_\mathcal{P}^q] (- \mathbf{E}[\Lambda_\mathcal{P}])^{k-q}  \right\vert,
\end{align*}
which we proceed to bound. From equation~\eqref{eq:product_identity}, it follows that
\begin{align*}
    \left\vert \mathbf{E}[\Lambda_\mathcal{Q}^q] (- \mathbf{E}[\Lambda_\mathcal{Q}])^{k-q} - \mathbf{E}[\Lambda_\mathcal{P}^q] (- \mathbf{E}[\Lambda_\mathcal{P}])^{k-q}  \right\vert & \leq \left\vert \mathbf{E}[\Lambda_\mathcal{Q}^q] - \mathbf{E}[\Lambda_\mathcal{P}^q] \right\vert \left\vert \mathbf{E}[\Lambda_\mathcal{P}]^{k-q} \right\vert \\
    & \hspace{0.4cm} + \left\vert \mathbf{E}[\Lambda_{\mathcal{Q}}^q] \right\vert \left\vert \mathbf{E}[\Lambda_\mathcal{Q}]^{k-q} - \mathbf{E}[\Lambda_\mathcal{P}]^{k-q} \right\vert \\
    & \leq \left\vert \mathbf{E}[\Lambda_\mathcal{Q}^q] - \mathbf{E}[\Lambda_\mathcal{Q}^q] \right\vert \left\vert \mathbf{E}[\Lambda_\mathcal{P}]^{k-q} \right\vert \\
    & \hspace{0.4cm} + \left\vert \mathbf{E}[\Lambda_\mathcal{Q}^q] \right\vert (k-q) \left\vert \mathbf{E}[\Lambda_\mathcal{Q}] - \mathbf{E}[\Lambda_\mathcal{P}] \right\vert \left\vert \mathbf{E}[\Lambda_\mathcal{Q}] \right\vert \left\vert \mathbf{E}[\Lambda_\mathcal{P}] \right\vert ,
\end{align*}
where we obtained the second inequality from equation~\eqref{eq:product_identity} because
\begin{align*}
    \left\vert \mathbf{E}[\Lambda_\mathcal{Q}]^{k-q} - \mathbf{E}[\Lambda_\mathcal{P}]^{k-q} \right\vert  & \leq \sum_{j = 1}^{k-q} \left\vert \mathbf{E}[\Lambda_\mathcal{Q}] - \mathbf{E}[\Lambda_\mathcal{P}] \right\vert \left\vert \mathbf{E}[\Lambda_\mathcal{Q}] \right\vert \left\vert \mathbf{E}[\Lambda_\mathcal{P}] \right\vert \\
    & = (k-q) \left\vert \mathbf{E}[\Lambda_\mathcal{Q}] - \mathbf{E}[\Lambda_\mathcal{P}] \right\vert \left\vert \mathbf{E}[\Lambda_\mathcal{Q}] \right\vert \left\vert \mathbf{E}[\Lambda_\mathcal{P}] \right\vert .
\end{align*}
Similarly, we obtain that
\begin{equation*}
    \left\vert \mathbf{E}[\Lambda_\mathcal{P}]^{k-q} \right\vert \leq \left\vert \mathbf{E}[\Lambda_\mathcal{Q}]^{k-q} \right\vert + (k-q) \left\vert \mathbf{E}[\Lambda_\mathcal{Q}] - \mathbf{E}[\Lambda_\mathcal{P}] \right\vert \left\vert \mathbf{E}[\Lambda_\mathcal{Q}] \right\vert \left\vert \mathbf{E}[\Lambda_\mathcal{P}] \right\vert.
\end{equation*}

Putting everything together, we are left with
\begin{align*}
    \left\vert \overline{m}_k^{(\varphi_r)}( \mathcal{P}_n) - \overline{m}_k^{(\varphi_r)}( \mathcal{Q}_n) \right\vert & \leq \sum_{q = 0}^k \binom{k}{q}  \left\vert \mathbf{E}[\Lambda_\mathcal{Q}^q] - \mathbf{E}[\Lambda_\mathcal{P}^q] \right\vert \left( \left\vert \mathbf{E}[\Lambda_\mathcal{Q}]^{k-q} \right\vert + (k-q) \left\vert \mathbf{E}[\Lambda_\mathcal{Q}] - \mathbf{E}[\Lambda_\mathcal{P}] \right\vert \left\vert \mathbf{E}[\Lambda_\mathcal{Q}] \right\vert \left\vert \mathbf{E}[\Lambda_\mathcal{P}] \right\vert \right) \\
    & \hspace{0.4cm} + \sum_{q = 0}^k \binom{k}{q} \left\vert \mathbf{E}[\Lambda_\mathcal{Q}^q] \right\vert (k-q) \left\vert \mathbf{E}[\Lambda_\mathcal{P}] - \mathbf{E}[\Lambda_\mathcal{Q}] \right\vert \left\vert \mathbf{E}[\Lambda_\mathcal{Q}] \right\vert \left\vert \mathbf{E}[\Lambda_\mathcal{P}] \right\vert
    ,
\end{align*}
and inequality~\eqref{eq:precise_central_moment} is finally obtained by plugging-in the concentration for raw moments (equation~\eqref{eq:precise_raw_moment}), applying the uniform bounds $c_j$, and making the concentration rates $\eps^{(\widetilde{\varphi_r})}$ explicit.

\section{Missing proofs from section~\ref{sect:ope_main}}

\subsection{Proof of proposition~\ref{prop:cv_orthogonal_polynomials}}
\label{sect:proof_ope}

We are going to show that, under the setting of proposition~\ref{prop:cv_orthogonal_polynomials}, there exist $A,B,C > 0$ such that, for any $\widetilde{\eps} > 0$ small enough that $\widetilde{\eps} \leq \frac{1}{2} \min\left( \Vert \mathcal{M}_1 \Vert_\mu, \Vert \mathcal{P}_1 \Vert_\mu, ..., \Vert \mathcal{P}_m \Vert_\mu \right)$, with probability at least $1 - C \exp\left( \frac{-2 \widetilde{\eps}^2 n}{B^2} \right)$,
\begin{equation*}
    \max_{x \in X_n} \vert \mathcal{P}_k(x) - P_k(x) \vert \leq A \widetilde{\eps} \ \ \forall k \in [m].
\end{equation*}
In particular, this is satisfied for any $\widetilde{\delta} \geq C\exp\left( \frac{-2 \widetilde{\eps}^2 n}{B^2} \right)$ which, substituting $\widetilde{\eps}$ for $A\widetilde{\eps}$, is equivalent to the condition in equation~\eqref{eq:condition_concentration_poly}. Finally remark that, following remark~\ref{rem:small_eps}, the condition $\widetilde{\eps} \leq \frac{1}{2} \min\left( \Vert \mathcal{M}_1 \Vert_\mu, \Vert \mathcal{P}_1 \Vert_\mu, ..., \Vert \mathcal{P}_m \Vert_\mu \right)$ can be discarded.

We reason inductively on $k \in [m]$, and will make repeated use of the fact that, if $\vert a_n - a \vert \leq c$ for some $0 < c \leq \frac{\vert a \vert}{2}$, then $\vert \frac{1}{a_n} - \frac{1}{a}\vert \leq \frac{2}{a^2} c$. We will also use the following lemma.

\begin{lem}
    Let $f,g: \mathcal{X} \rightarrow \R$ be two bounded functions on $\mathcal{X}$. Then, there exists some constant $c = 3 \left( \max_{x \in \mathcal{X}} \vert f(x) \vert + \max_{x \in \mathcal{X}} \vert g(x) \vert \right)$ such that, for all $e > 0$, with probability at least $1 - 2\exp\left( \frac{-2 e^2 n }{c^2} \right)$,
    \begin{equation*}
        \left\vert \langle f, g \rangle_{\mu_n} - \langle f, g \rangle_{\mu} \right\vert \leq e.
    \end{equation*}
    As an immediate consequence, when $\Vert f \Vert_\mu > 0$,
    \begin{equation}
        \left\vert \Vert f \Vert_{\mu_n} - \Vert f \Vert_\mu \right\vert \leq \frac{e}{\Vert f \Vert_{\mu_n} + \Vert f \Vert_\mu} \leq \frac{e}{\Vert f \Vert_{\mu}}.
        \label{eq:consequence_lemma_cv_inner_product}
    \end{equation}
    \label{lem:cv_inner_product}
\end{lem}

\begin{proof}
    This is a simple application of McDiarmid's inequality.

    First, we show that $\mathbf{E}\left[ \langle f, g \rangle_{\mu_n} \right] = \langle f, g \rangle_{\mu}$, where the expectation is over $X_n = \{x_1,...,x_n\}$:
    \begin{align*}
        \mathbf{E}\left[ \langle f, g \rangle_{\mu_n} \right] & = \mathbf{E}\left[ \sum_{i = 1}^n \frac{1}{n} f(x_i) g(x_i) \right] \\
        & = \sum_{i = 1}^n \frac{1}{n} \mathbf{E}[f(x_i)g(x_i)] \\
        & = \mathbf{E}_{x \sim \mu}[f(x)g(x)] \\
        & = \langle f, g \rangle_\mu.
    \end{align*}
    Letting $h(x_1,...,x_n) = \sum_{i = 1}^n \frac{1}{n} f(x_i)g(x_i)$, we have 
    \begin{align*}
        \vert g(x_1,...,x_j,...,x_n) - g(x_1,...,x_j',...,x_n) \vert & = \frac{1}{n} \vert f(x_j)g(x_j) - f(x_j')g(x_j') \vert \\
        & \leq \frac{1}{n} c
    \end{align*}
    so that we obtain the satisfaction of the bounded differences property, where the value of $c$ is obtained from equation~\eqref{eq:product_identity} and the triangle inequality.
    When $ \Vert f \Vert_\mu > 0$, equation~\eqref{eq:consequence_lemma_cv_inner_product} is obtained by taking $g = f$ and using the identity $(a^2 - b^2) = (a+b) (a-b)$.   
\end{proof}

\underline{Let us begin the induction. If $k=1$, we have}
\begin{equation*}
    \left\Vert \frac{M_1}{\Vert M_1 \Vert_{\mu_n} } - \frac{\mathcal{M}_1}{\Vert \mathcal{M}_1 \Vert_\mu} \right\Vert_\infty = \Vert M_1 \Vert_{\infty} \left\vert \frac{1}{\Vert M_1 \Vert_{\mu_n}} - \frac{1}{\Vert \mathcal{M}_1 \Vert_\mu} \right\vert,
\end{equation*}
where we adopt the notation $\Vert . \Vert_\infty = \Vert . \Vert_{L^\infty(X_n;\mu_n)}$ for brevity. Applying lemma~\ref{lem:cv_inner_product} to $f=g=\mathcal{M}_1$, it follows that, for $\widetilde{\eps}$ small enough that $\widetilde{\eps} \leq \frac{\Vert \mathcal{M}_1 \Vert_\mu}{2}$,
\begin{align*}
    \left\vert \frac{1}{\Vert M_1 \Vert_{\mu_n}} - \frac{1}{\Vert \mathcal{M}_1 \Vert_\mu} \right\vert & = \frac{1}{\Vert M_1 \Vert_{\mu_n} \Vert \mathcal{M}_1 \Vert_{\mu}} \left\vert \Vert M_1 \Vert_{\mu_n} - \Vert \mathcal{M}_1 \Vert_\mu \right\vert \\
    & \leq \frac{1}{\Vert M_1 \Vert_{\mu_n} \Vert \mathcal{M}_1 \Vert_{\mu}} \left(  \left\vert \Vert M_1 \Vert_{\mu_n} - \Vert \mathcal{M}_1 \Vert_{\mu_n} \right\vert +  \left\vert \Vert \mathcal{M}_1 \Vert_{\mu_n} - \Vert \mathcal{M}_1 \Vert_{\mu} \right\vert \right) \\
    & = \frac{1}{\Vert M_1 \Vert_{\mu_n} \Vert \mathcal{M}_1 \Vert_{\mu}}\left\vert \Vert \mathcal{M}_1 \Vert_{\mu_n} - \Vert \mathcal{M}_1 \Vert_{\mu} \right\vert \\
    & \leq \frac{2}{\Vert \mathcal{M}_1 \Vert_{\mu}^2} \left\vert \Vert \mathcal{M}_1 \Vert_{\mu_n} - \Vert \mathcal{M}_1 \Vert_{\mu} \right\vert \\
    & \leq \frac{2}{\Vert \mathcal{M}_1 \Vert_{\mu}^3} \widetilde{\eps}
\end{align*}
with probability at least $1 - 2\exp\left( \frac{-2 \widetilde{\eps}^2 n }{c_1^2} \right)$, where $c_1 = 6 \max_{x \in \mathcal{X}} \vert \mathcal{M}_1(x) \vert $.
Hence, with probability at least $1 - 2\exp\left( \frac{-2 \widetilde{\eps}^2 n }{c_1^2} \right)$,
\begin{equation*}
    \max_{x \in X_n} \vert \mathcal{P}_1(x) - P_1(x) \vert \leq \frac{2 \max_{x \in \mathcal{X}} \vert \mathcal{M}_1(x) \vert}{\Vert \mathcal{M}_1 \Vert_{\mu}^3} \widetilde{\eps}.
\end{equation*}

\underline{If $k > 1$, we proceed under the induction hypothesis}, and begin with the bound
\begin{align*}
    \left\Vert P_{k+1} - \mathcal{P}_{k+1} \right\Vert_\infty & = \left\Vert \frac{P_{k+1}'}{\Vert P_{k+1}' \Vert_{\mu_n}} - \frac{\mathcal{P}_{k+1}'}{\Vert \mathcal{P}_{k+1}' \Vert_{\mu}} \right\Vert_\infty \\
    & = \left\Vert (P_{k+1}' - \mathcal{P}_{k+1}' ) \frac{1}{\Vert \mathcal{P}_{k+1}' \Vert_{\mu}} + P_{k+1}' \left( \frac{1}{\Vert P_{k+1}' \Vert_{\mu_n} } - \frac{1}{\Vert \mathcal{P}_{k+1}' \Vert_{\mu}} \right) \right\Vert_\infty \\
    & \leq \frac{1}{\Vert \mathcal{P}_{k+1} \Vert_{\mu}} \Vert P_{k+1}' - \mathcal{P}_{k+1}'\Vert_{\infty} + \Vert P_{k+1}' \Vert_\infty \left\vert \frac{1}{\Vert P_{k+1}' \Vert_{\mu_n}} - \frac{1}{\Vert \mathcal{P}_{k+1}' \Vert_\mu} \right\vert.
\end{align*}
Let us now remark that
\begin{enumerate}
    \item $\Vert \mathcal{P}_{k+1}' \Vert_\mu$ is finite, and so is $\frac{1}{\Vert \mathcal{P}_{k+1}' \Vert_\mu}$;
    \item $\Vert P_{k+1}' \Vert_{\infty} \leq \Vert \mathcal{P}_{k+1}' \Vert_\infty + \Vert P_{k+1}' - \mathcal{P}_{k+1}' \Vert_\infty$, with $\Vert \mathcal{P}_{k+1}' \Vert_\infty$ finite;
    \item by lemma~\ref{lem:cv_inner_product} finally, if $\widetilde{\eps}$ and is small enough that $\widetilde{\eps} \leq \frac{\Vert \mathcal{P}_{k+1}\Vert_\mu}{2}$,
    \begin{align*}
        \left\vert \frac{1}{\Vert P_{k+1}' \Vert_{\mu_n}} - \frac{1}{\Vert \mathcal{P}_{k+1}' \Vert_{\mu}} \right\vert & = \frac{1}{\Vert P_{k+1}' \Vert_{\mu_n} \Vert \mathcal{P}_{k+1}' \Vert_{\mu}} \left\vert \Vert {P}_{k+1}' \Vert_{\mu_n} - \Vert \mathcal{P}_{k+1}' \Vert_{\mu} \right\vert \\
        & \leq \frac{1}{\Vert P_{k+1}' \Vert_{\mu_n} \Vert \mathcal{P}_{k+1}' \Vert_{\mu}} \left( \left\vert \Vert P_{k+1}' \Vert_{\mu_n} - \Vert \mathcal{P}_{k+1}' \Vert_{\mu_n} \right\vert + \left\vert \Vert \mathcal{P}_{k+1}' \Vert_{\mu_n} - \Vert \mathcal{P}_{k+1}' \Vert_{\mu} \right\vert \right) \\
        & \leq  \frac{1}{\Vert P_{k+1}' \Vert_{\mu_n} \Vert \mathcal{P}_{k+1}' \Vert_{\mu}} \left( \Vert P_{k+1}'- \mathcal{P}_{k+1}' \Vert_{\mu_n} + \left\vert \Vert \mathcal{P}_{k+1}' \Vert_{\mu_n} - \Vert \mathcal{P}_{k+1}' \Vert_{\mu} \right\vert \right)\\
        & \leq \frac{1}{\Vert P_{k+1}' \Vert_{\mu_n} \Vert \mathcal{P}_{k+1}' \Vert_{\mu}} \left( \Vert P_{k+1}'- \mathcal{P}_{k+1}' \Vert_{\infty} + \left\vert \Vert \mathcal{P}_{k+1}' \Vert_{\mu_n} - \Vert \mathcal{P}_{k+1}' \Vert_{\mu} \right\vert \right) \\
        & \leq \frac{2}{\Vert \mathcal{P}_{k+1}' \Vert_\mu^3} \left( \Vert P_{k+1}'- \mathcal{P}_{k+1}' \Vert_{\infty} + \widetilde{\eps} \right)
    \end{align*}
    with probability at least $1 - 2\exp\left( \frac{-2 \widetilde{\eps}^2 n }{c_{k+1}^2} \right)$, where $c_{k+1} = 6 \max_{x \in \mathcal{X}}\vert \mathcal{P}'_{k+1}(x)\vert$.
\end{enumerate}
Thus, we are left with showing that $\Vert P_{k+1}' - \mathcal{P}_{k+1}'\Vert_{\infty}$ concentrates towards $0$. By definition, we have
\begin{align*}
    \Vert P_{k+1}' - \mathcal{P}_{k+1}'\Vert_{\infty} & = \left\Vert \left( {M}_{1} - \frac{\langle {M}_{k+1},{P}_{k} \rangle_{\mu_n}}{\Vert {P}_{k} \Vert_{\mu_n}^2} {P}_k \right) - \left( \mathcal{M}_{1} - \frac{\langle \mathcal{M}_{k+1},\mathcal{P}_{k} \rangle_{\mu}}{\Vert \mathcal{P}_{k} \Vert_{\mu}^2} \mathcal{P}_k \right) \right\Vert_{\infty} \\
    & \leq \Vert M_{k+1} - \mathcal{M}_{k+1} \Vert_\infty +  \left\Vert  \frac{\langle {M}_{k+1},{P}_{k} \rangle_{\mu_n}}{\Vert {P}_{k} \Vert_{\mu_n}^2} {P}_k   - \frac{\langle \mathcal{M}_{k+1},\mathcal{P}_{k} \rangle_{\mu}}{\Vert \mathcal{P}_{k} \Vert_{\mu}^2} \mathcal{P}_k \right\Vert_{\infty},
\end{align*}
where the second line follows from the triangle inequality. As $\Vert M_{k+1} - \mathcal{M}_{k+1} \Vert_{\infty} = 0$, we proceed to bound the remaining term.

\noindent Using equation~\eqref{eq:product_identity} and the triangle inequality, we obtain
\begin{align*}
    \left\Vert  \frac{\langle {M}_{k+1},{P}_{k} \rangle_{\mu_n}}{\Vert {P}_{k} \Vert_{\mu_n}^2} {P}_k   - \frac{\langle \mathcal{M}_{k+1},\mathcal{P}_{k} \rangle_{\mu}}{\Vert \mathcal{P}_{k} \Vert_{\mu}^2} \mathcal{P}_k \right\Vert_{\infty} & \leq \left\vert \frac{1}{\Vert P_k \Vert^2_{\mu_n}} - \frac{1}{\Vert \mathcal{P}_{k} \Vert_{\mu}^2} \right\vert \vert \langle \mathcal{M}_{k+1}, \mathcal{P}_k \rangle_{\mu } \vert \left\Vert \mathcal{P}_{k} \right\Vert_{\infty} \\
    & \hspace{0.4cm} +  \frac{1}{\Vert P_k \Vert_{\mu_n}^2} \vert \langle M_{k+1}, P_k \rangle_{\mu_n} - \langle \mathcal{M}_{k+1}, \mathcal{P}_{k} \rangle_\mu \vert \Vert \mathcal{P}_k \Vert_{\infty}\\
    & \hspace{0.4cm} + \frac{1}{\Vert P_k \Vert_{\mu_n}^2} \vert \langle M_{k+1}, P_k \rangle_{\mu_n} \vert \Vert P_k - \mathcal{P}_{k} \Vert_\infty.
\end{align*}
Applying the induction hypothesis, it is clear that each of the the first and last summands of the rhs are smaller than $A \widetilde{\eps}$ with probability at $1 - 2 \exp\left( \frac{- 2 \widetilde{\eps} n}{B^2} \right)$, and it only remains to show that 
\begin{equation*}
    E_k = \vert \langle M_{k+1}, P_k \rangle_{\mu_n} - \langle \mathcal{M}_{k+1}, \mathcal{P}_{k} \rangle_\mu \vert
\end{equation*}
concentrates towards $0$.

Using a classical argument to bound differences of bilinear forms, we obtain
\begin{align*}
    E_k & \leq \vert \langle M_{k+1}, P_k \rangle_{\mu_n} - \langle \mathcal{M}_{k+1}, {P}_{k} \rangle_{\mu_n} \vert + \vert \langle \mathcal{M}_{k+1}, {P}_k \rangle_{\mu_n} - \langle \mathcal{M}_{k+1}, \mathcal{P}_{k} \rangle_\mu \vert \\
    & = \vert \langle \mathcal{M}_{k+1}, {P}_k \rangle_{\mu_n} - \langle \mathcal{M}_{k+1}, \mathcal{P}_{k} \rangle_\mu \vert \\
    & \leq \vert \langle \mathcal{M}_{k+1}, {P}_k - \mathcal{P}_{k} \rangle_{\mu_n} \vert + \vert \langle \mathcal{M}_{k+1}, \mathcal{P}_k \rangle_{\mu_n} - \langle \mathcal{M}_{k+1}, \mathcal{P}_{k} \rangle_\mu \vert \\
    & \leq \Vert \mathcal{M}_{k+1}\Vert_{\mu_n} \Vert P_k - \mathcal{P}_{k}\Vert_{\mu_n} + \vert \langle \mathcal{M}_{k+1}, \mathcal{P}_k \rangle_{\mu_n} - \langle \mathcal{M}_{k+1}, \mathcal{P}_{k} \rangle_\mu \vert,
\end{align*}
where the last line holds by the Cauchy-Schwarz inequality. Given that, using the induction hypothesis, $\Vert P_k - \mathcal{P}_{k}\Vert_{\mu_n} \leq \Vert P_k - \mathcal{P}_{k}\Vert_{\infty} \leq A \widetilde{\eps}$ with probability at least $1 - C \exp\left( \frac{-2 \widetilde{\eps} n}{B^2} \right)$, our final step is to show that this is also the case for $\vert \langle \mathcal{M}_{k+1}, \mathcal{P}_k \rangle_{\mu_n} - \langle \mathcal{M}_{k+1}, \mathcal{P}_{k} \rangle_\mu \vert$, which follows from an application of lemma~\ref{lem:cv_inner_product} for $f = \mathcal{M}_{k+1}$ and $g = \mathcal{P}_k$. Putting everything together and applying a union bound so that all probabilistic inequalities hold at once, we have that, up to a change of constants $A,B$ and $C$, with probability at least $1 - C \exp\left( \frac{-2 \widetilde{\eps} n}{B^2} \right)$,
\begin{equation*}
    \max_{x \in X_n} \vert \mathcal{P}_{k+1}(x) - P_{k+1}(x) \vert \leq A \widetilde{\eps}.
\end{equation*}

\subsection{Derivation of the rates for theorem~\ref{th:cv_polynomial_ensembles}}
\label{sect:rates_cv_ope}

Let us first notice that the kernels of multivariate orthogonal polynomial ensembles take the form
\begin{equation}
    K(x,y) = \sum_{i = 1}^m \phi_i(x) \phi_i(y)
\end{equation}
for some finite $p \in \N_{>0}$ and functions $\phi_i:\mathcal{X} \rightarrow \R$ orthogonal with respect to $\langle . , . \rangle_\nu$ (for some measure $\nu$). Determinantal point processes with kernels of this form are called \emph{projections DPPs}, and make up an important subclass of DPPs and a foundation for a large part of the modern theory of DPPs~\cite{hough_determinantal_2006}. For those projection DPPs, one has the following generic weak-coherency lemma.

\begin{lem}
    \label{lem:generic_awc}
    Let $\mu$ be a probability measure over some second-countable compact Hausdorff space $\mathcal{X}$, and $X_n \subseteq \mathcal{X}$ a set of $n$ points sampled iid according to $\mu$. Denote by $\mu_n$ the associated empirical measure, and consider sequences of functions $\phi_i:\mathcal{X} \rightarrow \C$ and $\phi_i^{[n]}:X_n \rightarrow \C$ such that the $\phi_i$'s are bounded. Suppose that, for any $\widetilde{\delta} \in (0,1)$ and $\widetilde{\eps} > 0$, there exists $N_\phi(\widetilde{\delta},\widetilde{\eps})$ such that
    \begin{equation}
        n \geq N_\phi(\widetilde{\delta},\widetilde{\eps}) ~ \Rightarrow ~ \mathbb{P}\left( \max_{i \in [m]} \max_{x \in X_n} \vert \phi_i(x) - \phi^{[n]}_i(x) \vert \geq \widetilde{\eps} \right) \leq \widetilde{\delta},
    \end{equation}
    and further define the projection kernels:
    \begin{equation}
        \mathcal{K}(x,y) = \sum_{i=1}^m \phi_i(x) \phi_i(y) \ \ \forall x,y \in \mathcal{X},
    \end{equation}
    \begin{equation}
        K_n(x,y) = \sum_{i = 1}^m \phi^{[n]}_i(x) \phi^{[n]}_i(y) \ \ \forall x,y \in X_n.
    \end{equation}
    Then, for $\widetilde{M} = \max_{i \in [p]}\left(\max_{x \in X_n} \left\vert \phi^{[n]}_i(x) \right\vert + \max_{x \in \mathcal{X}} \left\vert \phi_i(x) \right\vert \right)$, we have
    \begin{equation}
        n \geq N_\phi\left( \widetilde{\delta}, \frac{\widetilde{\eps}}{m \widetilde{M}} \right) ~ \Rightarrow ~ \mathbb{P}\left( \max_{x,y \in X_n} \vert \mathcal{K}(x,y) - K_n(x,y) \vert \geq \eps \right) \leq \delta.
    \end{equation}
\end{lem}

\noindent In particular, over compact sets $\mathcal{X}$, those $\mathcal{K}$ and $K_n$ satisfy the assumptions of theorem~\ref{th:detailed}, $ii)$, and weak coherency ensues (provided that those kernels actually define DPPs). Further, any instantiated rates $\widetilde{\delta} = \widetilde{\delta}_n$ and $\widetilde{\eps} = \widetilde{\eps}_n$ are carried over from the $\phi_i$'s to the kernels, up to a multiplicative constant. 

\begin{proof}[Proof of lemma~\ref{lem:generic_awc}]
    Applying the triangle inequality, we have
    \begin{align*}
        \max_{x,y \in X_n} \left\vert \mathcal{K}(x,y) - K_n(x,y) \right\vert & \leq m \max_{i \in [m]} \max_{x,y \in X_n} \vert \phi^{[n]}_i(x) \phi^{[n]}_i(y) - \phi_i(x) \phi_i(y) \vert \\
        & \leq m \max_{i \in [m]} \max_{x,y \in X_n}\left( \vert\phi^{[n]}_i(x) - \phi_i(x)\vert \vert \phi_i(y) \vert + \vert \phi^{[n]}_i(x) \vert \vert \phi^{[n]}_i(y) - \phi_i(y) \vert \right),
    \end{align*}
    from which the result is readily obtained. 
\end{proof}

Plugging-in the estimates from proposition~\ref{prop:cv_orthogonal_polynomials}, it follows that
\begin{equation*}
    n \geq \frac{A^2 B^2 \log\left( \frac{2}{\delta} \right)}{\widetilde{\eps}^2} ~ \Rightarrow ~ \mathbb{P}\left( \max_{x,y \in X_n} \vert \mathcal{K}(x,y) - K_n(x,y) \vert \geq {\eps} \right) \leq {\delta},
\end{equation*}
for any $\delta \in (0,1)$ \emph{and $\eps \geq m \widetilde{\eps} \widetilde{M}$},
where $\widetilde{M}= \max_{i \in [m]}\left( \max_{x \in X_n} \vert P_i(x) \vert + \max_{x \in X_n} \vert \mathcal{P}_i(x) \vert \right)$. Noting that $\max_{x \in X_n}\vert P_i(x) \vert \leq \widetilde{\eps}_n + \max_{x \in \mathcal{X}} \vert \mathcal{P}_i(x) \vert$, this is satisfied as soon as $\widetilde{\eps} \leq \frac{\eps}{m M}$ for $M = 2 \max_{i \in [m]} \max_{x \in \mathcal{X}} \vert \mathcal{P}_i(x) \vert $. That is, when
\begin{equation*}
    n \geq \frac{A^2 B^2 m^2 M^2 \log\left( \frac{2}{\delta} \right)}{\eps^2}  ~ \Rightarrow ~ \mathbb{P}\left( \max_{x,y \in X_n} \vert \mathcal{K}(x,y) - K_n(x,y) \vert \geq {\eps} \right) \leq {\delta}.
\end{equation*}

\section{Comparison with the DPP from~\cite{bardenet2021determinantal}
\label{sect:comparaison_bardenet}}

We start by recalling the construction of the DPP from~\cite{bardenet2021determinantal}, which is defined with respect to the measure $\mu_n$ with kernel $\widetilde{K_n}$ built as follows. 
\begin{enumerate}
    \item For a target Nevai-class probability distribution $\nu$ on $\mathcal{X}$ with density $q$, compute the $m$ first orthogonal polynomials $\widetilde{\mathcal{P}_i}$ with respect to $\nu$ and build the kernel $\widetilde{\mathcal{K}}(x,y) = \sum_{i = 1}^m \widetilde{\mathcal{P}_i}(x)\widetilde{\mathcal{P}_i}(y)$.
    \item Compute a kernel density estimator $\tilde{\gamma}$ of the density of the unknown measure $\mu$ according to which $X_n$ has been drawn.
    \item Define a re-weighted kernel
        \begin{equation*}
            \overline{\mathcal{K}}(x,y) = \sqrt{\frac{q(x)}{\tilde{\gamma}(x)}} \widetilde{\mathcal{K}}(x,y) \sqrt{\frac{q(y)}{\tilde{\gamma}(y)}}  
        \end{equation*}
    and compute its restriction $\left(\overline{\mathcal{K}}\right)_{\vert X_n \times X_n}$ to $X_n$.
    \item Finally, compute the first $m$ eigenvectors $\phi_1,...,\phi_m$ of $\left(\overline{\mathcal{K}}\right)_{\vert X_n \times X_n}$ and construct the projection kernel $\widetilde{K_n}(x,y) = \sum_{i = 1}^m \phi_i(x) \phi_i(y)$, which is a (projection) DPP with respect to $\mu_n$ satisfying the assumptions of the Macchi-Soshnikov theorem.
\end{enumerate} 
It is then argued in \cite{bardenet2021determinantal} that the variance of $1$-point linear statistics of $\mathrm{DPP}(\widetilde{K_n},\mu_n)$ with respect to a function $\varphi$ differs from that of $\mathrm{DPP}(\widetilde{\mathcal{K}},\nu)$ by a $\mathcal{O}_P\left( \frac{1}{\sqrt{n}} \right)$-term. This guarantee is analogous to that obtained in corollary~\ref{cor:variance_discrete_ope}.

There are two important elements of comparison between the two methods.
\begin{itemize}
    \item To the authors' admission, the spectral round-off step in the DPP from~\cite{bardenet2021determinantal} would ideally be bypassed, as it especially complicates the analysis. In comparison, the construction of $\mathrm{DPP}(K_n,\mu_n)$ in section~\ref{sect:def_ope} is much more straightforward: there is no need to choose a target density $q$, to perform density estimation (which are hyperparameter-dependent methods), nor spectral round-off. On the other hand, while both processes achieve a similar \emph{asymptotic} rate, the target density $q$ in the construction $\mathrm{DPP}(\widetilde{K_n},\mu_n)$ could theoretically be chosen to further minimize the non-asymptotic variance. The potential extent of the improvement effect is unclear and, to perform a theoretical comparison, one would first have to establish \emph{optimal} estimates for the approximation of the variance on both DPPs, which are available in neither case. In terms of practical performance, both methods are very similar~\cite{bardenet2024smallcoresetsnegativedependence}.
    \item Perhaps most importantly, exact DPP-sampling algorithms require the eigendecomposition of the kernel,~\emph{i.e.} the computation of $m$ eigenvectors, for which exact algorithms have cost in $\mathcal{O}(nm^2)$.\footnote{Let us emphasize that the matrices at play are typically \emph{not} sparse, and do not benefit from efficient algorithms and implementations of sparse linear algebraic routines.} This is the bottleneck for nearly all DPP-based algorithms. In the case of the process $\mathrm{DPP}(K_n,\mu_n)$ we propose, these eigenvectors are nothing but the discrete orthogonal polynomials themselves, for which more efficient algorithms exist. For low dimensions in particular, which is the case where the variance-reduction from theorem~\ref{th:bardenet_ghosh} is the most attractive, there are algorithms: in $\mathcal{O}(nm)$ for $d = 1$ (see~\cite{brubeck2021vandermonde} for a discussion); in $\mathcal{O}(nm^{3/2})$ for $d = 2$ (see~\cite{huhtanen2002generating} for an early example). 
\end{itemize}

\section{Missing proofs from section~\ref{sect:harmonic_ensemble}}

\subsection{Proof of proposition~\ref{prop:harmonic_noyaux_auxiliaires}}
\label{sect:app_harmonic_ensemble}

The equivalence in law of the DPPs is obtained from the following general assertion.

\begin{prop}
    \label{prop:density_exchange}
    Let $\Gamma$ be a second-countable, locally compact Hausdorff space, $\widetilde{\gamma}$ be a Radon measure over $\Gamma$, and $\gamma$ be a measure on $\Gamma$ absolutely continuous with respect to $\widetilde{\gamma}$, with density $p$ so that $\gamma = p \widetilde{\gamma}$. Consider two kernels $K: \Gamma \times \Gamma \rightarrow \C$ and $\widetilde{K}:\Gamma \times \Gamma \rightarrow \C$ such that
    \begin{equation}
        K(x,y) = \sqrt{\frac{1}{p(x)}} \widetilde{K}(x,y) \sqrt{\frac{1}{p(y)}}.
    \end{equation}
    Then, provided those exist, $(\widetilde{K},\widetilde{\gamma})$ and $(K,\gamma)$ define the same DPP.
\end{prop}

\begin{proof}
    Let $\varphi_r:\mathcal{X}^r \rightarrow \R$ be a bounded measurable function. Then, we have
    \begin{align*}
        \int_{\mathcal{X}^r} \varphi_r \rho_r[K] d\gamma^{\otimes r} & = \int_{\mathcal{X}^r} \varphi_r(x_1,...,x_r) \det\left( \left[ K(x_i,x_j) \right]_{i,j = 1}^r \right) d\gamma^{\otimes r}(\{x_1,...,x_r\}) \\
        & = \int_{\mathcal{X}^r} \varphi_r(x_1,...,x_r) \det\left( \left[ \widetilde{K}(x_i,x_j) \right]_{i,j = 1}^r \right) \prod_{i = 1}^r \frac{1}{p(x_i)}  d\gamma^{\otimes r}(\{x_1,...,x_r\}) \\
        & = \int_{\mathcal{X}^r} \varphi_r \rho_r[\widetilde{K}] d\widetilde{\gamma}^{\otimes r}.
    \end{align*}
\end{proof}

\noindent As a consequence, this also ensures that $\mathrm{DPP}(K_n,\mu_n)$ exists.

We now move on to the concentration of the kernels. Our first step is to show that the $v_i$'s concentrate towards the $\phi_i$'s. More precisely.

\begin{prop}
    \label{prop:harmonic_concentration_app}
    Under the hypotheses of proposition~\ref{prop:harmonic_noyaux_auxiliaires}, there exist constants $A_1,A_2,A_3,B,C>0$ such that, for any $\widetilde{\eps} > 0$, with probability at least $1 - C \left( \frac{1}{n^2} + \exp\left( \frac{-2\widetilde{\eps}^2 n}{B^2} \right) \right)$,
    \begin{equation*}
        \max_{x \in X_n} \vert v_i(x) -\phi_i(x) \vert \leq A_1 \widetilde{\eps} + A_2 \overline{\alpha} + A_3 \overline{\beta}  \ \ \forall i \in [m],
    \end{equation*}
    where $\alpha \propto \overline{\alpha} < 1$ and $\beta \propto \overline{\beta} < 1$.
\end{prop}

We remark that, similarly to the proof of proposition~\ref{prop:cv_orthogonal_polynomials}, our proof is carried out under the assumption that $\widetilde{\eps}$ is small enough, and then discard this condition following remark~\ref{rem:small_eps}. A precise upper-bound on $\widetilde{\eps}$ could be derived from the proof, but is quite tedious. Further, the hypothesis that $\alpha$ (resp. $\beta$) is proportional to some $\overline{\alpha}$ is mostly cosmetic, in order to hide some constants in proposition~\ref{prop:harmonic_noyaux_auxiliaires}.

We will make use of the following inequality: 

\begin{align}
    \label{eq:concentration_inner_product}
    \vert \langle a_n,b_n \rangle_{\omega_n} - \langle a,b \rangle_{\omega} \vert & \leq \vert \langle a,b \rangle_{\omega_n} - \langle a,b \rangle_{\omega} \vert \\
    & \hspace{0.4cm} + \Vert a \Vert_{\omega_n} \Vert b_n - b \Vert_{\omega_n} + \Vert b \Vert_{\omega_n} \Vert a_n - a \Vert_{\omega_n} \nonumber \\
    & \hspace{0.4cm} + \Vert a_n - a \Vert_{\omega_n} \Vert b_n - b \Vert_{\omega_n}, \nonumber
\end{align}
where $\Vert . \Vert_{\omega_n}$ and $\Vert . \Vert_{\omega}$ (resp. $\langle ., . \rangle_{\omega_n}$ and $\langle . , . \rangle_\omega$) denote the $L^2$ norms (resp. inner products) associated to the measures $\omega_n$ and $\omega$. This is because
\begin{equation*}
    \vert \langle a_n,b_n \rangle_{\omega_n} - \langle a,b \rangle_{\omega} \vert = \vert \langle a_n,b_n \rangle_{\omega_n} - \langle a,b \rangle_{\omega_n} \vert + \vert \langle a,b \rangle_{\omega_n} - \langle a,b \rangle_{\omega} \vert,
\end{equation*}
and equation~\eqref{eq:concentration_inner_product} is obtained from the identity
\begin{equation*}
    \vert \langle a_n,b_n \rangle_{\omega_n} - \langle a,b \rangle_{\omega_n} \vert = \vert \langle a_n - a, b \rangle_{\omega_n} + \langle a, b_n - b \rangle_{\omega_n} + \langle a_n - a, b_n - b \rangle_{\omega_n} \vert
\end{equation*}
for vectors $a_n,b_n,a,b \in \R^n$, followed by applications of the triangle and Cauchy-Schwarz inequalities.

By the triangle inequality, we need to show that

\begin{equation}
    \label{eq:goal_harmonic_concentration}
    \Vert v_i - \phi_i \Vert_\infty \leq \Vert v_i - u_i \Vert_\infty + \Vert u_i - \phi_i \Vert_{\infty}
\end{equation}
concentrates towards $0$ for $i = 1,...,m$, where we adopt the notation $\Vert a \Vert_\infty = \max_{x \in X_n} \vert a(x) \vert$ and we recall that the $v_i$'s are defined by the Gram-Schmidt orthogonalization process:

\begin{equation}
    \label{eq:gram-schmidt_harmonic}
    \begin{cases}
        v_1 = \frac{u_1}{\Vert u_i \Vert_{\omega_n}} \\
        v'_{k+1} = u_{k+1} - \frac{u_{k+1},v_k}{\Vert v_k \Vert_{\omega_n}} v_k \\
        v_{k+1} = \frac{v_{k+1}'}{\Vert v'_{k+1} \Vert_{\omega_n}}.
    \end{cases}
\end{equation}

For the rightmost term in the rhs of equation~\eqref{eq:goal_harmonic_concentration}, the bound
\begin{equation*}
    \Vert u_i - \phi_i \Vert_\infty \leq \overline{\alpha}
\end{equation*}
is obtained from the following result~\cite{dunson2021spectral}.

\begin{thm}[\cite{dunson2021spectral}]
    \label{th:dunson}
    Let $m \in \N_{> 0}$ and define $\Gamma_m = \min_{i \in [m]} \mathrm{dist}(\lambda_i,\sigma(-\mathcal{L}\setminus\{\lambda_i\})$ (where ($\sigma(-\mathcal{L}) \subseteq \R_{\geq 0}$ denotes the spectrum of $\sigma(- \mathcal{L})$ and $\mathrm{dist}$ the point-set euclidean distance), and suppose that
    \begin{equation}
        h_1(n) \leq C_1 \min\left( \left(\frac{\min(\Gamma_m,1)}{C_2 + \lambda_p^{d_\mathcal{X}/2 + 5}}\right)^2, \frac{1}{\left(C_3 + \lambda_p^{(5d_\mathcal{X} + 7)/4}\right)^2} \right),
    \end{equation}
    where $C_1,C_2,C_3$ are constants depending on the geometry of $\mathcal{X}$ and the density $p$. Then, assuming that $h_1(n) \geq \left( \frac{\log(n)}{n}\right)^{\frac{1}{4d_\mathcal{X} + 8}}$ and with probability at least $1 - \frac{1}{n^2}$, there exist $a_i \in \{-1,1\}$ such that
    \begin{equation}
        \max_{i \in [m]} \max_{x \in X_n} \left\vert a_i u_i(x) - \phi_i(x) \right\vert \leq C_4 h_1(n)^{1/2},
    \end{equation}
    where $C_4$ is a constant depending on the geometry of $\mathcal{X}$ and $p$.
\end{thm}

\noindent We let $\overline{\alpha} = C_4 h_1(n)^{1/2}$ and, for simplicity, further assume that all the $a_i = 1$ (which has no bearing on the definition of the kernel $\widetilde{\mathcal{K}}$). We are going to show that $\Vert v_i - u_i \Vert_{\infty} \leq A_1 \widetilde{\eps} + A_2 \overline{\alpha} + A_3 \overline{\beta}$, with appropriate probability and for $\overline{\beta} \propto \beta$ defined below.

\underline{We proceed by induction over $i$. {When $i = 1$}, we have}
\begin{align*}
    \Vert v_1 - u_1 \Vert_\infty & = \Vert u_1 \Vert_\infty \left\vert 1 - \frac{1}{\Vert u_1 \Vert_{\omega_n}} \right\vert \\
    & \leq (\max_{x \in \mathcal{X}} \vert \phi_1(x) \vert + \overline{\alpha} ) \left\vert 1 - \frac{1}{\Vert u_1 \Vert_{\omega_n}} \right\vert,
\end{align*}
where the second line holds with probability at least $1 - \frac{1}{n^2}$. At this point, we note that $\Vert \phi_1 \Vert_\omega = 1$ so that it remains to establish the concentration towards $0$ of $\left\vert \frac{1}{\Vert \phi_1 \Vert_\omega} - \frac{1}{\Vert u_1 \Vert_{\omega_n}} \right\vert$. 

Let us assume that $\vert \Vert \phi_1 \Vert_\omega - \Vert u_1 \Vert_{\omega_n} \vert \leq \eta \leq \frac{1}{2}$, then
\begin{equation*}
    \left\vert \frac{1}{\Vert \phi_1 \Vert_\omega} - \frac{1}{\Vert u_1 \Vert_{\omega_n}} \right\vert \leq 2 \eta,
\end{equation*}
and we proceed to find such an $\eta$. From equation~\eqref{eq:concentration_inner_product}, we have
\begin{align*}
    \vert \Vert \phi_1 \Vert_{\omega} - \Vert u_1 \Vert_{\omega_n} \vert & \leq \frac{\vert \Vert \phi_1 \Vert_{\omega}^2 - \Vert u_1 \Vert_{\omega_n}^2 \vert}{\Vert \phi_1 \Vert_{\omega} + \Vert u_1 \Vert_{\omega_n}} \nonumber \\
    & \leq \vert \Vert \phi_1 \Vert_{\omega}^2 - \Vert u_1 \Vert_{\omega_n}^2 \vert \nonumber \\
    & \leq \vert \Vert \phi_1 \Vert_{\omega}^2 - \Vert \phi_1 \Vert_{\omega_n}^2 \vert + 2 \Vert \phi_1 \Vert_{\omega_n} \Vert u_1 - \phi_1 \Vert_{\omega_n} + \Vert u_1 - \phi_1 \Vert_{\omega_n}^2 \nonumber \\
    & \leq \vert \Vert \phi_1 \Vert_{\omega}^2 - \Vert \phi_1 \Vert_{\omega_n}^2 \vert + 2 \left( \Vert \phi_1 \Vert_{\omega} + \frac{\vert \Vert \phi_1 \Vert_{\omega}^2 - \Vert \phi_1 \Vert_{\omega_n}^2 \vert}{\Vert \phi_1 \Vert_\omega} \right) \overline{\alpha} + \overline{\alpha}^2 \nonumber \\
    & \leq \vert \Vert \phi_1 \Vert_{\omega}^2 - \Vert \phi_1 \Vert_{\omega_n}^2 \vert + 2 \left( \Vert \phi_1 \Vert_{\omega} + \vert \Vert \phi_1 \Vert_{\omega}^2 - \Vert \phi_1 \Vert_{\omega_n}^2 \vert \right) \overline{\alpha} + \overline{\alpha}^2 ,
\end{align*}
and we still have to bound $\vert \Vert \phi_1 \Vert_{\omega}^2 - \Vert \phi_1 \Vert_{\omega_n}^2 \vert$. To do so, we first note that
\begin{equation*}
    \begin{cases}
        \Vert \phi_1 \Vert_\omega^2 = \langle \phi_1^2, \frac{1}{p}\rangle_\mu \\
        \Vert \phi_1 \Vert_{\omega_n}^2 = \langle \phi_1^2,\frac{1}{e[p]} \rangle_{\mu_n}.
    \end{cases}
\end{equation*} Thus, applying equation~\eqref{eq:concentration_inner_product} again, we have 
\begin{equation*}
    \vert \Vert \phi_1 \Vert_{\omega}^2 - \Vert \phi_1 \Vert_{\omega_n}^2 \vert \leq \Vert \phi_1^2 \Vert_{\mu_n} \left\Vert \frac{1}{{p}} - \frac{1}{{e[p]}} \right\Vert_{\mu_n} + \left\vert \langle \phi_1^2, \frac{1}{p}\rangle_\mu - \langle \phi_1^2, \frac{1}{p}\rangle_{\mu_n} \right\vert
\end{equation*}
and, according to lemma~\ref{lem:cv_inner_product} it holds with probability at least $1 - 2 \exp\left( \frac{-2\widetilde{\eps}^2n}{b^2} \right)$ that
\begin{equation*}
    \left\vert \langle \phi_1^2, \frac{1}{p}\rangle_\mu - \langle \phi_1^2, \frac{1}{p}\rangle_{\mu_n} \right\vert \leq \widetilde{\eps},
\end{equation*}
for some constant $b > 0$.
It only remains to bound the $\Vert \phi_1^2 \Vert_{\mu_n} \left\Vert \frac{1}{{p}} - \frac{1}{{e[p]}} \right\Vert_{\mu_n} \leq \left( \max_{x \in \mathcal{X}} \vert \phi_1(x)^2 \vert \right) \left\Vert \frac{1}{{p}} - \frac{1}{{e[p]}} \right\Vert_{\infty}$ term, for which we rely on the  following result of~\cite{wu2022strong}.

\begin{thm}[\cite{wu2022strong}]
    \label{th:th_ref_density}
    If $h_2(n) \leq C_1'$ for some $C_1'$, with probability at least $1 - \frac{1}{n^2}$, 
    \begin{equation}
        \max_{x \in \mathcal{X}} \vert e[p](x) - p(x) \vert \leq C_2' \left( \frac{\log(n)}{n h_2(n)^{d_\mathcal{X}}} \right)^{\kappa/2},
    \end{equation}
    where $C_1'$ and $C_2'$ are two constants depending on the geometry of $\mathcal{X}$ and the probability density $p$.
\end{thm}

\noindent We let $\overline{\beta} = C_2' \left( \frac{\log(n)}{n h_2(n)^{d_\mathcal{X}}} \right)^{\kappa/2}$, and remark that this result also implies that
\begin{equation*}
    \Vert e[p] - p \Vert_\infty \leq \overline{\beta}
\end{equation*}
which, when $\overline{\beta} < \frac{\min_{x \in \mathcal{X}} \vert p(x) \vert}{2} = \frac{p_{min}}{2}$, entails that
\begin{align*}
    \left\vert {\frac{1}{e[p](x)}} - {\frac{1}{p(x)}} \right\vert & = \frac{1}{\vert e[p](x) \vert \vert p(x) \vert} \vert e[p](x) - p(x) \vert \nonumber \\
    & \leq  \frac{2}{p_{min}^2} \overline{\beta}
\end{align*}
for all $x \in X_n$ and with probability at least $1 - \frac{1}{n^2}$.

Putting everything together, applying a union bound so that all probabilistic bounds hold at once, and using that $\overline{\alpha},\overline{\beta} < 1$, we finally find that, with probability at least $1 - \left( \frac{1}{n^2} +  2 \exp\left( \frac{-2\widetilde{\eps}^2n}{b^2} \right) \right)$, 
\begin{equation*}
    \Vert v_1 - \phi_1 \Vert_\infty \leq A_1 \widetilde{\eps} + A_2 \overline{\alpha} + A_3 \overline{\beta},
\end{equation*}
for some constants $A_1$, $A_2$ and $A_3$.

\underline{We now move on to {the case $i > 1$}} and continue our induction. We will omit some minute computations similar to those of the case $i = 1$. Here, applying equation~\eqref{eq:product_identity} and the triangle inequality to equation~\eqref{eq:gram-schmidt_harmonic} yields
\begin{align*}
    \Vert v_{i+1} - u_{i+1} \Vert_\infty & \leq \Vert v_{i+1}' - u_{i+1} \Vert_\infty \frac{1}{\Vert v_{i+1}' \Vert_{\omega_n}} + \left\vert 1 - \frac{1}{ \Vert v_{i+1}' \Vert_{\omega_n}} \right\vert \Vert u_{i+1} \Vert_\infty \\
    & \leq \Vert v_{i+1}' - u_{i+1} \Vert_\infty \frac{1}{\Vert v_{+1}' \Vert_{\omega_n}} + \left\vert \frac{1}{\Vert \phi_{i+1} \Vert_\omega} - \frac{1}{ \Vert v_{i+1}' \Vert_{\omega_n}} \right\vert \Vert u_{i+1} \Vert_\infty \\
    & \leq \Vert v_{i+1}' - u_{i+1} \Vert_\infty \frac{1}{\Vert v_{i+1}' \Vert_{\omega_n}} + \left\vert \frac{1}{\Vert \phi_{i+1} \Vert_\omega} - \frac{1}{ \Vert v_{i+1}' \Vert_{\omega_n}} \right\vert \left(\max_{x \in \mathcal{X}} \vert \phi_{i+1}(x) \vert + \overline{\alpha} \right),
\end{align*}
where we note that 
\begin{equation*}
    \left\vert \frac{1}{\Vert \phi_{i+1} \Vert_\omega} - \frac{1}{\Vert v_{i+1}' \Vert_{\omega_n}} \right\vert \leq \left\vert \frac{1}{\Vert \phi_{i+1} \Vert_\omega} - \frac{1}{\Vert u_{i+1} \Vert_{\omega_n}} \right\vert + \left\vert \frac{1}{\Vert u_{i+1} \Vert_{\omega_n}} - \frac{1}{\Vert v_{i+1}' \Vert_{\omega_n}} \right\vert,
\end{equation*}
with $\left\vert \frac{1}{\Vert u_{i+1} \Vert_{\omega_n}} - \frac{1}{\Vert v_{i+1}' \Vert_{\omega_n}} \right\vert \leq 2 \Vert u_{i+1} - v'_{i+1} \Vert_{\infty}$ whenever $\Vert u_{i+1} - v'_{i+1} \Vert_{\infty}$ is small enough, and where we can bound $\left\vert \frac{1}{\Vert \phi_{i+1} \Vert_\omega} - \frac{1}{\Vert u_{i+1} \Vert_{\omega_n}} \right\vert$ by repeating the argument used for the case $i = 1$. 

Let us now bound 
\begin{align*}
    \Vert v_{i+1}' - u_{i+1} \Vert_\infty & = \left\Vert \frac{\langle u_{i+1},v_i \rangle_{\omega_n}}{\Vert v_i \Vert_{\omega_n}^2} v_i \right\Vert_\infty \\
    & = \vert \langle u_{i+1},v_i \rangle_{\omega_n} \vert \frac{\Vert v_i \Vert_\infty}{ \Vert v_i \Vert_{\omega_n}^2}.
\end{align*}
Applying the triangle inequality and inequality~\eqref{eq:concentration_inner_product} to the $\vert \langle u_{i+1},v_i \rangle_{\omega_n} \vert$ term, we find that, with probability at least $1 - C \left( \frac{1}{n^2} + \exp\left( \frac{-2 \widetilde{\eps}^2 n}{B^2} \right) \right)$ for some $B > 0$,
\begin{align*}
     \vert \langle u_{i+1},v_i \rangle_{\omega_n} \vert & \leq \underbrace{\vert \langle \phi_{i+1},\phi \rangle_{\omega} \vert}_{0\text{ by orthogonality}} + \vert \langle u_{i+1}, v_i \rangle_{\omega_n} - \langle \phi_{i+1},\phi_i \rangle_\omega \vert \\
     & \leq \Vert u_{i+1} \Vert_{\omega_n} \Vert v_i - \phi_i \Vert_{\omega_n} + \Vert v_i \Vert_{\omega_n} \Vert u_{i+1} - \phi_{i+1} \Vert_{\omega_n}  \\
     & \hspace{0.4cm} + \Vert v_i - \phi_i \Vert_{\omega_n} \Vert u_{i+1} - \phi_{i+1} \Vert_{\omega_n} \\
     & \hspace{0.4cm} + \vert \langle \phi_{i+1},\phi_i \rangle_{\omega_n} - \langle \phi_{i+1}, \phi_i \rangle_{\omega} \vert \\
     & \leq \left( \Vert \phi_{i+1} \Vert_{\omega_n} + \Vert u_{i+1} - \phi_{i+1} \Vert_{\omega_n} \right) \widetilde{e_i} + (\Vert u_i \Vert_{\omega_n} + e_i) \Vert u_{i+1} - \phi_{i+1} \Vert_{\omega_n}  \\
     & \hspace{0.4cm} +\widetilde{e_i}  \Vert u_{i+1} - \phi_{i+1} \Vert_{\omega_n} \\
     & \hspace{0.4cm} + \vert \langle \phi_{i+1},\phi_i \rangle_{\omega_n} - \langle \phi_{i+1}, \phi_i \rangle_{\omega} \vert, \\
\end{align*}
where $\widetilde{e_i}, e_i \leq A_1 \widetilde{\eps} + A_2 \overline{\alpha} + A_3 \overline{\beta}$ are obtained through the induction hypothesis. Invoking theorem~\ref{th:dunson}, we are left with
\begin{equation*}
    \vert \langle u_{i+1},v_i \rangle_{\omega_n} \vert \leq \left( \Vert \phi_{i+1} \Vert_{\omega_n} + \overline{\alpha} \right) \widetilde{e_i} + (\Vert \phi_i \Vert_{\omega_n} + \overline{\alpha} + e_i) \overline{\alpha}  +\widetilde{e_i}  \overline{\alpha} + \vert \langle \phi_{i+1},\phi_i \rangle_{\omega_n} - \langle \phi_{i+1}, \phi_i \rangle_{\omega} \vert, \\
\end{equation*}
and applying inequality~\eqref{eq:concentration_inner_product} to $\vert \langle \phi_{i+1},\phi_i \rangle_{\omega_n} - \langle \phi_{i+1}, \phi_i \rangle_{\omega} \vert = \left\vert \langle \phi_{i+1}\phi_i, \frac{1}{e[p]} \rangle_{\mu_n} - \langle \phi_{i+1} \phi_i, \frac{1}{p} \rangle_{\mu} \right\vert$ yields
\begin{equation*}
     \vert \langle \phi_{i+1},\phi_i \rangle_{\omega_n} - \langle \phi_{i+1}, \phi_i \rangle_{\omega} \vert\leq \left\vert \langle \phi_{i+1} \phi_i , \frac{1}{p} \rangle_{\mu_n} - \langle \phi_{i+1} \phi_i, \frac{1}{p} \rangle_{\mu_n} \right\vert + \Vert \phi_{i+1} \phi_i \Vert_{\mu_n} \left\Vert \frac{1}{p} - \frac{1}{e[p]} \right\Vert_{\mu_n},
\end{equation*}
where we adopt the notation $(\phi_{i+1} \phi_i)(x) = \phi_{i+1}(x) \phi_i(x)$. Similarly to the case $i = 1$, we have $\Vert \phi_{i+1} \phi_i \Vert_{\mu_n} \left\Vert \frac{1}{p} - \frac{1}{e[p]} \right\Vert_{\mu_n} \leq \left( \max_{x \in \mathcal{X}} \vert \phi_{i+1}(x) \phi_i(x) \vert \right) \frac{2}{p_{min}^2} \overline{\beta}$ with probability at least $1 - \frac{1}{n^2}$, and $\left\vert \langle \phi_{i+1} \phi_i , \frac{1}{p} \rangle_{\mu_n} - \langle \phi_{i+1} \phi_i, \frac{1}{p} \rangle_{\mu_n} \right\vert < \widetilde{\eps}$ with probability at least $1 - 2 \exp\left( \frac{-2 \widetilde{\eps}^2 n}{b^2} \right)$ by lemma~\ref{lem:cv_inner_product}. To complete our induction, it remains only to bound $\frac{\Vert v_i \Vert_\infty}{ \Vert v_i \Vert_{\omega_n}^2}$. Keeping in mind that $\Vert \phi_i \Vert_{\omega} = 1$, we have
\begin{equation*}
    \frac{\Vert v_i \Vert_\infty}{ \Vert v_i \Vert_{\omega_n}^2} \leq \Vert \phi_i \Vert_\infty + \left\vert \frac{\Vert v_i \Vert_\infty}{ \Vert v_i \Vert_{\omega_n}^2} - \frac{\Vert \phi_i \Vert_\infty}{\Vert \phi_i \Vert_\omega^2} \right\vert,
\end{equation*}
where, applying equation~\eqref{eq:product_identity} and the triangle inequality,
\begin{align*}
    \left\vert \frac{\Vert v_i \Vert_\infty}{ \Vert v_i \Vert_{\omega_n}^2} - \frac{\Vert \phi_i \Vert_\infty}{\Vert \phi_i \Vert_\omega^2} \right\vert & \leq \vert \Vert v_i \Vert_\infty - \Vert \phi_i \Vert_\infty \vert \frac{1}{\Vert v_i \Vert_{\omega_n}^2} + \left\vert \frac{1}{\Vert v_i \Vert_{\omega_n}^2} - \frac{1}{\Vert \phi_i \Vert_{\omega}^2} \right\vert \Vert \phi_i \Vert_\infty \\
    & \leq \widetilde{e_i} \left( \frac{1}{\Vert \phi_i \Vert_{\omega_n}^2} + \left\vert \frac{1}{\Vert v_i \Vert_{\omega_n}^2} - \frac{1}{\Vert \phi_i \Vert_{\omega}^2} \right\vert \right) + \left\vert \frac{1}{\Vert v_i \Vert_{\omega_n}^2} - \frac{1}{\Vert \phi_i \Vert_{\omega}^2} \right\vert \max_{x \in \mathcal{X}} \vert \phi_i(x) \vert,
\end{align*}
Applying the induction hypothesis and a union bound thus completes the proof of proposition~\ref{prop:harmonic_concentration_app}.

Reproducing the proof of lemma~\ref{lem:generic_awc} and since $\max_{x \in X_n} \vert \phi_i(x) \vert \leq \max_{x \in \mathcal{X}} \vert \phi_i(x) \vert + \overline{A} (\widetilde{\eps} + \alpha + \beta)$, where $\overline{A} = \max(C_4 A_1, C_2' A_2,A_3)$, and taking $\widetilde{A} \propto \overline{A}$ so that $\alpha, \beta < 1$, we recover proposition~\ref{prop:harmonic_noyaux_auxiliaires} for $\widetilde{B} = B$ and $\widetilde{C} = C$.

\subsection{Proof of theorem~\ref{th:harmonic_coherency}}
\label{sect:proof_harmonic_theorem}

Using equation~\eqref{eq:product_identity} simply produces the bound
\begin{align*}
    \left\vert [K_n]_{a,b} - \mathcal{K}(x_a,x_b) \right\vert & \leq \left\vert \sqrt{\frac{1}{e[p](x_a)}} - \sqrt{\frac{1}{p(x_a)}} \right\vert \left\vert \mathcal{K}(x_a,y_a) \right\vert \left\vert \sqrt{\frac{1}{p(x_b)}}\right\vert \\
    & \hspace{0.4cm} + \left\vert \sqrt{\frac{1}{e[p](x_a)}} \right\vert \left\vert \left[\widetilde{K_n}\right]_{a,b} - \widetilde{\mathcal{K}}(x_a,x_b) \right\vert \left\vert \sqrt{\frac{1}{p(x_b)}}\right\vert \\
    & \hspace{0.4cm} + \left\vert \sqrt{\frac{1}{e[p](x_a)}} \right\vert \left\vert  \left[\widetilde{K_n}\right]_{a,b}  \right\vert \left\vert \sqrt{\frac{1}{e[p](x_b)}} - \sqrt{\frac{1}{p(x_b)}}\right\vert
\end{align*}
and, as previously noted, we have with probability at least $1 - \frac{1}{n^2}$ that
\begin{align*}
    \left\vert {\frac{1}{e[p](x)}} - {\frac{1}{p(x)}} \right\vert & = \frac{1}{\vert e[p](x) \vert \vert p(x) \vert} \vert e[p](x) - p(x) \vert \\
    & \leq  \frac{2}{p_{min}^2} C_2' \left( \frac{\log(n)}{n h_2(n)^{d_\mathcal{X}}} \right)^{\kappa/2}
\end{align*}
whenever $\beta$ is small enough that $\beta < \frac{p_{min}}{2}$. It follows that
\begin{align*}
    \left\vert \sqrt{\frac{1}{e[p](x)}} - \sqrt{\frac{1}{p(x)}} \right\vert & = \frac{1}{e[p](x) + p(x)} \left\vert {\frac{1}{e[p](x)}} - {\frac{1}{p(x)}} \right\vert \\
    & \leq \frac{2}{p_{min}^3} \beta,
\end{align*}
and we obtain from proposition~\ref{prop:harmonic_noyaux_auxiliaires} that, with probability at least $1 - \widetilde{C} \left( \frac{1}{n^2} + \exp\left( \frac{-2 \widetilde{\eps}^2 n}{\widetilde{B}^2} \right) \widetilde{\delta_n} \right) $,
\begin{align*}
    \left\vert [K_n]_{a,b} - \mathcal{K}(x_a,x_b) \right\vert & \leq \frac{2}{p_{min}^{3 + 1/2}} \beta M \\
    & \hspace{0.4cm} + \left( \frac{1}{p_{min}} + \frac{2}{p_{min}^{3 + 1/2}} \beta \right) m \widetilde{A}e \left( \widetilde{M} + \widetilde{A}e \right) \\
    & \hspace{0.4cm} + \left( \frac{1}{p_{min}^{1/2}} + \frac{2}{p_{min}^3} \beta \right) \frac{2}{p_{min}^3} \beta \left( M + m \widetilde{A} e \left( \widetilde{M} + \widetilde{A} e \right) \right) \\
    & \leq m A e,
\end{align*}
where $M = \max_{x,y \in \mathcal{X}} \vert \mathcal{K}(x,y) \vert$, and the last rough bound is obtained for some large constant $A > 0$ and assuming that $\widetilde{\eps} < \frac{1}{3}$, so that $e = \widetilde{\eps} + \alpha + \beta < \frac{1}{3} + \frac{1}{3} + \frac{1}{3} = 1$ thanks to the admissibility of $h_1(n)$ and $h_2(n)$. Following remark~\ref{rem:small_eps}, we neglect this last condition in the statement of theorem~\ref{th:harmonic_coherency}.

To enforce that $\max_{a,b \in [n]}\vert [K_n]_{a,b} - \mathcal{K}(x_a,x_b) \vert \leq \eps$ with probability at least $1 - \delta$, it thus suffices to find $\delta$ such that, with probability at least $1 - \delta$, 
\begin{equation*}
    \begin{cases}
        m A \widetilde{\eps} \leq \frac{\eps}{3}, \\
        m A \alpha \leq \frac{\eps}{3}, \\
        m A \beta \leq \frac{\eps}{3}.
    \end{cases}
\end{equation*}
The last two bounds are actually deterministic, and satisfied for $\frac{\eps}{3mA}~\geq~\min\left( h_1(n)^{1/2},  \left(\frac{\log(n)}{n h_2(n)^{d_\mathcal{X}}}\right)^{\kappa/2}\right)$. For the first one to be satisfied, it suffices that 
\begin{equation*}
    \delta \geq \widetilde{C} \left( \frac{1}{n^2} + \exp\left( \frac{- 2 \left( \frac{\eps}{3 m A} \right)^2 n}{\widetilde{B}^2} \right)\right), 
\end{equation*}
which is satisfied if $\delta$ is such that
\begin{equation*}
    \begin{cases}
        \delta \geq 2 \frac{\widetilde{C}}{n^2}, \\
        \delta \geq 2 \widetilde{C} \exp\left( \frac{- 2 \left( \frac{\eps}{3 m A} \right)^2 n}{\widetilde{B}^2} \right).
    \end{cases}
\end{equation*}
In particular, this is the case when
\begin{equation*}
    n \geq \max\left( \sqrt{\frac{2\widetilde{C}^2}{\delta}}, \frac{9 A^2 \widetilde{B}^2 m^2 \log\left( \frac{2\widetilde{C}}{\delta}\right)}{\eps^2}  \right).
\end{equation*}

\section{Proof of proposition~\ref{prop:usvt}}
\label{sect:proof_usvt}

We are first going to prove the following proposition.

\begin{prop}
    \label{prop:usvt_appendix}
    For any $q>0$, there exist three constants $\rho_q, c_q, C_q\in \R_{>0}$ such that, taking $\gamma_n = \rho_q (\alpha_n n)^{3/4}$,
     with probability at least $1-n^{-q}$,
    \begin{equation}
        \begin{cases}
            \frac{1}{n} \| K_n - \mathcal{K}_{\vert X_n \times X_n}\|_F \leq \frac{c_q}{(\alpha_n n)^{1/8}}, \\
            \left|\frac{\mathrm{tr}(K_n)}{n} - c\right| \leq \frac{C_q}{(\alpha_n n)^{1/4}}.
        \end{cases}
    \end{equation}
\end{prop}

In particular, both upper bounds are smaller than $\eps$ when
\begin{equation*}
    n \geq \frac{c_q^8}{\alpha_n \eps^8} \ \ \text{and} \ \ n \geq \frac{C_q^4}{\alpha_n \eps^4},
\end{equation*}
and with probability at least $\delta$ as soon as
\begin{equation*}
    \delta \geq n^{-q} ~ \Leftrightarrow ~ n \geq \frac{1}{\delta^q}.
\end{equation*}
Taking $b_q = c_q^8$ and $B_q = C_q^4$ yields proposition~\ref{prop:usvt}.

We begin by recalling the concentration result for symmetric matrices with Bernoulli entries from~\cite{lei2015consistency}, as applied to our setting.

\begin{thm}[\cite{lei2015consistency}]
    \label{th:lei_rinaldo}
    Suppose that we are under the setting of section~\ref{sect:usvt} so that, in particular, $\alpha_n \gtrsim \frac{\log(n)}{n}$. Then, for all $r>0$, there exists $c_r \in \R_{>0}$ such that, with probability at least $1-n^{-r}$
    \begin{equation}
        \label{eq:lei-rinaldo}
        \|A - \alpha_n \mathcal{W}_{\vert X_n \times X_n}\| \leq c_r \sqrt{\alpha_n n},
    \end{equation}
    where denotes the $\|\cdot\|$ is the operator norm with respect to the euclidean norm on $\R^n$.
\end{thm}

This is the starting point for our proof, and we will always assume that equation~\eqref{eq:lei-rinaldo} holds in the following.

Let us now consider the eigendecomposition of $W = \mathcal{W}_{\vert X_n \times X_n}$:
\begin{equation*}
    W = \sum_{i = 1}^n \tau_i w_i w_i^t.
\end{equation*}

Denoting by $S \subseteq [n]$ the subset of indices such that $\lambda_i \geq \gamma_n$, we further define
\begin{equation}
    G = \sum_{i \in S} \tau_i w_i w_i^t, \qquad \hat A = \sum_{i \in S} \lambda_i u_i u_i^t,
\end{equation}
where where we recall that the $\lambda_i$'s and $u_i$'s are the eigenvalues and eigenvectors of the adjacency matrix $A$. In particular, $\widetilde{A}_{\gamma_n} = \frac{\hat{A}}{\alpha_n}$.

We are now going to bound $\left\Vert \widetilde{A}_{\gamma_n} - \mathcal{W}_{\vert X_n \times X_n} \right\Vert_F$. From the triangle inequality, it holds that
\begin{equation*}
    \label{eq:todo_frob}
    \left\Vert \frac{\hat{A}}{\alpha_n} - W \right\Vert_F \leq \left\Vert \frac{\hat{A}}{\alpha_n} - G \right\Vert_F + \left\Vert G - W \right\Vert_F.
\end{equation*}

To bound the first term on the rhs, observe that $\hat{A}$ and $G$ are both of rank $\vert S \vert$, so that
\begin{equation*}
    \left\Vert \frac{\hat{A}}{\alpha_n} - G \right\Vert_F \leq \sqrt{2 \vert S \vert} \left\Vert \frac{\hat{A}}{\alpha_n} - G \right\Vert,
\end{equation*}
where the norm on the rhs is the operator norm. From the triangle inequality, this operator-norm difference decomposes as
\begin{equation*}
    \left\Vert \frac{\hat{A}}{\alpha_n} - G \right\Vert \leq \frac{1}{\alpha_n} \left\Vert \hat{A} - A \right\Vert + \left\Vert \frac{\hat{A}}{\alpha_n} - W\right\Vert + \Vert W - G \Vert
\end{equation*}
where, by definition of $S$ and $\hat{A}$, $\left\Vert \hat{A} - A \right\Vert \leq \gamma_n$. Further, according to theorem~\ref{th:lei_rinaldo} and with probability at least $1 - n^{-r}$, $\Vert A - \alpha_n W \Vert \leq c_r \sqrt{\alpha_n n}$. By Kato's inequality~\cite{kato2013perturbation,rosasco2010learning}, this implies that
\begin{equation}
    \max_{i \in [n]} \vert \lambda_i - \alpha_n \tau_i \vert \leq \Vert A - \alpha_n W \Vert \leq c_r \sqrt{\alpha_n n},
\end{equation}
so that, for all $i \in [n] \setminus S$,
\begin{equation}
    \label{eq:kato_consequence}
    0 \leq \alpha_n \tau_i \leq \lambda_i + c_r \sqrt{\alpha_n n} \leq \gamma_n + c_r \sqrt{\alpha_n n}.
\end{equation}
It follows that $\Vert W - G \Vert = \left\Vert \sum_{i \in [n] \setminus S} \tau_i w_i w_i^t \right\Vert \leq \frac{\gamma_n}{\alpha_n} + c_r \sqrt{\frac{n}{\alpha_n}}$, and we find that
\begin{equation*}
    \left\Vert \frac{\hat{A}}{\alpha_n} - G \right\Vert \leq 2 \left( \frac{\gamma_n}{\alpha_n} + c_r \sqrt{\frac{n}{\alpha_n}}\right).
\end{equation*}
A bound on the first term in the rhs of equation~\eqref{eq:todo_frob} will follow by bounding the cardinality of $S$. To do so, let $i \in S$ and notice that
\begin{equation*}
    \gamma_n - c_r\sqrt{\alpha_n n} \leq \lambda_i - \Vert A - \alpha_n W \Vert \leq \alpha_n \tau_i,
\end{equation*}
so that summing over all such $i$'s yields
\begin{equation*}
    \vert S \vert (\gamma_n - c_r \sqrt{\alpha_n n}) \leq \alpha_n \sum_{i \in S} \tau_i \leq \alpha_n \mathrm{tr}(W),
\end{equation*}
where we note that $\mathrm{tr}(W) \leq n c \leq n$ by our assumptions on $\mathcal{W}$. We thus obtain that $\vert S \vert \leq \frac{\alpha_n \mathrm{tr}(W)}{\gamma_n - c_r \sqrt{\alpha_n n}}$, resulting in
\begin{equation*}
    \left\Vert \frac{\hat{A}}{\alpha_n} - G \right\Vert_F \leq 2 \sqrt{2 \frac{\alpha_n \mathrm{tr}(W)}{\gamma_n - c_r \sqrt{\alpha_n n}}} \left( \frac{\gamma_n}{\alpha_n} + c_r \sqrt{\frac{n}{\alpha_n}} \right).
\end{equation*}

We now move on to bound the second term in the rhs of equation~\eqref{eq:todo_frob}. From equation~\eqref{eq:kato_consequence}, we simply have
\begin{align*}
    \Vert G - W \Vert_F^2 & = \sum_{i \in [n] \setminus S} \tau_i^2 \\
    & \leq \left( \frac{\gamma_n}{\alpha_n} + c_r \sqrt{\frac{n}{\alpha_n}} \right) \sum_{i \in [n] \setminus S} \tau_i \\
    & \leq \left( \frac{\gamma_n}{\alpha_n} + c_r \sqrt{\frac{n}{\alpha_n}} \right) \mathrm{tr}(W).
\end{align*}

Putting everything together, we find that
\begin{equation*}
    \left\Vert \frac{\hat{A}}{\alpha_n} - W \right\Vert_F \leq 2 \sqrt{2 \frac{\alpha_n \mathrm{tr}(W)}{\gamma_n - c_r \sqrt{\alpha_n n}}} \left( \frac{\gamma_n}{\alpha_n} + c_r \sqrt{\frac{n}{\alpha_n}} \right) + \sqrt{\frac{\gamma_n}{\alpha_n} + c_r \sqrt{\frac{n}{\alpha_n}} \mathrm{tr}(W)}
\end{equation*}
and, taking $\gamma_n \sim (\alpha_n n)^{3/4}$, we indeed obtain
\begin{equation*}
    \left\Vert \frac{\widetilde{A}_{\gamma_n} - \mathcal{W}_{\vert X_n \times X_n}}{n} \right\Vert_F \lesssim \frac{1}{(\alpha_n n)^{1/8}}.
\end{equation*}
To obtain the Frobenius-norm concentration of the kernels, we remark that, since $\mathrm{tr}\left( \widetilde{A}_{\gamma_n} \right) \geq 0$, $\frac{1}{n} \left\Vert \max\left( c - \mathrm{tr}\left( \frac{\widetilde{A}_{\gamma_n}}{n}\right), 0 \right) I \right\Vert_F \leq \frac{c}{\sqrt{n}}$ so that, applying the triangle inequality, the concentration rate is not impacted, and we find that
\begin{equation*}
    \left\Vert \frac{\bar A_{\gamma_n} - \mathcal{K}_{\vert X_n \times X_n}}{n} \right\Vert_F \lesssim \frac{1}{(\alpha_n n)^{1/8}}.
\end{equation*}

We next prove the concentration of $\mathrm{tr}\left( \frac{\bar A_{\gamma_n}}{n}\right)$ towards $c$. First, we note that
\begin{align*}
    \mathrm{tr}\left( \widetilde{A}_{\gamma_n}\right) & = \sum_{i \in S} \frac{\lambda_i}{\alpha_n} \\
    & \leq \sum_{i \in S} \tau_i + c_r \vert S\vert \sqrt{\frac{n}{\alpha_n}} \\
    & \leq \mathrm{tr}(W) + c_r \sqrt{\frac{n}{\alpha_n}} \frac{\alpha_n n}{\gamma_n - c_r \sqrt{\alpha_n n}},
\end{align*}
where the first and second inequalities follow from equation~\eqref{eq:kato_consequence}. Hence, for $\gamma_n \sim (\alpha_n n)^{3/4}$, 
\begin{equation*}
    0 \leq \frac{\mathrm{tr}\left( \widetilde{A}_{\gamma_n} \right)}{n} \leq c + \frac{C_q}{(\alpha_n n)^{1/4}}
\end{equation*}
for some constant $C_q \in \R_{>0}$. 
It follows that, by adding $C(\widetilde{A}_{\gamma_n}) I$, we indeed obtain
\begin{equation*}
    \left\vert \frac{\mathrm{tr}\left( \bar{A}_{\gamma_n} \right)}{n} - c \right\vert \leq \frac{C_q}{(\alpha_n n)^{1/4}}.
\end{equation*}

Then, we have 
\begin{equation*}
    \lambda_{\max}(\bar A_{\gamma_n}) \leq \mathrm{tr}(\bar A_{\gamma}) \leq n\left(1 +  \frac{C_q}{(\alpha_n n)^{1/4}}\right)
\end{equation*}
using $c\leq 1$, from which 
\begin{align*}
    \left(1 +  \frac{C_q}{(\alpha_n n)^{1/4}}\right)^{-1} &\leq C'(\bar A_{\gamma_n}) \leq \left(1 +  \frac{1}{(\alpha_n n)^{1/4}}\right)^{-1} \\
    1-\left(1 +  \frac{C_q}{(\alpha_n n)^{1/4}}\right)^{-1} &\geq 1- C'(\bar A_{\gamma_n}) \geq 1-\left(1 +  \frac{1}{(\alpha_n n)^{1/4}}\right)^{-1} \\
    \frac{C_q}{(\alpha_n n)^{1/4}} &\gtrsim 1- C'(\bar A_{\gamma_n}) \gtrsim \frac{1}{(\alpha_n n)^{1/4}}
\end{align*}
such that 
\begin{equation*}
    \left\vert 1- C'(\bar A_{\gamma_n})\right\vert \lesssim \frac{1}{(\alpha_n n)^{1/4}}.
\end{equation*}

We can then check the bounds on $K_n$:
\begin{align*}
    \left\Vert \frac{K_n - \mathcal{K}_{\vert X_n \times X_n}}{n} \right\Vert_F &\leq \left\Vert\frac{\bar A_{\gamma_n} - \mathcal{K}_{\vert X_n \times X_n}}{n}\right\Vert_F + \left\Vert\frac{K_n - \bar A_{\gamma_n}}{n}\right\Vert_F \\
    &\lesssim \frac{1}{(\alpha_n n)^{1/8}} + \left\vert \frac{1}{n} - \frac{C'(\bar A_{\gamma_n})}{n}\right\vert \left\Vert \bar A_{\gamma_n} \right\Vert_F \\
    &\lesssim \frac{1}{(\alpha_n n)^{1/8}} + \frac{1}{(\alpha_n n)^{1/4}} \lesssim \frac{1}{(\alpha_n n)^{1/8}}
\end{align*}
using $\left\Vert \bar A_{\gamma_n} \right\Vert_F \lesssim \left\Vert W\right\Vert_F + o(n)  \lesssim n$.

Similarly, for the trace,
\begin{align*}
    \left\vert \mathrm{tr}\left(\frac{K_n - \mathcal{K}_{\vert X_n \times X_n}}{n} \right)\right\vert &\leq \left\vert\mathrm{tr}\left(\frac{\bar A_{\gamma_n} - \mathcal{K}_{\vert X_n \times X_n}}{n}\right)\right\vert + \left\vert\mathrm{tr}\left(\frac{K_n - \bar A_{\gamma_n}}{n}\right)\right\vert \\
    &\lesssim \frac{1}{(\alpha_n n)^{1/4}} + \left\vert \frac{1}{n} - \frac{C'(\bar A_{\gamma_n})}{n}\right\vert \left\vert\mathrm{tr}\left( \bar A_{\gamma_n}\right) \right\vert \lesssim \frac{1}{(\alpha_n n)^{1/4}}
\end{align*}
since $\left\vert\mathrm{tr}\left( \bar A_{\gamma_n}\right) \right\vert \lesssim cn + o(n) \lesssim n$, 
which concludes the proof.

\end{document}